\documentclass[12pt]{article}
\usepackage{amssymb}
\usepackage{color, graphics, graphicx}
\usepackage{amsmath}
\usepackage{amsthm}
\usepackage{mathrsfs}
\usepackage{fontenc}
\usepackage{undertilde}
\usepackage{lscape}
\usepackage{natbib}
\usepackage{url}
\usepackage{bbm}
\usepackage{dsfont}
\usepackage{lscape}
\usepackage{setspace}
\usepackage{epstopdf}
\usepackage[shortlabels]{enumitem}
\usepackage{bigints}
\usepackage{rotating}
\usepackage{booktabs}
\usepackage{multirow}
\usepackage[ruled,vlined]{algorithm2e}
\usepackage{algorithmic}

\newtheorem{theorem}{Theorem}

\newtheorem{lemma}[theorem]{Lemma}

\newtheorem{proposition}[theorem]{Proposition}

\RequirePackage[OT1]{fontenc}
\numberwithin{equation}{section}
\RequirePackage[colorlinks]{hyperref}
\hypersetup{
    colorlinks,%
    citecolor=blue,%
    filecolor=red,%
    linkcolor=blue,%
    urlcolor=blue
} 

\newcommand{\cd}{\overset{\mathcal{D}}{\rightarrow}}

\newcommand{\beq}{\begin{equation}}
\newcommand{\eeq}{\end{equation}}
\newcommand{\bea}{\begin{eqnarray}}
\newcommand{\eea}{\end{eqnarray}}
\newcommand{\ba}{\begin{array}}
\newcommand{\ea}{\end{array}}
\newcommand{\bi}{\begin{itemize}}
\newcommand{\ei}{\end{itemize}}
\newcommand{\ben}{\begin{enumerate}}
\newcommand{\een}{\end{enumerate}}
\newcommand{\nn}{\nonumber}

\renewcommand{\r}{\right}
\renewcommand{\l}{\left}
\long\def\symbolfootnote[#1]#2{\begingroup\def\thefootnote{\fnsymbol{footnote}}\footnote[#1]{#2}\endgroup}

\newcommand{\nt}{_{nt}}
\newcommand{\Snl}{S_n(\lambda)}
\newcommand{\Rnr}{R_n(\rho)}
\newcommand{\tr} {\text{tr}}

\newcommand{\plim}{\text{plim}}
\newcommand{\tXn}{\tilde{X}_{nt}}

\newcommand{\Hn}{H_n(\rho)}

\newcommand{\blind}{1}

\begin{document}

\def\spacingset#1{\renewcommand{\baselinestretch}%
{#1}\small\normalsize} \spacingset{1}


\if1\blind
{
  \title{\bf Saddlepoint approximations for \\ spatial panel data models}
  \author{Chaonan Jiang\thanks{
    Research Center for Statistics and Geneva School of Economics and Management, University of Geneva, Blv. Pont 
d'Arve 40, 1211 Geneva, Switzerland, \url{chaonan.jiang@unige.ch}}, \hspace{0.2cm} 
    Davide La Vecchia \thanks{Research Center for Statistics and Geneva School of Economics and Management, University of Geneva, Blv. Pont 
d'Arve 40, 1211 Geneva, Switzerland, \url{davide.lavecchia@unige.ch}}, \hspace{0.2cm} Elvezio Ronchetti \thanks{Research Center for Statistics 
and Geneva School of Economics and Management, University of Geneva, Blv. Pont d'Arve 40,
1211 Geneva, Switzerland, \url{elvezio.ronchetti@unige.ch}}, \hspace{0.2cm} Olivier Scaillet 
\thanks{Geneva Finance Research Institute, Geneva School of Economics and Management, University of Geneva and Swiss Finance Institute, Blv. Pont d'Arve 40,
1211 Geneva, Switzerland, \url{olivier.scaillet@unige.ch}}}
  \maketitle
} \fi

\if0\blind
{
  \bigskip
  \bigskip
  \bigskip
  \begin{center}
    {\LARGE\bf Saddlepoint approximations for \\ spatial panel data models}
\end{center}
  \medskip
} \fi

\bigskip
\begin{abstract}
We develop new higher-order asymptotic techniques for the Gaussian maximum likelihood
estimator in a spatial panel data model, with fixed effects, time-varying covariates, and spatially correlated errors.   Our saddlepoint density and tail area approximation
feature \textit{relative error} of order $O(1/(n(T-1)))$  with $n$ being the cross-sectional dimension and  $T$ the  time-series dimension. The main theoretical tool is the tilted-Edgeworth technique in a non-identically distributed setting.  The density
approximation  is always non-negative, does not need resampling, and is accurate in the tails. 
Monte Carlo experiments on density approximation and testing in the presence of nuisance parameters illustrate the good performance of our approximation over first-order asymptotics and Edgeworth expansions. An empirical application to the investment-saving relationship in OECD (Organisation for Economic Co-operation and Development) countries shows disagreement between testing results based on first-order asymptotics and saddlepoint techniques.
\end{abstract}

\noindent%
\vfill

\newpage
\spacingset{1.5} 
\addtolength{\textheight}{.5in}%
\newpage
\section{Introduction} \label{Sec: Intro}

Accounting for spatial dependence is of interest both from an applied and a theoretical point of view. Indeed,
panel data with spatial cross-sectional interaction enable empirical researchers to take into account the temporal dimension and, at the same time, control for the 
spatial dependence. From a theoretical point of view, the special features of panel data with spatial effects
present the challenge to develop new methodological tools.

Much of the machinery for conducting statistical inference on panel data models has been established under 
the simplifying assumption of cross-sectional independence. This assumption may be 
inadequate in many cases. For instance, 
correlation across spatial data comes typically from competition, spillovers, or aggregation. The presence of such a 
correlation might be anticipated in observable variables and/or in the unobserved disturbances in a statistical model and
ignoring it can have adverse effects on routinely-applied inferential procedures. 
See, e.g.,  \citet{GG10}, \citet{R12}, \citet{C15}, \citet{CW15}, and recently \citet{CMW19} 
for book-length discussions in the statistical literature. In the econometric literature, see, e.g., 
\citet{KKP07}, \citet{LY10}, \citet{RR14_ET}, \citet{RR15}, and, for book-length presentations,  \citet[Ch. 13]{BaltagiBook}, \citet{An13}, and \citet{KP17}.

Different nonparametric, semiparametric, and parametric approaches have been proposed to incorporate cross-sectional
dependence in panel data models. The choice on the modeling approach depends on the time series  
($T$) and cross-sectional ($n$) dimensions.
A nonparametric approach 
is only feasible when  $T$ is large relative to  $n$. In other situations, typically when $T$ is very small (e.g., $T=2$) and $n$ is large, semiparametric models have been employed, including time varying regressors (namely factor models) 
and spatial autoregressive component, when information on spatial distances is available. 
Least squares and quasi maximum-likelihood estimator represent the main popular tools for estimation 
within this setting. 
When both 
$T$ and $n$ are small, the parametric approach is the sensible choice and
(Gaussian) likelihood-based procedures are applied to define the maximum likelihood estimator (MLE).

{There is a vast literature on the MLE for spatial autoregressive models, an early reference being \citet{O75}. The derivation of the first-order asymptotics is available in the econometric literature; we refer to the seminal paper by \citet{L04}.  For the class of spatial autoregressive processes, with fixed effects, time-varying covariates, and spatially correlated errors that we consider in this paper, the first-order asymptotic results for the Gaussian MLE  are available  in \citet{LY10}, where the authors derive asymptotic approximations (the exact finite-sample distribution 
being intractable), when the cross-sectional dimension $n$ is large and $T$ is finite or large.} 

{The main issue related to first-order asymptotic approximations is that, when $n$ is not very large, such approximations may be unreliable: alternatives are highly recommended.  \citet{BU07} provide analytic formulae for the second-order bias and mean squared error of the MLE for the spatial parameter $\lambda$, in a Gaussian model. \citet{B13}  and \citet{Ya05} extend these approximations to include also exogenous explanatory variables, which remain valid also when the process is not Gaussian.  \citet{RR14_EJ,RR14_ET} develop Edgeworth-improved tests for no 
spatial correlation in spatial autoregressive models for pure cross-sectional data based on least squares estimation and Lagrange 
multiplier tests. Moreover, \citet{RR15} work on the concentrated likelihood and derive an Edgeworth expansion for MLE of $\lambda$  
in the setting of a  first-order spatial autoregressive panel data model, with fixed effects and without covariates.  \citet{HM18} (see their \S 6) and \citet{HM20} (see their \S 3.5) propose saddlepoint approximations  for the profile likelihood estimator of $\lambda$.} 

Resampling methods are also available alternatives to improve on the first-order asymptotics, 
achieving higher-order asymptotic refinements in terms of {\it absolute} error.
However, it requires
either a bias correction or an asymptotically pivotal statistics;
see \citet{Hall92} and \citet{Horowitz01}
in the i.i.d.\ setting. To the best of our knowledge, for spatial panel models considered in this paper, such results are not
available. 

The aim of this paper is to introduce \textit{saddlepoint approximations}
for parametric spatial autoregressive panel 
data models with fixed effects and time-varying covariates. They overcome the problems mentioned above by means
of the tilted-Edgeworth technique. For general references on saddlepoint approximations in the i.i.d.\ setting,
see the seminal paper of \cite{D54} and the book-length presentations of
\cite{FR90}, \cite{Jensen95}, \cite{K06}, 
and  \cite{brazzaleetal2007}. For a result about testing on spatial dependence, see \citet{T02}, and for developments 
in time series models, see \citet{LR19}.   

We remark that we could cast the methodology of this paper into the framework of statistical analysis of random fields on a network graph, where
the underlying,  known, network graph describes the spatial structure of the stochastic process; see e.g. \citet{K09} Ch. 8 for a book-length introduction. In \S 2, we briefly comment on this approach. For the ease-of-reference to the extant econometric literature,
in the rest of the paper, we prefer to stick to the econometric notation and terminology of spatial panel data models.

The paper is organized as follows. In \S\ref{Sec: Motivation}, we provide a
motivating example.  
\S\ref{Sec: Setting} defines the general model setting and the estimation method, whereas the detailed methodology is presented in
\S\ref{Sec: Methodology}.
In particular in \S\ref{Sec: Links_econometrics} we provide a detailed discussion about the connections with the econometric literature. 
Algorithms and computational aspects are
discussed in \S\ref{Sec: algm}. 
\S\ref{Sec: MC} provides numerical comparison with other methods and, in \S\ref{Sec: test_nuisance_parameters}, we tackle the problem of testing in the presence of nuisance parameters.  
In \S\ref{Sec: Real}, we present an empirical application.  The online 
Supplementary Material contains detailed derivations, technical appendices and additional numerical experiments.

\vspace{-0.5cm}

\section{Motivating example} \label{Sec: Motivation}

We motivate our research by a Monte Carlo (MC) exercise illustrating  the low accuracy of the routinely applied first-order asymptotics in the setting of spatial panel data model. We consider the model:
\begin{equation}
\begin{aligned}
Y\nt &= \lambda_0 W_n Y\nt + X\nt\beta_0 + c_{n0} + E\nt,& \\ 
E\nt &=\rho_0 M_n E\nt +V\nt, & \quad t=1,....,T, & 
\end{aligned}
\label{Eq: GeneralY}
\end{equation}
where $Y\nt = (y_{1t},y_{2t},...,y\nt)$,  $X\nt$ is an $n \times k$ matrix of non stochastic time-varying regressors,  $c_{n0}$ is an $n \times 1$ vector of fixed effects, and $V\nt  =(v_{1t},v_{2t},..,v\nt)'$ are $n \times 1$ vectors with  $v_{it} \sim \mathcal{N}(0,\sigma_0^2)$, i.i.d. across $i$ and $t$. The matrices $W_n$ and $M_n$ are weighting matrices 
describing the spatial dynamics. Following the literature, we label this model SARAR(1,1) to emphasize the spatial dependence in both the response variable $Y\nt$ and the error $E\nt$.

As in the  MC example in \citet{LY10} p.\ 172, we generate samples from (\ref{Eq: GeneralY}) using $\theta_0 = (\beta_0,\lambda_0,\rho_0,\sigma_0^2)' =  (1.0,0.2,0.5,1)'$, $T=5$, and $k=4$ covariates. The quantities $X_{nt}$, $c_{n0}$ and $V_{nt}$ are generated from independent standard normal distributions and, as it is customary in the econometric literature, we set $W_n = M_n$, where  the off-diagonal elements are different from zero, while the diagonal elements are all zero. We consider two sample sizes: $n=24$ (small sample) and $n=100$ (moderate/large sample). The choice of $n=24$ is related to the empirical data analysis that we consider in \S \ref{Sec: Real}, where we apply the model in (\ref{Eq: GeneralY}) to conduct inference on the investment-saving relation for the 24 OECD (Organisation for Economic Co-operation and Development) countries. Similar sample sizes arise in many 
real data analyses, where panel datasets contain a limited number of cross-sectional units, e.g., 
sampling can be expensive and/or time consuming, as it is typically the case in field studies.

As it is customary in the statistical/econometric software, we illustrate the inference issues related to the use of the first-order asymptotics by means of three different spatial weight matrices: Rook matrix, Queen matrix, and Queen matrix with torus.  In Figure \ref{Fig: Wn1}, we display the geometry of $Y_{nt}$ as implied by each considered spatial matrix: the plots highlight that different matrices imply different spatial relations. For instance, we see that the Rook matrix implies fewer links than the Queen matrix. Indeed, the Rook criterion defines neighbours by the existence of a common edge between two spatial units, whilst the Queen criterion is less rigid and defines neighbours as spatial units sharing an edge or a vertex. Besides, we may interpret $\{Y_{nt}\}$ as a $n$-dimensional random field on the network graph which describes the known underlying spatial structure. Then, $W_n$ represents the weighted adjacency matrix (in the spatial econometrics literature, $W_n$ is called contiguity matrix). In Figure \ref{Fig: Wn1}, we display the geometry of a random field on a regular lattice (undirected graph). 
In the real data example of \S \ref{Sec: Real}, we consider a random field over a manifold (a sphere), providing two additional examples for $W_n$.

\begin{figure}[htbp]
\begin{center}
\begin{tabular}{l}
\hspace{2.cm} Rook \hspace{3.75cm} Queen   \hspace{3.8cm} Queen torus   \\

\includegraphics[width=0.25\textwidth, height=0.25\textheight]{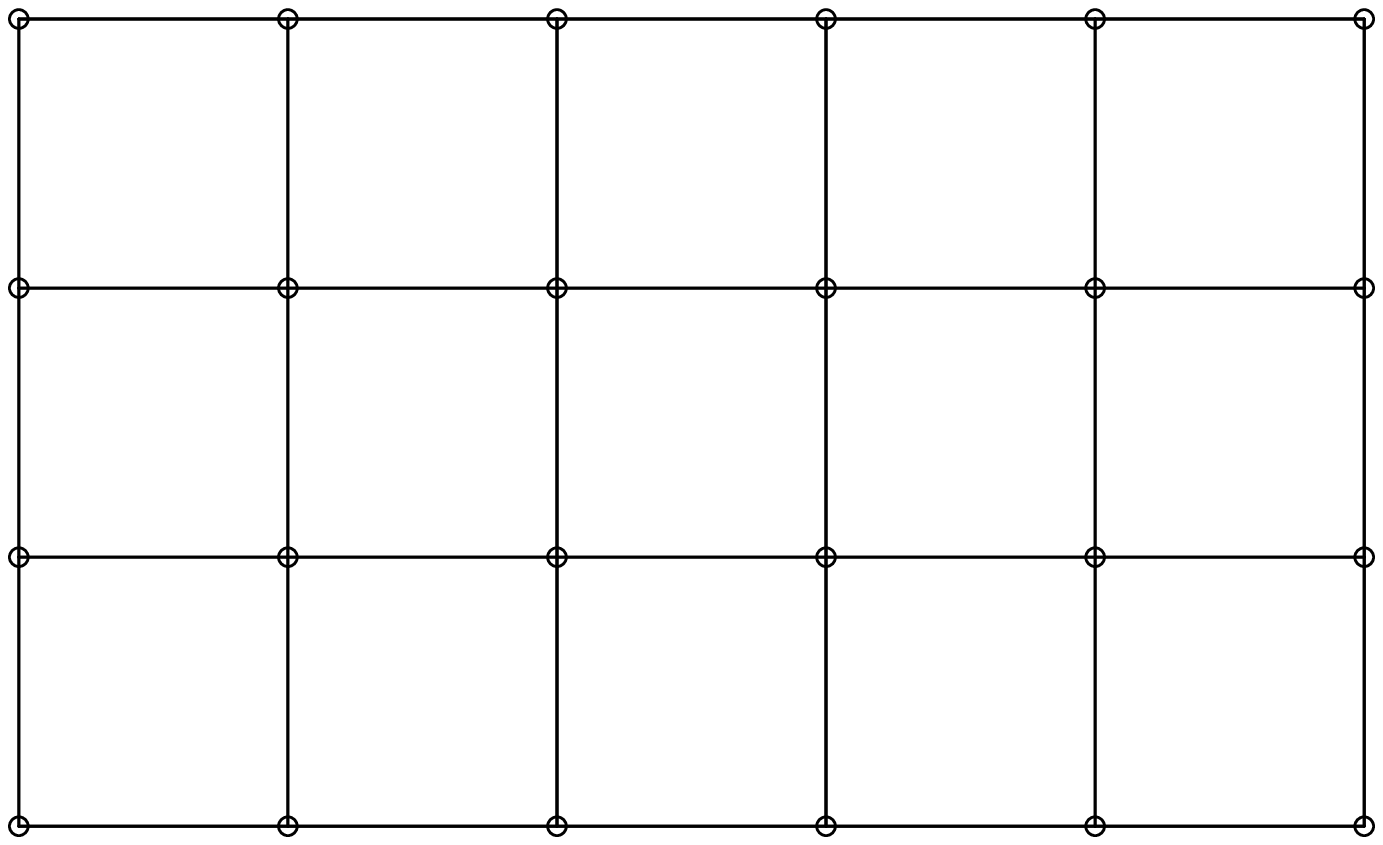}\hspace{1.4cm}
\includegraphics[width=0.25\textwidth, height=0.25\textheight]{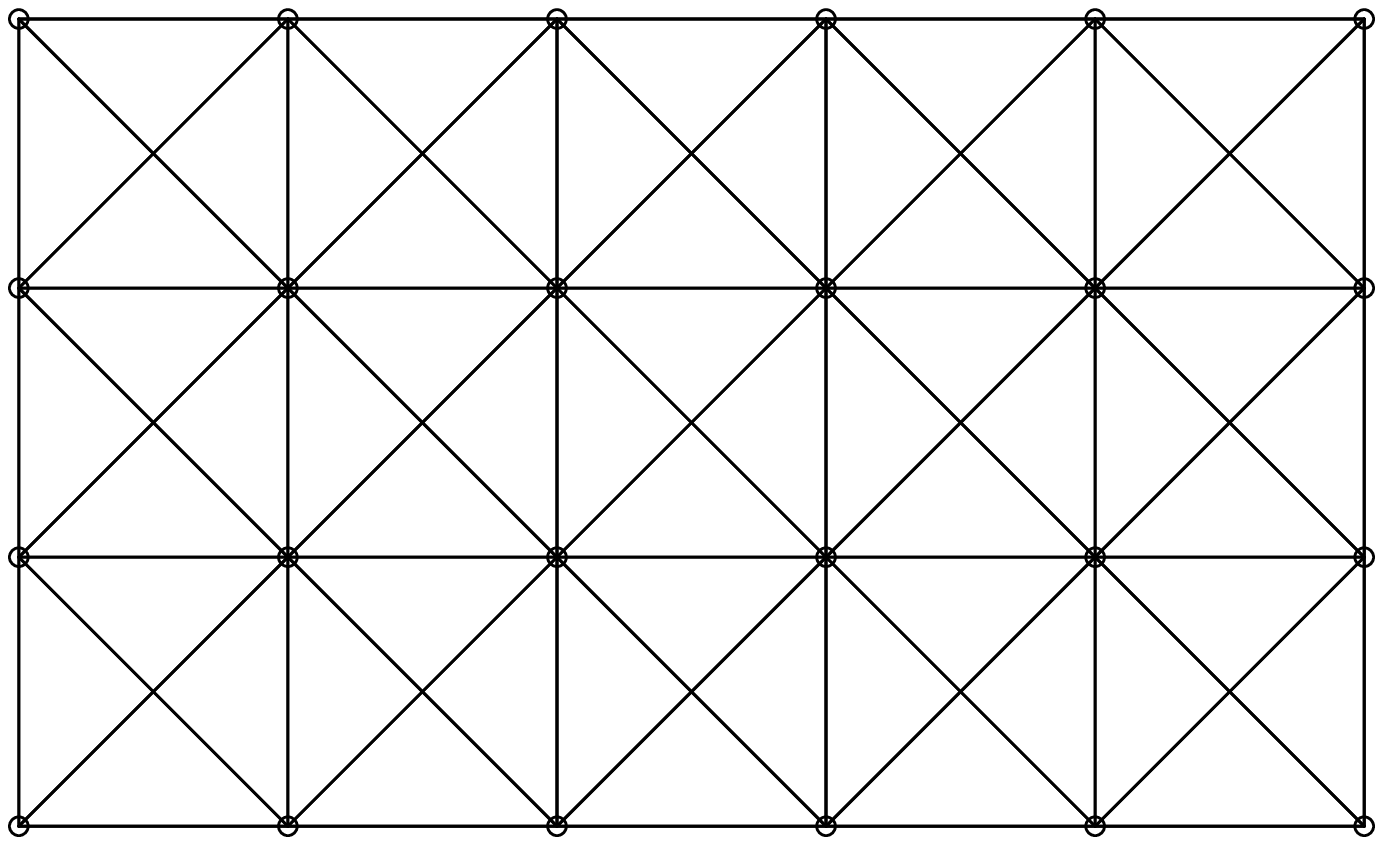}\hspace{1.5cm}
\includegraphics[width=0.25\textwidth, height=0.25\textheight]{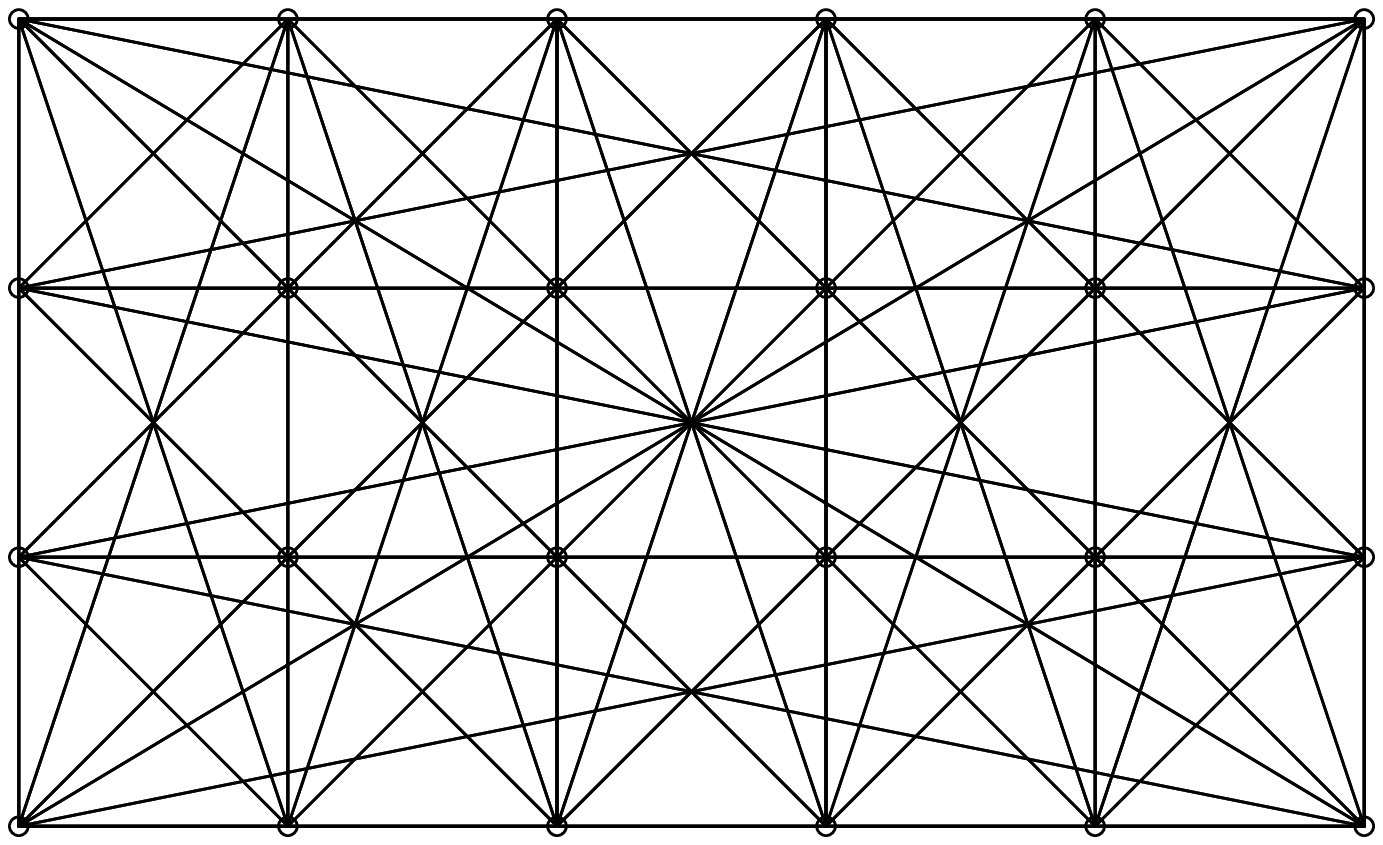}\\
\end{tabular}
\caption{Different types of neighboring structure for $Y_{nt}$, as implied by different types of $W_n$ matrix, for $n=24$.} 
    \label{Fig: Wn1}
\end{center}
\end{figure}

To illustrate the inferential issues entailed by the use of first-order asymptotics, for each type of $W_n$,  we generate a sample of $n$ observations. Since $c_{n0}$ creates an incidental parameter issue, we 
eliminate it by the standard differentiation procedure, and  for each MC run we estimate the model parameter $\theta$ using the transformation approach of \citet{LY10}, with maximum likelihood estimation method;  we refer to the R package \texttt{spml} for implementation details.  We set the MC size to 5000.  

We illustrate graphically the behavior of the first-order asymptotic theory in finite sample by comparing the distribution of $\hat\lambda$ to the Gaussian asymptotic distribution (see \S \ref{Sec: M_fun} 
for details). Via QQ-plot analysis,  Figure \ref{fig1} shows that the Gaussian approximation can be either too thin or too thick in the tails with respect to the ``exact'' distribution.  For instance, when $n=24$ and $W_n$ is rook, the Gaussian quantiles are larger than the ``exact'' ones in the left tail, while we observe the opposite phenomenon in the right tail. Similar considerations hold for the other types of $W_n$. The more complex is
the geometry of $W_n$ (e.g., $W_n$ has Queen structure) the more pronounced are the departures from the Gaussian. For  $n=100$, and $W_n$ Rook, the MLE displays a distribution which is in line with the Gaussian one (see bottom left panel). However, when $W_n$ becomes more complex (e.g., Queen with torus),
larger departures in the tails are still evident. In Appendix D.1, we illustrate that similar conclusions are available also for the simpler SAR(1) model:
\begin{equation}
\begin{aligned}
Y\nt &=& \lambda_0 W_n Y\nt + c_{n0} + V\nt,   \quad {\text{for}} \quad t=1,2, 
\end{aligned}
\label{Eq: RR}
\end{equation}
where $\theta_0=(\lambda_0,\sigma_0^2)'$. More generally, unreported results suggest that, in the considered SARAR setting, the ``exact'' and the asymptotic distribution, as well as the saddlepoint approximation, agree for the considered types of $W_n$,  when $n\geq250$.  \\

\begin{figure}[htbp]
\begin{center}
\begin{tabular}{l} \hspace{3cm} Rook \hspace{3.1cm} Queen   \hspace{2cm} Queen torus   \\

\begin{turn}
{90} \hspace{2.1cm}  $n = 24$ 
\end{turn}
\includegraphics[width=0.9\textwidth, height=0.175\textheight]{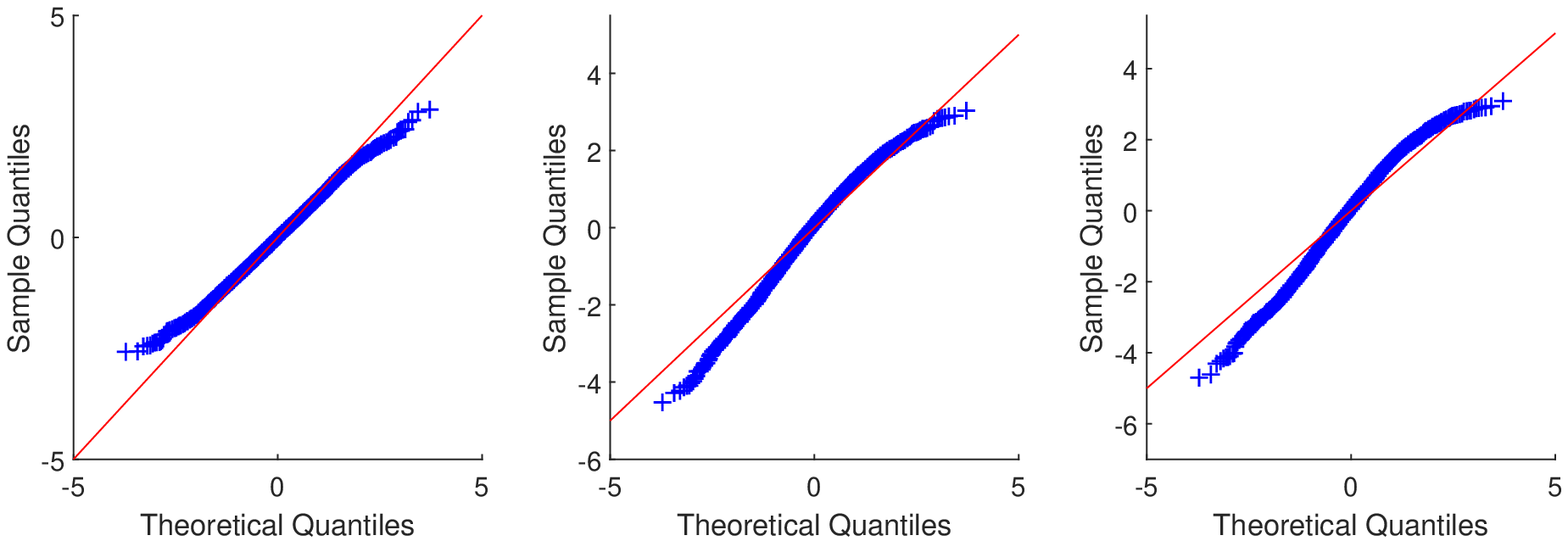} \\
\begin{turn}
{90} \hspace{2.1cm} $n=100$
\end{turn}
\includegraphics[width=0.9\textwidth, height=0.175\textheight]{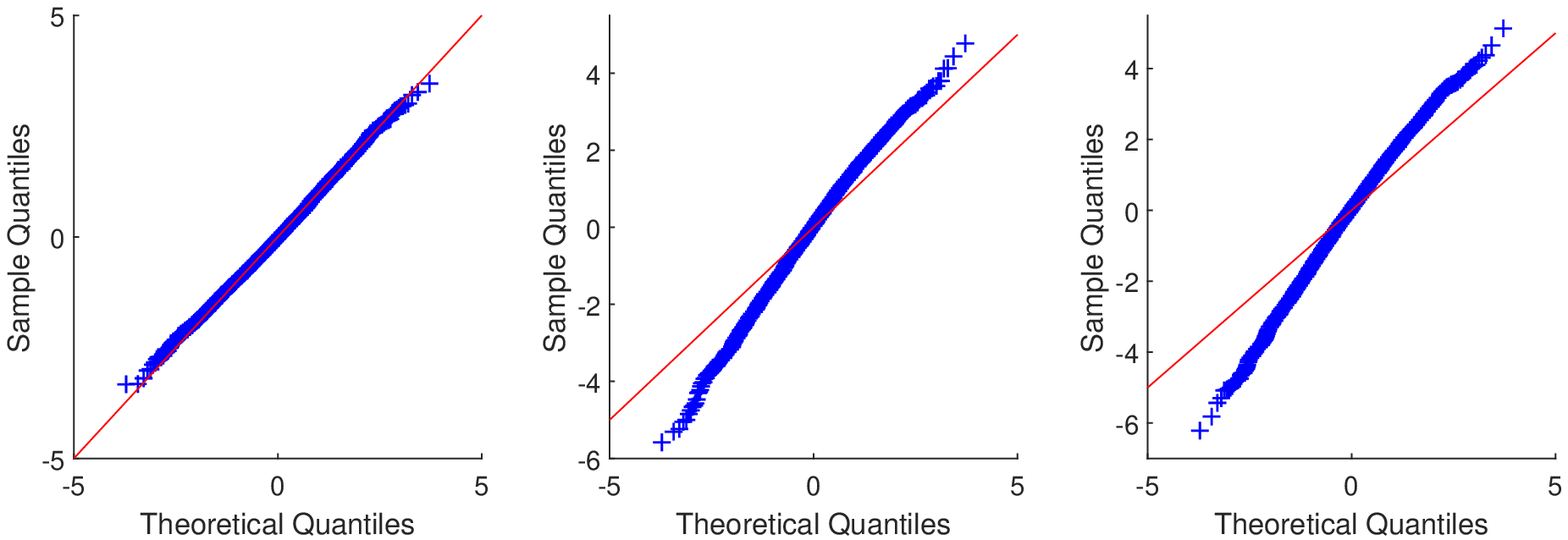} 
\end{tabular}
\caption{SARAR(1,1) model: QQ-plot vs normal of the MLE $\hat\lambda$, for different sample sizes ($n=24$ and $n=100$), $\lambda_0=0.2$, and different types of $W_n$ matrix.}
    \label{fig1}
\end{center}
\end{figure}

\section{Model setting and estimation method} \label{Sec: Setting}

Let us consider a random field described by the {SARAR(1,1) model in (\ref{Eq: GeneralY}).
{We label by $P_{\theta_0}\in \mathcal{P}$, with $\theta_0 \in \Theta \subset \mathbb{R}^d$, the actual underlying distribution, which is characterized by $\theta_0 = (\beta_0,\lambda_0,\rho_0,\sigma_0^2)'$, the true parameter value.} The matrix $W_n$ is an $n \times n$ nonstochastic spatial weight matrix that generates the spatial dependence 
on $y_{it}$ among cross sectional units. The matrix $X\nt$ is an $n \times k$ matrix of non stochastic time varying regressors, and $c_{n0}$ is an $n \times 
1$ vector of fixed effects. Similarly, $M_n$ is an $n \times n$ spatial weight matrix for the disturbances --- quite often $W_n = M_n$. Moreover, we define $\Snl= I_n - \lambda W_n$, and  analogously
 $\Rnr = I_n - \rho M_n$.

The vector $c_{n0}$ introduces an incidental parameter problem; see \citet{LY10} and  \citet{RR15}. To cope with this issue, we follow the standard approach, and we transform the model in order to derive consistent estimator for the model parameter $\theta=(\beta',\lambda,\rho,\sigma^2)'$ and $\theta\in \Theta \subset \mathbb{R}^{d}$. To achieve the goal, we first  eliminate the individual effects by the deviation from the time-mean operator $J_T = (I_T - \frac{1}{T} l_T l'_T)$, where
$I_T$ is the $T \times T$ identity matrix, and $l_T = (1,...,1)'$, namely the $T\times 1$ vector of ones. Without creating linear dependence in the resulting
disturbances, we adopt the transformation introduced by \citet{LY10}. 
 
First, let the orthonormal eigenvector matrix of $J_T$ be 
$[F_{T,T-1}, \frac{1}{\sqrt{T}} l_T]$, {where $[\cdot]$ represents a matrix horizontal concatenation and} $F_{T,T-1}$ is the $T \times (T-1)$ submatrix
corresponding to the unit eigenvalues. 
Then, for any $n \times T$ matrix $[Z_{n1}, ... , Z_{nT}]$, we
define the transformed $n\times (T - 1)$ matrix 
$
[Z^*_{n1}, ... , Z^*_{nT}]=[Z_{n1}, ... , Z_{nT}]F_{T,T-1}.
$ 
Similarly, we define $X^*\nt = [X^*_{nt,1},X^*_{nt,2},...,X^*_{nt,k}]$. Thus, we transform the model in (\ref{Eq: GeneralY}) and we obtain:

\begin{equation}
\begin{aligned}
Y^*\nt &= \lambda_0 W_n Y^*\nt + X^*\nt\beta_0 + E^*\nt, &\\ 
E^*\nt &= \rho_0 M_n E^*\nt +V^*\nt, &\quad t=1,2,...,T.&  
\end{aligned}
\label{Eq: GeneralY_Transf}
\end{equation}

Since $
\l(V^{*'}_{n1},...,V^{*'}_{n(T-1)}\r)' = \l(F'_{T,T-1} \otimes I_n\r)\l(V'_{n1},...,V'_{n(T-1)}\r)',
$
and the $v_{it}$ are i.i.d., we have
$$
\mathbb{E}\left[\left(V^{*'}_{n1},...,V^{*'}_{n(T-1)}\right)'\left(V^{*'}_{n1},...,V^{*'}_{n(T-1)}\right)\right] = \sigma_0^2 I_{n(T-1)},
$$
where $\mathbb{E}[\cdot]$ represents the expectation taken w.r.t. $P_{\theta_0}$. 
{Now, the Gaussian assumption on the innovation terms implies that   $v^{*}_{it}$ are independent for all $i$ and $t$; without this
assumption, they would be simply uncorrelated. See \citet{LY10} p. 167.} 
Thus, defining $\zeta =(\beta',\lambda,\rho)'$, the log-likelihood is:
\begin{eqnarray*}
\ln L_{n,T}(\theta) = \ell_{n,T}(\theta) &=& - \frac{n(T-1)}{2} \ln (2\pi\sigma^2) + (T-1) [\ln \vert \Snl \vert + \ln \vert \Rnr \vert] \nonumber \\
 &  &  \qquad - \frac{1}{2\sigma^2} \sum_{t=1}^{T-1}  V^{*'}_{nt}(\zeta)V^*_{nt}(\zeta), 
\label{Eq: LogL1}
\end{eqnarray*}
where $V^*_{nt}(\zeta) = \Rnr [\Snl Y^*\nt - X^*\nt \beta].$ {As remarked in \citet{LY10}, the function $L_{n,T}$  has a conditional likelihood interpretation: it is the likelihood conditional on the time average $\sum_{t=1}^{T} Y_{nt}/T$, which is a sufficient statistic for $c_{n0}$, under normality.}

We rewrite $\ell_{n,T}(\theta)$  in terms of a quadratic form in $\tilde{V}_{nt}(\zeta)$ as: 
\begin{eqnarray}
 \ell_{n,T}(\theta) &=& - \frac{n(T-1)}{2} \ln (2\pi\sigma^2) + (T-1) [\ln \vert \Snl \vert + \ln \vert \Rnr \vert] \nonumber \\
 &  & - \qquad   \frac{1}{2\sigma^2} \sum_{t=1}^{T}  \tilde{V}'_{nt}(\zeta)\tilde{V}_{nt}(\zeta), 
\label{Eq: LogL2}
\end{eqnarray}
where 
$
\tilde{V}_{nt}(\zeta) =  \Rnr [\Snl \tilde{Y}\nt - \tilde{X}\nt \beta],
$
with 
\begin{equation}
\tilde{Y}\nt = {Y}\nt - \sum_{t=1}^{T} Y\nt/T,  \quad \tilde{X}\nt = X\nt - \sum_{t=1}^{T} X\nt/T. \label{Eq. transf}
\end{equation}

The MLE $\hat\theta_{n,T}$ for $\theta$ is 
an $M$-estimator obtained by solving 
$
\hat\theta_{n,T}=\text{arg}\max_{\theta \in \Theta} \ell_{n,T}(\theta).
$
It implies the system of estimating equations:
\begin{equation}
0=\frac{\partial\ell_{n,T}(\hat\theta_{n,T})}{\partial {\theta} } =  \sum_{t=1}^{T} (T-1)^{-1}\psi\nt(\hat\theta_{n,T}),  
\label{Eq: M_est}
\end{equation}
{ where $\psi\nt$
is the likelihood score function}
\begin{equation} \label{Eq: score_text}
\psi\nt(\theta) = \left( \begin{array}{cc} 
     \frac{T-1}{\sigma^2}  (\Rnr \tilde{X}\nt)' \tilde{V}_{nt}(\zeta) \\ 
     \frac{T-1}{\sigma^2}\left(  \left(\ddot{G}_n\ddot{X}_{nt}\beta\right)' \tilde{V}_{nt}(\zeta) + \tilde{V}_{nt}^{'}\ddot{G}_{n}^{'}\tilde{V}_{nt}\right)-\frac{(T-1)^{2}}{T}\tr({G}_n(\lambda))
\\
     \frac{T-1}{\sigma^2}   (\Hn  \tilde{V}\nt(\zeta))' \tilde{V}_{nt}(\zeta) - \frac{(T-1)^{2}}{T}\tr(H_n(\rho)) \\ 
     \frac{T-1}{2\sigma^4}   \left(\tilde{V}'\nt(\zeta)\tilde{V}\nt(\zeta)  -  \frac{n(T-1)}{T} \sigma^2 \right)
      \end{array} \right),
\end{equation}
where $ G_n(\lambda) = W_n S_n^{-1}, \quad \Hn = M_n R_n^{-1},  \quad \ddot{G}_n(\lambda) = R_nG_nR_n^{-1}$, and $\ddot{X}_{nt} =R_n \tXn.$

\section{Methodology}
\label{Sec: Methodology}
We assume $ n \gg T$, so we deal with so-called \textit{micro panels}: in the econometric literature, this type of data typically involve annual records covering a short time span for each individual. Within this setting for  $T$ being fixed,
the standard asymptotic arguments  rely crucially on the number $n$
of individuals  tending to infinity; see \citet{LY10}. In contrast, in our development, 
we consider small sample cross-sectional asymptotics (\citet{FR90}), and we still leave $T$ fixed (possibly small). However, 
we will keep $T$ in the notation of normalizing factors to demonstrate the
improved rate of convergence that would result if $T \to \infty$ or it is large. The derivation of our higher-order techniques relies on three steps: (i) defining a second-order asymptotic (von Mises) expansion for the MLE, see \S\ref{Sec: M_fun}; (ii) identifying the corresponding $U$-statistic, see \S\ref{Sec: IIordervM}; 
(iii) deriving the Edgeworth expansion for the $U$-statistic as in \citet{BGVZ86} and deriving the saddlepoint 
density by means of the tilted-Edgeworth technique, see \S \ref{Sec: Ustat} and \S\ref{Sec: sadd}. 
Similar approaches are available in the standard setting of i.i.d.\ random variables in \citet{ER86}, \citet{BNSC89}, and \citet{GR96}.

\subsection{The $M$-functional related to the MLE and its first-order asymptotics} \label{Sec: M_fun}

Let us first define the $M$-functional related to the MLE.  To this end, we remark that the likelihood score function in (\ref{Eq: score_text}) is a vector in $\mathbb{R}^{d}$,  and each $l$-th element of 
this vector, for $l=1,...,d$, is a sum of $n$ terms. In what follows, for $i=1,...,n$, we denote by $\psi_{i,t,l}(\theta)$ the $i$-th term, at time $t$, of this sum for the $l$-th component of the score. 

To specify $\psi_{i,t,l}(\theta)$, we set 
$\Rnr =\left (r_1^{'}(\rho), r_2^{'}(\rho), \cdots, r_n^{'}(\rho)\right)',$ 
$$\tXn=\left[\tilde{X}_{nt,1},\tilde{X}_{nt,2}, \cdots, \tilde{X}_{nt,k}\right],$$
$\tilde{V}_{nt}(\zeta) =\left (\tilde{v}_{1t}(\zeta), \tilde{v}_{2t}(\zeta), \cdots, \tilde{v}_{nt}(\zeta)\right)'$ and $\Hn =\left (h_1^{'}(\rho), h_2^{'}(\rho), \cdots, h_n^{'}(\rho)\right)'$, where $r_i(\rho)$ and $h_i(\rho)$ are the $i_{th}$ row of $R_n(\rho)$ and $H_n(\rho)$, $g_{ii}$ and $h_{ii}$ are $i_{th}$ element of the diagonal of $G_n(\lambda)$ and $H_n(\rho)$, respectively. Then, from (\ref{Eq: score_text}), it follows 
\begin{equation} 
\psi_{i,t}(\theta)=\left( \begin{array}{c} \psi_{i,t,1}(\theta),\\ \psi_{i,t,2}(\theta) \\  \vdots \\ \psi_{i,t,d}(\theta)
    \end{array} \right)_{d\times 1}
= \left( \begin{array}{cc} 
     \frac{T-1}{\sigma^2}  r_i(\rho)\tilde{X}_{nt,1}\tilde{v}_{it}(\zeta) \\ 
      \frac{T-1}{\sigma^2}  r_i(\rho)\tilde{X}_{nt,2}\tilde{v}_{it}(\zeta)\\
    \vdots \\ 
      \frac{T-1}{\sigma^2}  r_i(\rho)\tilde{X}_{nt,k}\tilde{v}_{it}(\zeta) \\
     \frac{T-1}{\sigma^2}r_i(\rho) \left(G_n\tXn \beta +G_nR_n^{-1}(\rho)\tilde{V}_{nt}(\zeta)\right) \tilde{v}_{it}(\zeta) - \frac{(T-1)^{2}}{T}g_{ii} \\
      \frac{T-1}{\sigma^2} h_i(\rho)\tilde{V}_{nt}(\zeta)  \tilde{v}_{it}(\zeta) - \frac{ (T-1)^{2}}{T}h_{ii}\\
      \frac{T-1}{2\sigma^4} \left(\tilde{v}_{it}(\zeta)^2 - \frac{T-1}{T} \sigma^2 \right)
      \end{array} \right)_{d\times 1} . \label{Eq: s_it}
\end{equation}
Thus, for every $t=1,2,...,T$, we have
$\displaystyle 
\psi\nt(\theta)=\left(\sum_{i=1}^{n} \psi_{i,t,1}(\theta),..., \sum_{i=1}^{n} \psi_{i,t,d}(\theta)\right)',
$
and, from (\ref{Eq: M_est}), it follows that the MLE is the solution to
\begin{equation}
 \frac{1}{n} \sum_{t=1}^{T} \left(\sum_{i=1}^{n} (T-1)^{-1}\psi_{i,t,1}(\hat\theta_{n,T}),..., \sum_{i=1}^{n} (T-1)^{-1}\psi_{i,t,d}(\hat\theta_{n,T})\right)'  = 0. \label{Eq. MestPn}
\end{equation}
The $M$-functional $\vartheta$ related to the MLE  
is implicitly defined as 
the unique functional root of: 
\begin{equation}
\mathbb{E}\left\{\sum_{t=1}^{T} \left({T-1}\right)^{-1} \psi\nt\left[\vartheta(P_{\theta_0})\right] \right\}=0, \label{Eq: Mfunctional}
\end{equation}
{or equivalently via the asymptotic maximization  $\theta_0 = \text{arg}\max_{\theta \in \Theta}  \mathbb{E}[\ell_{n,T}(\theta_0)]$; see e.g., \citet{L04}. In what follows,  we write $\theta_0 = \vartheta(P_{\theta_0})$ to emphasize the dependence of the functional on the measure $P_{\theta_0}$.}
{The finite sample version of the $M$-functional in (\ref{Eq: Mfunctional}) is the $M$-estimator defined in (\ref{Eq. MestPn}), or equivalently via the finite sample maximization  $\hat\theta_{n,T}=\text{arg}\max_{\theta \in \Theta} \ell_{n,T}(\theta)$. 
In what follows, we write $\hat\theta_{n,T}= \vartheta(P_{n,T})$, where  $P_{n,T}$ is the measure associated to the $n$-dimensional  sample}. We can check the  uniqueness of the M-estimator  on a case-by-case basis, using Assumption A (see below) and working on the Gaussian log-likelihood. For instance, in the case of the SAR model, we can compute the second derivative of $\ell_{n,T}$ w.r.t. $\lambda$  and check that  $\ell_{n,T}$ is a concave function, admitting a unique maximizer. Alternatively, we can solve the estimating equations implied by first-order conditions related to $\ell_{n,T}$ resorting on a one-step procedure and using for instance the GMM estimator (see \citet{LY10} and reference therein) as a preliminary estimator; for a book-length description of one-step procedure; see, e.g., \citet{vdW98} Ch. 5.

In what follows,  we set
$m:=n(T-1)$, with  $m\rightarrow\infty$, as $n\rightarrow\infty$. Then, we introduce 

\textbf{Assumption A.}{\it
\begin{enumerate}[(i)]
\item The elements $\omega_{n,ij}$ of $W_n$ and the elements $m_{n,ij}$ of $M_n$
in (\ref{Eq: GeneralY}) are at most of order $\tilde{h}_n^{-1}$, denoted by $O(1/\tilde{h}_n)$, uniformly in all i,j, where the rate sequence $\{\tilde{h}_n\}$ is bounded, and $\tilde{h}_n$ is bounded away from zero for all $n$. 
As a normalization, we have $\omega_{n,ii} = m_{n,ii}=0$, for all i.
\item $n$ diverges, while $T\geq 2$ and it is finite.
\item Assumptions 2-5 and Assumption 7  in \citet{LY10} are satisfied.
\item Denote $C_n = \ddot{G}_n-{n}^{-1}{tr(\ddot{G}_n})I_n$ and $D_n = H_n- {n}^{-1} {tr(H_n)}I_n$ where $\ddot{G}_n = R_nG_nR_n^{-1}$ and $H_n=M_nR_n^{-1}$. Then $C_n^s = C_n+C_n^{'}$ and $D_n^s= D_n+D_n^{'}$. The limit of ${n^{-2}}\left[ tr(C_n^sC_n^s)tr(D_n^sD_n^s)-tr^2(C_n^sD_n^s)\right]$ is strictly positive as
$n \rightarrow \infty$.
\end{enumerate}
}

Assumptions A$(i)$ 
characterizes the behavior of $W_n$ and $M_n$ in terms of $n$, and $W_n$ and $M_n$ are row-normalized. It means  $\omega_{n,ij}= d_{ij}/\sum_{j=1}^{n}d_{ij}$, where $d_{ij}$ is the spatial distance of the $i-${th} and the $j-${th} units in some (characteristic) space. For each $i$, the weight $\omega_{n,ij}$ defines an average of neighboring values. In what follows, we consider spatial weight matrices (like, e.g., Rook and Queen) such that $\sum_{j=1}^n d_{ij}=O(\tilde{h}_n)$ uniformly in $i$ and the row-normalized weight matrix satisfies Assumption A$(i)$; see e.g. \citet{L04}. For instance,  $W_n$ as Rook creates a square tessellation with $\tilde{h}_n=4$ for the inner fields on the chessboard, and $\tilde{h}_n=2$ and $\tilde{h}_n=3$ for the corner and border fields, respectively. 
Assumption A$(ii)$ defines the asymptotic scheme of our theoretical development, in which 
we consider $n$ cross-sectional units and we leave $T$ fixed. Assumption A$(iii)$ refers to \citet{LY10}, who develop the first-order asymptotic theory.
All $W_n$, $M_n$, $S_n^{-1}(\lambda)$, $R_n^{-1}(\rho)$ are uniformly bounded by Assumption A$(iv)$ which guarantees the convergence of the asymptotic variance, see below. 
 Assumption A$(iv)$ states the identification conditions of the model and the conditions for the nonsingularity of the limit of the information matrix. In particular, it implies that the $(d\times d)$-matrix
\begin{equation}
M_{i,T}(\psi,P_{\theta _{0}}) = \mathbb{E}\left[-(T-1)^{-1}\sum_{t=1}^T {\partial  \psi_{i,t}(\theta)}/{\partial\theta}\Big\vert_{\theta = \theta_0}\right] 
\label{Eq: Mit}
\end{equation}
is non-singular. {Under Assumption A$(i)$-A$(iv)$, Theorem 1 part(ii) in \citet{LY10} shows 
that
$\displaystyle 
\lim_{n \rightarrow \infty} \hat\theta_{n,T} 
= \theta_0 $.} Furthermore, Theorem 2 point (ii) in \citet{LY10} implies, as $n \to \infty$, that 
$\hat\theta_{n,T}$ satisfies
$
\sqrt{m} (\hat\theta_{n,T}- \theta_0) \cd \mathcal{N} \left( 0, \Sigma^{-1}_{0,T} \right),$ and $\Sigma_{0,T} 
= \plim_{n \rightarrow \infty}  \Sigma_{0,n,T} 
$. The operator $\plim$ stands for the limit in probability and the expression of $ \Sigma_{0,n,T}$ is available in the online Supplementary Material (see Appendix B). The first-order asymptotics is obtained letting $n\rightarrow \infty$; 
 there is no need for $T\rightarrow \infty$ to obtain a consistent and asymptotically normal $M$-estimator.

\vspace{-0.4cm}

\subsection{Second-order von Mises expansion} \label{Sec: IIordervM}

To define a higher-order density approximation to the finite-sample density of the MLE, we need to derive its higher-order asymptotic expansion, making use of \\
\textbf{Assumption B.} {\it 
\begin{enumerate}[(i)]
\item 
${\partial^2 \psi_{i,t,l}(\theta)}/{\partial\theta \partial\theta'}$ exists at $\theta=\theta_0$, for every $i=1,..,n$, $t=1,...,T$ and $l=1,..,d$. \\
\item 
The $d\times d$-matrix
$\mathbb{E} \left[(T-1)^{-1} \sum_{t=1}^{T} {\partial^2 \psi_{i,t,l}(\theta)}/{\partial\theta \partial\theta'}\Big\vert_
{\theta=\theta_0}  \right ]$ 
is positive semi-definite, for every $l=1,..,d$. 
\end{enumerate} 
}
Then, we state the following 
\begin{lemma} \label{Lemma_IIvM}
Let the MLE be defined as in (\ref{Eq: M_est}). Under Assumptions A-B, the following expansion holds:
\begin{equation} \label{IIorder_txt}
\vartheta(P_{n,T})-\vartheta(P_{\theta _{0}})=\frac{1}{n}   \sum_{i=1}^n  IF_{i,T} (\psi,P_{\theta _{0}}) +\frac{1}{2n^2}  \sum_{i=1}^n \sum_{j=1}^n \varphi_{i,j,T}(\psi, P_{\theta _{0}}) + O_P(m^{-3/2}), 
\end{equation}
where 
\begin{equation} \label{IF_it}
IF_{i,T}(\psi,P_{\theta _{0}})=M_{i,T}^{-1}(\psi,P_{\theta _{0}}) (T-1)^{-1}\sum_{t=1}^T \psi_{i,t}(\theta_{0}), 
\end{equation}
and 
\begin{eqnarray} \label{kerneldue_txt}
\varphi_{i,j, T}(\psi,P_{\theta_0})
&=& IF_{i,T}(\psi, P_{\theta_0})+IF_{j,T}(\psi, P_{\theta_0})
+ M_{i,T}^{-1}(\psi, P_{\theta_0})\Gamma_{i,j,T}(\psi,P_{\theta_0})  \notag \\
&&+M_{i,T}^{-1}(\psi,P_{\theta _{0}})\left\{(T-1)^{-1}\sum_{t=1}^{T}\frac{\partial \psi_{j,t}(\theta)}{\partial\theta}\Big\vert_{\theta_0}IF_{i,T}(\psi,P_{\theta_0}) \notag \right. \\ 
 && \left.+ (T-1)^{-1}\sum_{t=1}^{T} \frac{\partial \psi_{i,t}(\theta)}{\partial\theta}\Big\vert_{\theta=\theta_0} IF_{j,T}(\psi, P_{\theta_0})\right\},
 \end{eqnarray}
where

\begin{equation}
\Gamma_{i,j,T}(\psi,P_{\theta_0})'=\left(
\begin{array}{ccc}
IF'_{j,T}(\psi,P_{\theta_0})& \mathbb E\left[\sum_{t=1}^{T}  \frac{\partial^2 \psi_{i,t,1}(\theta_0)}{\partial\theta \partial\theta'}\Big\vert_
{\theta=\theta_0}  \right
]&
IF_{i,T}(\psi,P_{\theta_0}) \\
\vdots \\
IF'_{j,T}(\psi ,P_{\theta_0}) & \mathbb{E}\left[\sum_{t=1}^{T}  \frac{\partial^2 \psi_{i,t,d}(\theta_0)}{\partial\theta \partial\theta'}
\Big\vert_{\theta=\theta_0}  \right
] &
IF_{i,T}(\psi,P_{\theta_0})

\end{array}
\right)\ , 
\label{gammader}
\end{equation}
and $M_{i,T}(\psi, P_{\theta _{0}})$ is defined by (\ref{Eq: Mit}).

\end{lemma}

In (\ref{IIorder_txt}), we interpret the quantities $ IF_{i,T}(\psi, P_{\theta _{0}}) $, the first-order von Mises kernel, and $ \varphi_{i,j,T}(\psi, P_{\theta _{0}})$, the second-order von Mises kernel, as functional derivatives  of the $M$-functional related to the MLE; 
see \citet{Fernholz2001}. Specifically, the first term, of order $m^{-1} \propto n^{-1}$, is the Influence Function (IF) and represents the standard tool applied to derive the first-order (Gaussian) asymptotic theory of the MLE; see, e.g., \citet{vdW98} and \citet{BaltagiBook} for a book-length introduction. The second term in (\ref{IIorder_txt}), of order $m^{-2} \propto n^{-2}$, plays a pivotal role in our derivation of higher-order approximation. 

\vspace{-0.5cm} 
 
\subsection{Approximation via $U$-statistic} \label{Sec: Ustat}

The result of Lemma \ref{Lemma_IIvM} together with the chain rule define a second-order asymptotic expansion for a real-valued function of the MLE, 
such as a component of $\vartheta(P_{n,T})$ or a linear contrast. In Lemma \ref{Lemma_IIvM_q}, we show that
we can write the asymptotic expansion in terms of a $U$-statistic of order two. To this end, we introduce the following assumption.

\textbf{Assumption C.} \\
{\it Let
$q$ be a function from $\mathbb{R}^{d}$ to $\mathbb{R}$, which has continuous and nonzero 
gradient at $\theta = \theta_0$ and continuous second derivative at $\theta = \theta_0$. }

Then, we have 

\begin{lemma} \label{Lemma_IIvM_q}

Under Assumptions A-C, the following expansion holds:
$$
q[\vartheta(P_{n,T})] - q[\vartheta(P_{\theta_0})] = \frac{2}{n(n-1)}
\sum_{i=1}^{n-1} \sum_{j=i+1}^{n} h_{i,j,T}\left(\psi,P_{\theta_0} \right) +  O_P(m^{-3/2}),
$$
where
\begin{eqnarray} \label{Eq: h_ijT}
 h_{i,j,T}\left(\psi,P_{\theta_0} \right) &=& g_{i,T}\left(\psi,P_{\theta_0} \right)+ g_{j,T}\left(\psi,P_{\theta_0} \right)+\gamma_{i,j,T}(\psi,P_{\theta_0}) \nn \\
 &=& \frac{1}{2} \left\{ IF'_{i,T}(\psi,P_{\theta_0}) +IF'_{j,T}(\psi,P_{\theta_0}) + \varphi'_{i,j, T}(\psi,P_{\theta_0})\right\}\frac{\partial q(\vartheta)}{\partial\vartheta}\Big\vert
_{\theta=\theta_0}   \nn \\
&+& \frac{1}{2} IF'_{i,T}(\psi,P_{\theta_0}) \frac{\partial^2 q(\vartheta)}{\partial\vartheta\partial\vartheta'}\Big\vert
_{\theta=\theta_0} IF_{j,T}(\psi,P_{\theta_0})   ,
\end{eqnarray}
with
\begin{equation} \label{Eq: g_iT}
g_{i,T}(\psi,P_{\theta_0}) = \frac{1}{2} \left( IF'_{i,T}(\psi,P_{\theta_0}) \frac{\partial q(\vartheta)}{\partial\vartheta}\Big\vert_{\theta=\theta_0} \right),
\end{equation}
\begin{equation}
 \gamma_{i,j,T}(\psi,P_{\theta_0})=\frac{1}{2} 
\left( \varphi'_{i,j, T}(\psi,P_{\theta_0})\frac{\partial q(\vartheta)}{\partial\vartheta}\Big\vert
_{\theta=\theta_0}   + IF'_{i,T}(\psi,P_{\theta_0}) \frac{\partial^2 q(\vartheta)}{\partial\vartheta\partial\vartheta'}\Big\vert
_{\theta=\theta_0} IF_{j,T}(\psi,P_{\theta_0})  \right). \label{Eq: gamma_ijT}
\end{equation}

\end{lemma}

The function $q$ may select, e.g., a single component of the vector $\theta_0$. In many empirical applications, the most interesting parameter is  the spatial correlation coefficient $\lambda_0$, and the null hypothesis is zero correlation versus the alternative hypothesis of positive spatial correlation---the aim being to check whether there is a contagion effect. 

\vspace{-0.3cm}

\subsection{Higher-order asymptotics} \label{Sec: sadd}

Making use of Lemma \ref{Lemma_IIvM} and Lemma \ref{Lemma_IIvM_q},  we derive
 the Edgeworth and the saddlepoint approximation 
to the distribution of a real-valued function $q$ of the MLE.

Let $f_{n,T}(z)$ be the true density of $q[\vartheta(P_{n,T})] - q[\vartheta(P_{\theta_0})]$ at the point 
$z \in \mathcal{A}$, where $\mathcal{A}$ is a compact subset of $\mathbb{R}^d$. 
%
Our derivation of the saddlepoint density approximation to $f_{n,T}(z)$ is based on the 
tilted-Edgeworth expansion for $U$-statistics of order two. With this regard, a remark is in order.
From (\ref{Eq: s_it}), we see that
the terms in the random vector $\psi_{nt}(\theta_0)$ depend on the rows of the weight matrix $W_n(\rho)$ and $M_n(\lambda)$. As
a consequence, these terms are independent but not identically distributed random variables, and we need to derive 
the Edgeworth expansion for our $U$-statistic taking into account this  aspect. 
To this end, we approximate
the cumulant generating function (c.g.f.) of our $U$-statistic by 
\textit{summing (in $i$ and $j$) the (approximate) c.g.f.} of each $h_{i,j,T}$ kernel.
This is an extension of the derivation by 
\citet{BGVZ86} for i.i.d. random variables.
To elaborate further, we introduce   



\textbf{Assumption D.} \\ {\it 
Suppose that there exist positive numbers $\delta$, $\delta_1$, $C$ and positive and continuous functions $\chi_{j}$: $(0,\infty)\to (0,\infty)$, $j=1,2,$ satisfying $\lim_{z\to\infty} \chi_1(z)=0$, $\lim_{z\to\infty} \chi_2(z)\geq \delta_1 >0$, and a real number $\alpha$ such that $\alpha \geq 2+\delta>2$, 
\begin{enumerate}[(i)]
\item $ \mathbb{E}\left[\vert \gamma_{i,j,T}(\psi,P_{\theta_0})\vert^{\alpha}\right]<C$ for any $i$ and $j$, $1\leq i< j\leq n$, 
\item $\mathbb{E}\left[ g_{i,T}(\psi,P_{\theta_0})^{4}1_{[z,\infty)}(\vert g_{i,T}(\psi,P_{\theta_0})\vert)\right]<\chi_1(z)$ for all $z>0$ and any $i$, $1\leq i\leq n$,
\item $\Big\vert\mathbb{E}\left[e^{\iota \nu g_{i,T}(\psi,P_{\theta_0})}\right]\Big\vert\leq 1-\chi_2(z) <1$ for all $z>0$ and any $i$, $1\leq i\leq n$ and $\iota^2=-1$,
\item $\vert \vert M_{i,T}(\psi,P_{\theta _{0}})- M_{j,T}(\psi,P_{\theta _{0}})  \vert \vert=O(n^{-1})$ uniformly in $\lambda$ and $\rho$.
\end{enumerate} } 

A few comments are in order. Assumptions D$(i)$-$(iii)$ are similar to the technical assumptions in \citet{BGVZ86} p.
1465 and 1477. However, there are some differences between our assumptions and theirs.  Indeed, to
take into account the non identical distribution of $\psi_{i,t}$ and $\psi_{j,t}$, for $i \neq j$,  
we consider the first- and second-order von Mises kernels for each $i$ (as in D$(i)$-$(iii)$). 
It is different from \citet{BGVZ86}: compare, e.g., our D$(ii)$ to  their Eq. (1.17).  
D$(iv)$ is not considered in \citet{BGVZ86}: it is a peculiar assumption needed for our higher-order asymptotics (the technical aspects are available in 
Lemma A.1 and its proof in Appendix A). {In Appendix D.2,   
we illustrate that, in the case of the SAR(1) model, the validity of D($iv$) is related to more primitive expressions involving the entries of (some powers of) $W_n$.} For other models, one should derive such  primitive expressions on a case-by-case basis. For the sake of generality, here we provide an intuition on D($iv$). Let us consider two different locations $i$ and $j$.
From (\ref{Eq: Mit}), we see that D$(iv)$ 
imposes a structure on the information available at different locations. Indeed, $ M_{i,T}(\psi,P_{\theta _{0}})$ and $ M_{j,T}(\psi,P_{\theta _{0}})$ contribute to the asymptotic variance of the MLE. Since $ M_{i,T}(\psi,P_{\theta _{0}}) $ is related to the information available at the $i$-th location, D$(iv)$ essentially assumes that 
there exists an informative content which is common to location $i$ and $j$, whilst the ({Frobenious} norm of the) information content specific to each location is of order $O(n^{-1})$.
\color{black}

\begin{proposition} \label{Prop_EDG}
Under Assumptions A-D, the Edgeworth expansion $\Lambda_m(z)$ for the c.d.f. $F_m$ of $\sigma_{n,T}^{-1}\{q[\vartheta(P_{n,T})] - q[\vartheta(P_{\theta_0})]\}$ is
\begin{eqnarray}
\Lambda_m(z) &=& \Phi(z) - \phi(z)\left\{n^{-1/2} \frac{\kappa_{n,T}^{(3)}}{3!} (z^2-1) + n^{-1} \frac{\kappa_{n,T}^{(4)}}{4!} (z^3-3z) + n^{-1} \frac{\kappa_{n,T}^{(3)}}{72} (z^5-10z^2+15z) \right\} \nn \\
\label{Eq: EDG_CDF}
\end{eqnarray}
where $z \in \mathcal{A}$, $\sigma_{n,T}$ is the standard deviation of $q[\vartheta(P_{n,T})] - q[\vartheta(P_{\theta_0})]$, $\Phi(z)$ and $\phi(z)$ are the c.d.f. and p.d.f of a standard normal r.v. respectively, $\kappa_{n,T}^{(3)}n^{-1/2}$ and $\kappa_{n,T}^{(4)}n^{-1}$ are the approximate third and fourth cumulants of $\sigma_{n,T}^{-1}\{q[\vartheta(P_{n,T})] - q[\vartheta(P_{\theta_0})]\}$, as defined in (A.15) and (A.18), respectively.
Then 
$\sup_z|F_m(z)-\Lambda_m(z)| =o(m^{-1}).$ 
\end{proposition}

In addition, we can get the saddlepoint density approximation by exponentially tilting the Edgeworth expansion. 

\begin{proposition} \label{Prop_SAD}
Under Assumption A-D, the saddlepoint density approximation to the density of $q[\vartheta(P_{n,T})]-q[\vartheta(P_{\theta_0})]$ at the point $z \in \mathcal{A}$ is 
\begin{equation}
p_{n,T}(z)=\left[ \frac{n}{2\pi \tilde{\mathcal{K}}_{n,T}^{\prime \prime }(\nu)}%
\right] ^{1/2}\exp \left\{ n \left[\tilde{\mathcal{K}}_{n,T}(\nu)-\nu z\right]\right\},
\label{Eq: SAD_panel}
\end{equation} 
with relative error of order $O(m^{-1})$,  $\nu:=\nu(z)$ is	 the saddlepoint defined by 
\begin{equation}
\tilde{\mathcal{K}}_{n,T}^{\prime}(\nu)=z,
\label{Eq. sadd}
\end{equation}
the function $\tilde{\mathcal{K}}_{n,T}$ is the approximate c.g.f. of
$\sqrt{n}(q[\vartheta(P_{n,T})]- q[\vartheta(P_{\theta_0})]) $, as defined in (A.42), 
while $\tilde{\mathcal{K}}_{n,T}^{\prime}$ and $\tilde{\mathcal{K}}_{n,T}^{\prime \prime}$ represent the first and second derivative of $\tilde{\mathcal{K}}_{n,T}$, respectively. Moreover, 
\begin{eqnarray}
P\left\{ q[\vartheta(P_{n,T})] - q[\vartheta(P_{\theta_0})] > z \right\} & = & \left[ 1- \Phi(r) + \phi(r) \left( \frac{1}{c}- \frac{1}{r} \right) \right] \left[1+O(m^{-1})\right], 
\label{LR_pvalue}
\end{eqnarray}

$$c=\nu \left[ \tilde{\mathcal{K}}_{n,T}^{\prime \prime}(\nu )\right] ^{1/2} \quad \text{and} \quad
r=\text{sgn}(\nu )\left\{ 2n\left[ \nu z-\tilde{\mathcal{K}}_{n,T}(\nu )\right] \right\}^{1/2}.$$
\end{proposition}
The proofs of these Propositions are available in Appendix A. 
They rely on a argument similar to the one applied in the proof of \citet{F82} for the derivation of a saddlepoint 
density approximation of  multivariate $M$-estimators,
and in \citet{GR96} for general statistics.  Following \citet{D80}, we can further normalize  
$p_{n,T}$ to obtain a proper density by dividing the right hand side of (\ref{Eq: SAD_panel}) by its integral with respect to $z$. This normalization typically improves 
even further the accuracy of the approximation. 


\subsection{Links with the econometric literature} 
\label{Sec: Links_econometrics} 

The expansions in Proposition \ref{Prop_EDG} and Proposition \ref{Prop_SAD} are connected with  the results on higher-order expansions available in the spatial econometric literature, as cited in \S \ref{Sec: Intro}. However, some key differences are worth a mention.

(i) The Edgeworth expansion in \citet{RR15} is for the concentrated MLE of $\lambda$  and it is based on a higher-order Taylor expansion of the concentrated likelihood score; see also \citet{RR14_EJ,RR14_ET, RR15}, \citet{HM18}, and \citet{HM20}. In contrast, our method is based on a von Mises expansion of the MLE functional of the whole model parameter and we resort on a marginalization procedure to obtain the saddlepoint density approximation of the parameter(s) of interest.  
Therefore, Lemma \ref{Lemma_IIvM} and Lemma \ref{Lemma_IIvM_q} give generality and flexibility to our approach: not only we may focus on $\lambda$, but also, e.g., on $\rho$ (which contains information on the spatial dependence of the innovation terms) and/or on $\beta$ (which convey information on the significance of the time-varying covariates). 

(ii) Our saddlepoint approximation is more general than the Edgeworth-based approximations available in the econometric literature, since we work with a larger class of models, which includes the model in \citet{RR15} as a special case. 

(iii) Although the inference (e.g., testing) 
derived using the Edgeworth expansion improves on the standard first-order asymptotics,
it is well-known (see, e.g., \citet{FR90}) that, in general, this technique provides a 
good approximation in the center of the distribution, but can be inaccurate in the tails, where the Edgeworth expansion 
can even become negative. It can lead to inaccurate approximations. Our saddlepoint approximation
is a density-like object and is always nonnegative. 

(iv)  Our saddlepoint approximation yields a tail-area approximation via a Lugannani-Rice type formula. A similar result is not available for the Edgeworth expansion of the concentrated MLE derived in \citet{RR15}. Recently, \citet{HM20} studied the adjusted profile likelihood estimation method and obtained a result similar to our tail-area approximation. Their formula is derived for the spatial autoregressive model with covariates. However,  they do not prove the  higher-order properties of their approximation. 
In Proposition \ref{Prop_EDG}, we prove that our saddlepoint density approximation 
features \textit{relative error} of order $O(1/(n(T-1)))$. This has to be contrasted with the extant Edgeworth expansion, which entails an absolute error of lower order---more precisely, the error order is $o((nT)^{-1/2})$, when the entries of the spatial matrix are $O(1)$; see Eq. (2.15) in \citet{RR15}.  Achieving a small relative error is appealing in 
tail areas where the probabilities are small. 

(v) In the comparison with the bootstrap,  our methodology 
does not need resampling.  Moreover, it does neither require
bias correction, nor any studentization.

\section{Computational aspects} \label{Sec: algm}

Most of the quantities related to the saddlepoint density approximation $p_{n,T}$ and the tail area in (\ref{LR_pvalue}) are available in closed-form. 
In Appendix C.1, we provide an algorithm (see Algorithm 1) in which we itemize the main computational steps needed to implement the saddlepoint tail area approximation, 
for a given transformation $q$ 
and for a given reference parameter $\theta_0$---it is, e.g., the parameter characterizing the null hypothesis in a simple hypothesis testing, where the tail area
 probability is an approximate $p$-value.  

\vspace{-0.5cm}

\section{Comparisons with other approximations and testing in the presence of nuisance parameters} \label{Sec: MC}


\vspace{-0.25cm}

We compare the performance of our saddlepoint approximations to other routinely-applied asymptotic techniques. To start with, we consider the SAR(1) model, where $\lambda$ is the only unknown parameter. Then, we move to the SARAR(1,1) model, where we illustrate how to take care of nuisance parameters.  We use the same setting as in \S \ref{Sec: Motivation}; we refer to the online Supplementary Material (Appendix D) for more details and for additional results.

\subsection{Comparisons with other asymptotic techniques} \label{Sec: AsyvsSad}

\textit{Saddlepoint vs first-order asymptotics.} For the SAR(1) model, we analyse the behaviour of the MLE of $\lambda_0$,
whose PP-plots are available in Figure \ref{Fig: PP}. For each type of
$W_n$, for $n=24$ and $n=100$, the plots show that the saddlepoint approximation is closer to the ``exact'' probability than the first-order asymptotics approximation. 
For $W_n$ Rook, the saddlepoint approximation improves on the routinely-applied first-order asymptotics. In Figure 
\ref{Fig: PP}, the accuracy gains are evident also for $W_n$ Queen and  Queen with torus, where the first-order asymptotic theory displays large errors essentially over the whole support (specially in the tails). On the contrary, 
the saddlepoint approximation is  close to the 45-degree line. 


\begin{figure}[hbt!]
\begin{center}
\begin{tabular}{ccc}
 Rook & Queen  & Queen torus   \\
\begin{turn}
{90} \hspace{2.5cm}  $n = 24$ 
\end{turn}
\includegraphics[width=0.3\textwidth, height=0.265\textheight]{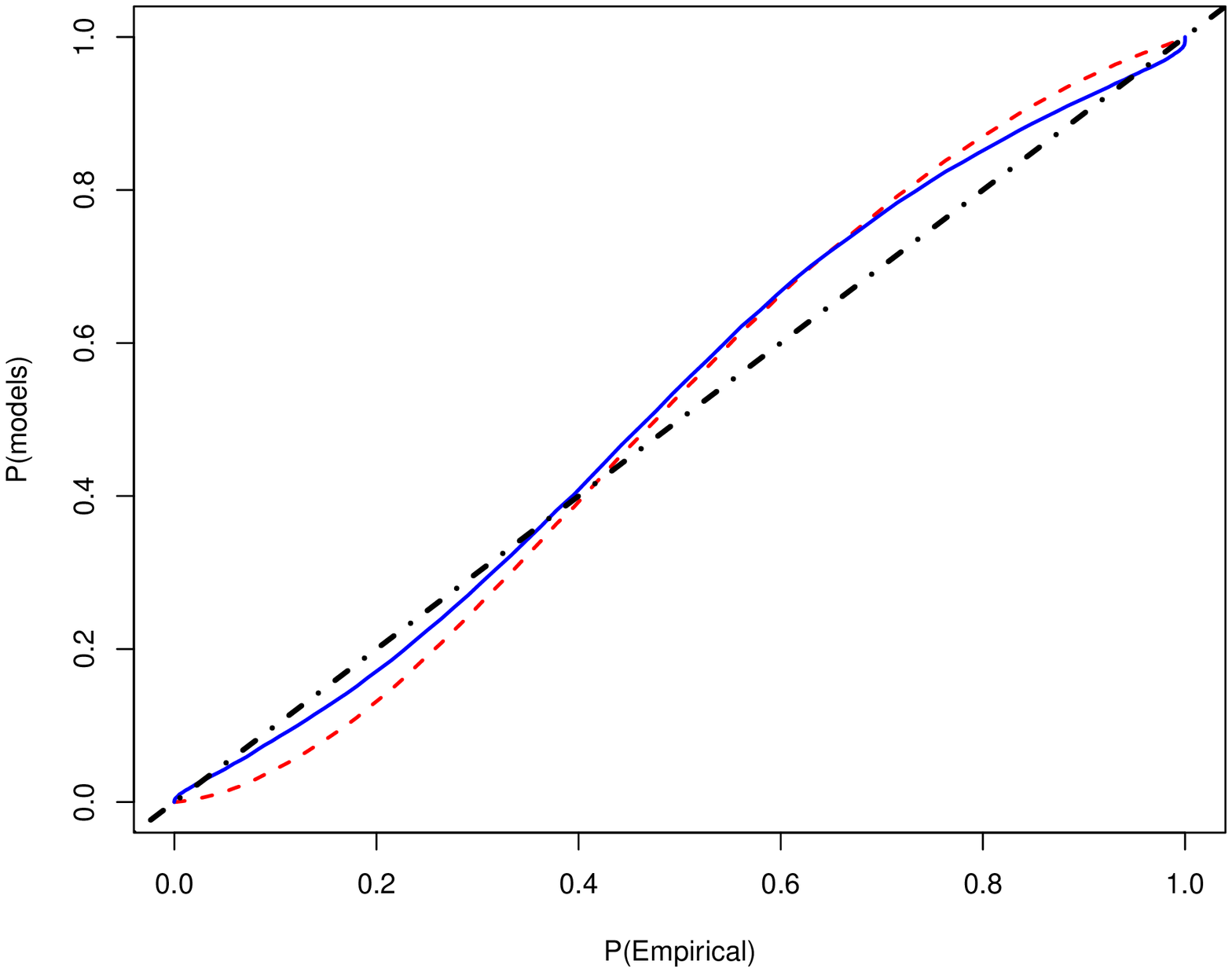} &
\includegraphics[width=0.3\textwidth, height=0.265\textheight]{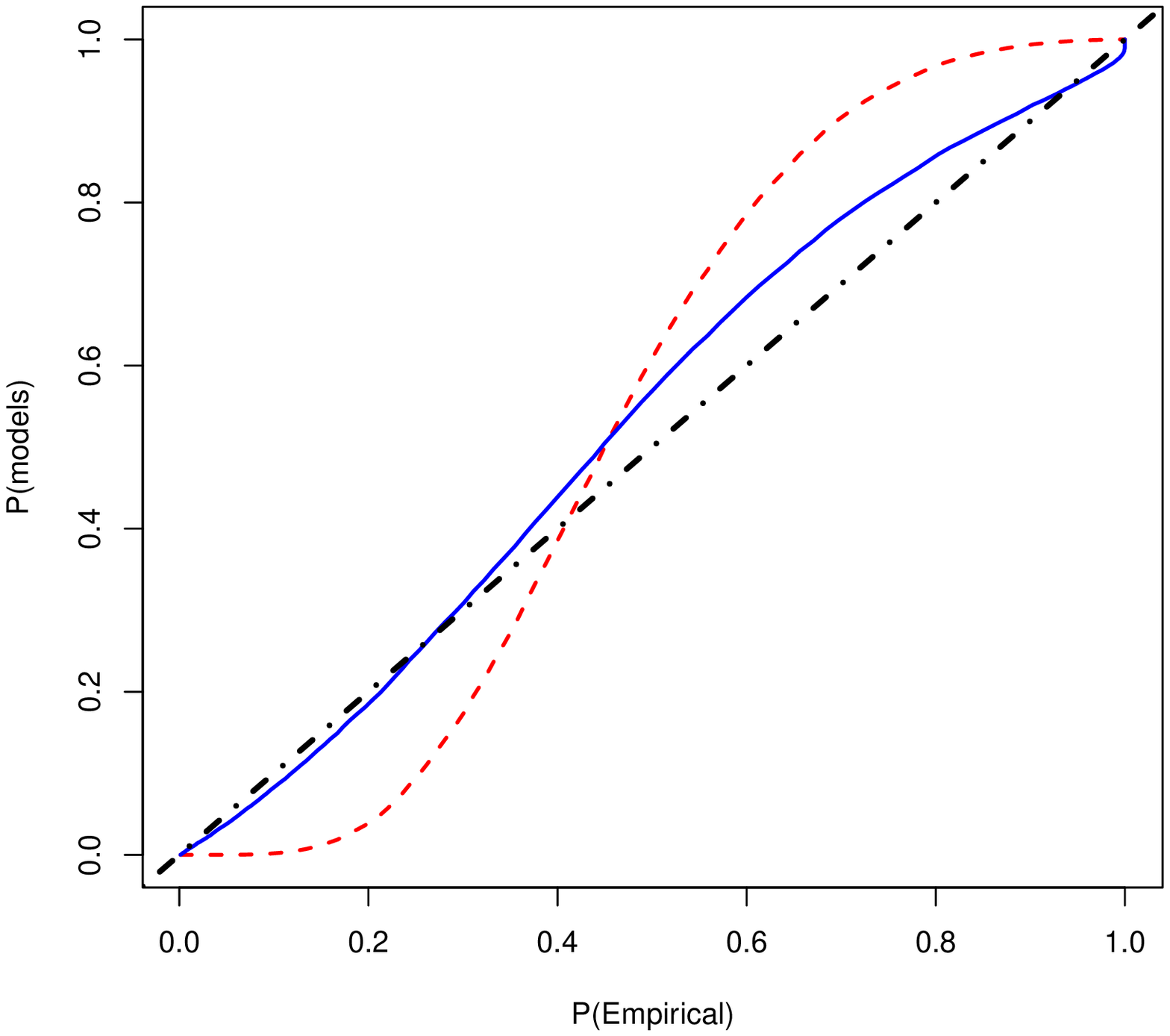} &
\includegraphics[width=0.3\textwidth, height=0.265\textheight]{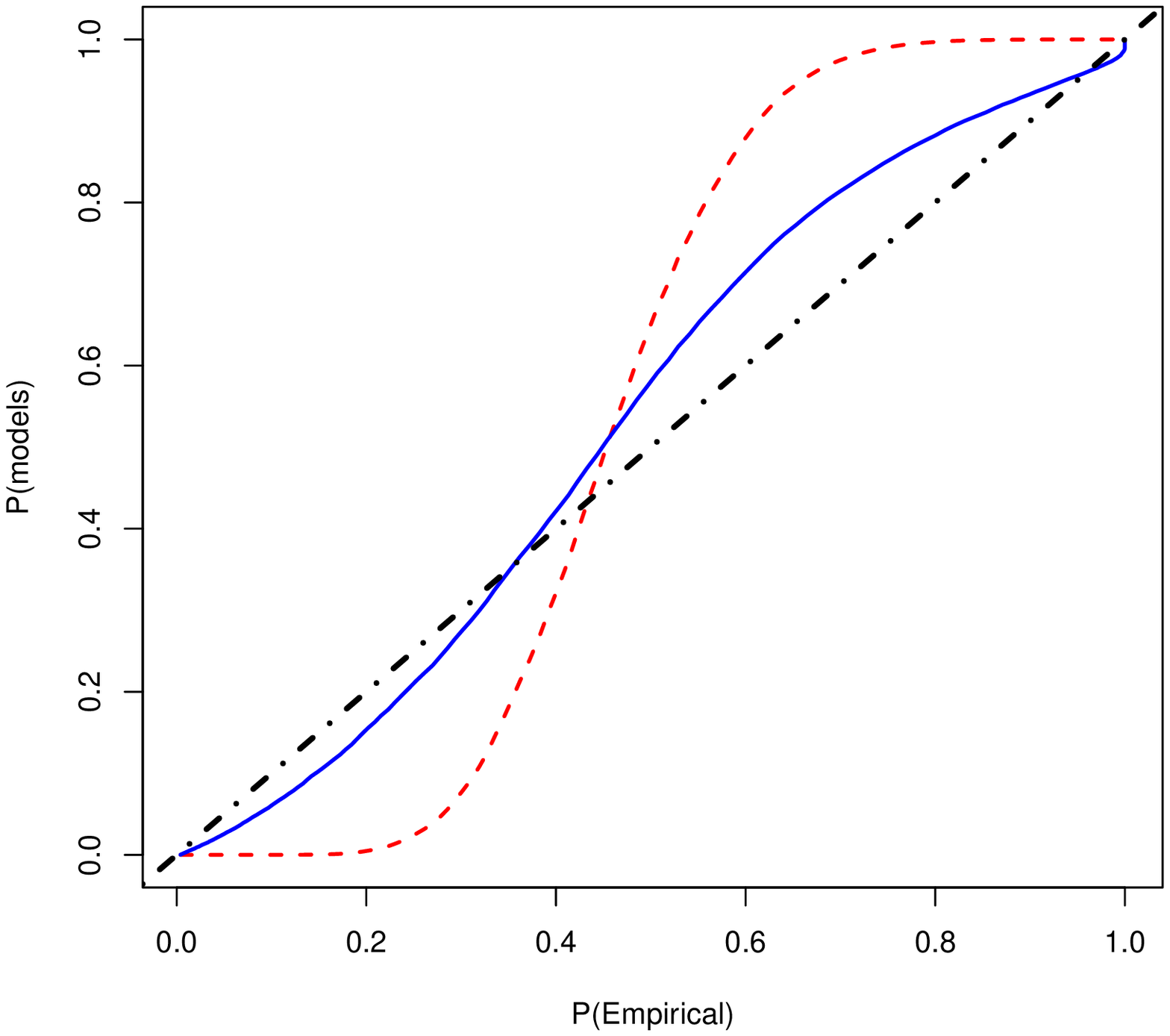} 
\\
\end{tabular}
\caption{SAR(1) model: PP-plots for saddlepoint (continuous line) vs asymptotic normal (dotted line) probability approximation, for the MLE $\hat\lambda$, for $n=24$ and $n=100$, $\lambda_0=0.2$, and different $W_n$.}
    \label{Fig: PP}
\end{center}
\end{figure}


\textit{Saddlepoint vs Edgeworth expansion (testing simple hypotheses).} The Edgeworth expansion derived in Proposition \ref{Prop_EDG} represents the natural 
alternative to the saddlepoint approximation since it is fully analytic.
To gain insights into the different behavior of the saddlepoint and Edgeworth approximations, 
we investigate the size of a hypothesis test based on the  approximations. 
We set $n=24$ and we assume that $\sigma^2$ is known and equal to one.
We consider the simple null hypothesis $H_0$: $\lambda_0=0$ for a one-sided test of zero against positive values of  
spatial correlation. We use 25,000 replications of  $\hat\lambda_{n,T}$ to get the empirical estimate $\hat F_0$ of the c.d.f. $F_0$ of the estimator under the null hypothesis. We use the generic notation $G$ for the c.d.f. of one of the Edgeworth, or saddlepoint approximations, under the null hypothesis. For the sake of completeness, we also display the results for the Gaussian (first-order) approximation. The empirical rejection probabilities $\hat \alpha = 1-\hat F_0(G^{-1}(1-\alpha))$ are shown in Figure \ref{Fig: Size} for nominal size $\alpha$ ranging from 1\% to 10\%, and correspond to an estimated size. We have overrejection when we are above the 45-degree line. We observe strong size distortions for the asymptotic and Edgeworth approximations as expected from the previous results.
The saddlepoint approximation exhibits only mild size distortions. For example,  we get an estimated size $\hat \alpha$ of 11.72\%, 7.36\%, 5.70\%, for the Normal, Edgeworth, and saddlepoint approximations, for a nominal size of 5\%.

\begin{figure}[htbp]
\begin{center}
\includegraphics[width=0.5\textwidth, height=0.35\textheight]{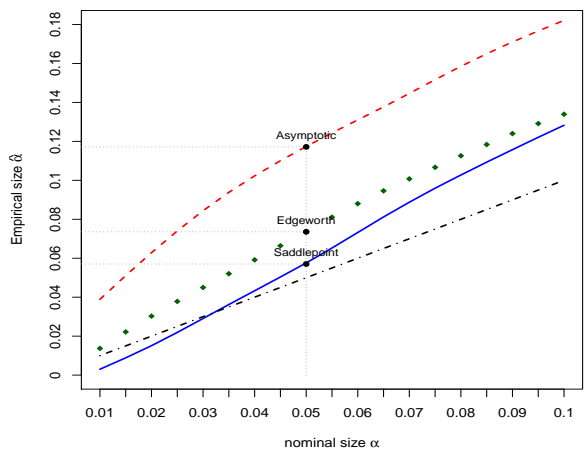}

\caption{SAR(1) model: Estimated $\hat \alpha$ versus nominal size $\alpha$ between 1\% and 10\% under saddlepoint (continuous line), Edgeworth (dotted line with diamonds) and first-order asymptotic approximation (dotted line). $W_n$ is Rook, $n=24$ and $\lambda_0=0.0$.}
    \label{Fig: Size}

\end{center}
\end{figure}

\textit{Saddlepoint vs parametric bootstrap.} 
The parametric bootstrap represents a (computer-based) competitor, commonly applied in statistics and econometrics. 
To compare our saddlepoint approximation to the one obtained by bootstrap, we consider different numbers of bootstrap repetitions, labeled as $B$: we use $B=499$ and $B=999$. For space constraints, in Figure \ref{Fig: FB}, we display the results for $B=499$ (similar plots are available for $B=999$) showing the functional boxplots (as obtained iterating the procedure 100 times) of the bootstrap approximated density, for sample 
size $n =24$ and for $W_n$ is Queen. 
\begin{figure}[hbt!]
\begin{center}
\begin{tabular}{cc} 
\includegraphics[width=0.485\textwidth, height=0.35\textheight]{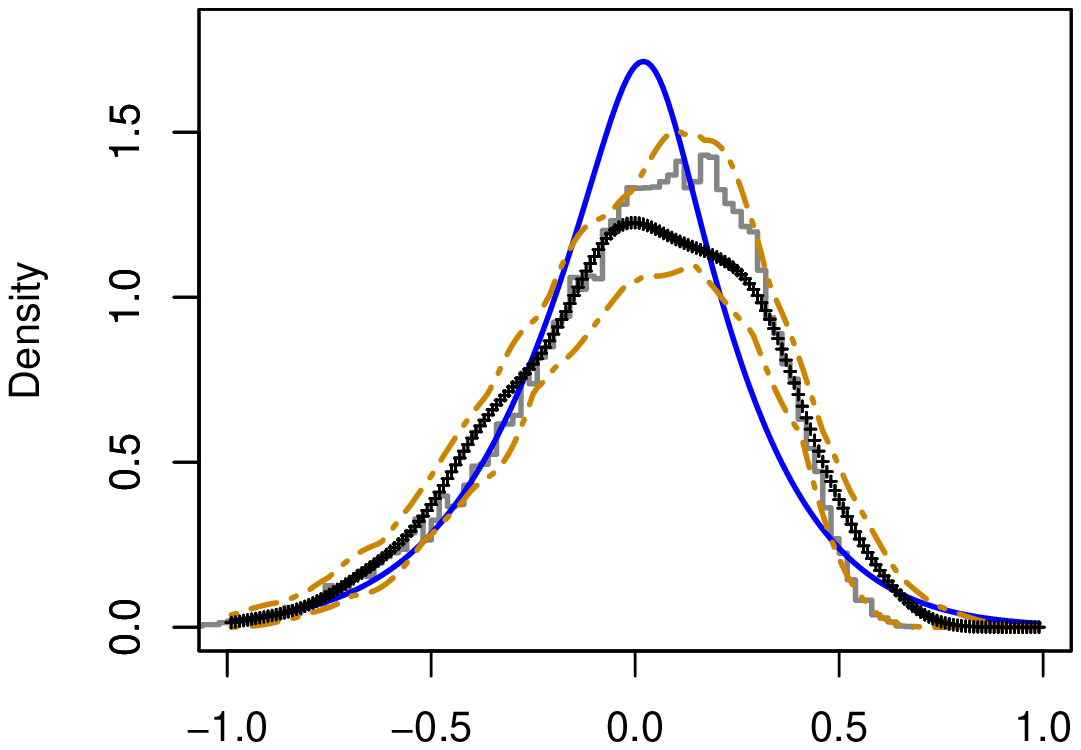} &
\includegraphics[width=0.485\textwidth, height=0.35\textheight]{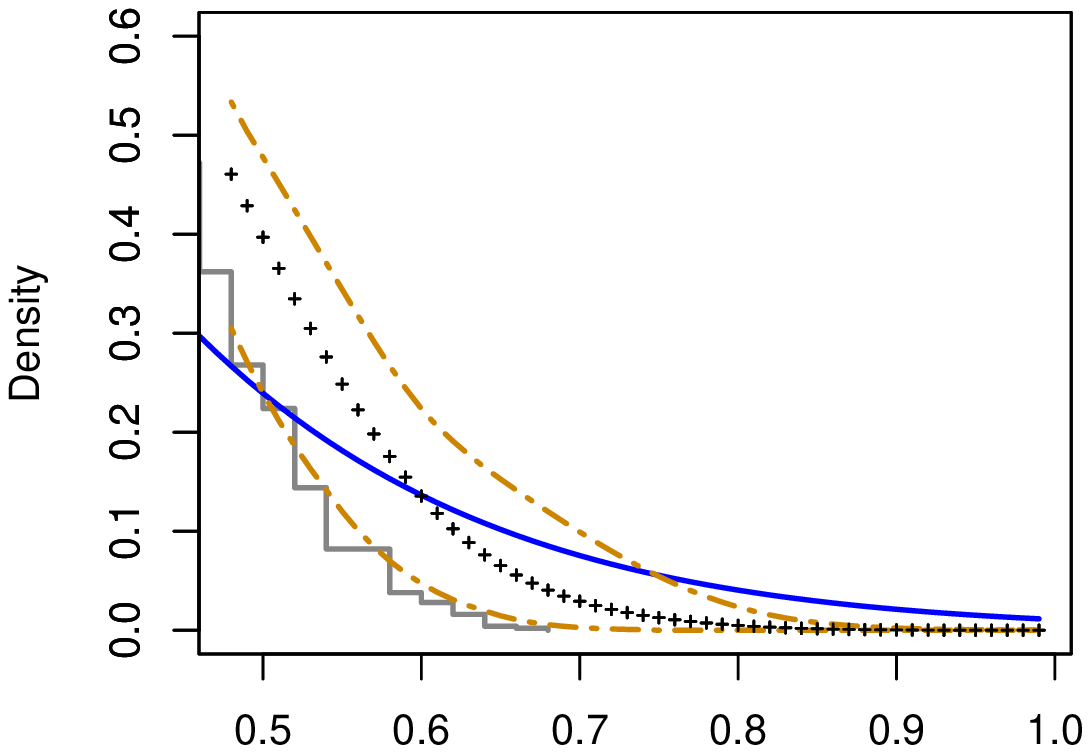}
\end{tabular}
\caption{SAR(1) model. Left panel: Density plots for saddlepoint (continuous line) vs the functional boxplot of the parametric bootstrap probability approximation to the exact density (as expressed by the histogram and obtained using 
MC with size 25000), for the MLE $\hat\lambda$ and $W_n$ is Queen. Sample size is $n=24$, while $\lambda_0=0.2$. Right panel: zoom on the right tail. In each plot, we display the functional central curve (dotted line with crosses), the $1_{st}$ and $3_{rd}$ functional quartile (two-dash lines). 
}
    \label{Fig: FB}

\end{center}
\end{figure}
To visualize the variability entailed by the bootstrap, we display the first and third quartile curves (two-dash lines) and the median functional curve (dotted line with crosses); for details about functional boxplots, we refer to \citet{Sun11} and to R routine \texttt{fbplot}. We notice that, while the bootstrap median functional curve (representing a typical bootstrap density approximation) is close to the actual density (as represented by the histogram), the range between the quartile curves illustrates that the bootstrap approximation has a variability.
Clearly, the variability depends on $B$: the larger is $B$, the smaller is the variability. However, larger values of $B$ entail bigger computational costs: when $B=499$, the bootstrap is almost as fast as the saddlepoint density approximation ({computation time  
about 7 minutes,  on a 2.3 GHz Intel Core i5 processor}), 
but for $B=999$, it is three times slower. We refer to Appendix D.5 for additional numerical results.

\subsection{Testing in the presence of nuisance parameters}
\label{Sec: test_nuisance_parameters}

\subsubsection{Saddlepoint test for composite hypotheses} \label{Sec: testcomp}

Our saddlepoint density and/or tail approximations are helpful for testing simple hypotheses about $\theta_0$; see \S \ref{Sec: AsyvsSad}. 
Another interesting case suggested by the Associate Editor and an anonymous referee that has a strong practical relevance is related to testing a
composite null hypothesis. It is a problem which is different from the one considered so far in the paper, because it raises the issue of dealing with nuisance parameters. 

To tackle this problem, several possibilities are available. 
For instance, we may 
fix the nuisance  parameters at the MLE estimates. Alternatively, we may consider to use
the (re-centered) profile estimators, as suggested, e.g., in \citet{HM18} and \citet{HM20}.
Combined with the saddlepoint density in (\ref{Eq: SAD_panel}), these techniques  yield a ready solution to the nuisance parameter problem. In our numerical experience (see  Appendix D.6 for an experiment about the SAR(1)), these solutions may preserve 
reasonable accuracy in some cases. 
Nevertheless, the main theoretical drawback related
to the use of MLE values for the nuisance parameter(s) is that it would not guarantee that
the second-order properties derived in the previous sections still hold.
To cope with this issue, we propose to build on 
\citet{RRY03}, who  derive a saddlepoint test statistic which 
takes into account explicitly the nuisance parameters, while preserving relative error in normal region. We feel this test statistic represents the natural candidate within our setting: it shares the same spirit  as our saddlepoint density approximation and it is derived going through steps which are similar to ours. 
The paper by \citet{RRY03} defines the test statistic in the i.i.d.\ setting, while \citet{Lo09} and \citet{CR10} extend it to the non-i.i.d.\ data setting. 

Let us consider a SARAR model whose parameter is
$\theta = (\theta_{10}, \theta_2)'$, 
where $\theta_{10}$ is specified by the null composite hypothesis: typically, the null concerns $\lambda$ only, while $\theta_2$ contains all the nuisance parameters. More specifically, the parameter is
$\theta = (\lambda, \beta, \rho, \sigma^2)'$ and the general function
$q(\theta)$ used in the previous sections is simply 
$q(\theta) = \lambda$. Thus, we have the composite hypothesis:
\begin{equation}
\mathcal{H}_0: \lambda=\lambda_0 = 0 \quad \text{vs} \quad \mathcal{H}_1: \lambda > 0, \label{TestComp}
\end{equation}
where $\theta= (\lambda,\theta_2)'$, with $\theta_{10}=\lambda_0$ and $\theta_2=(\beta, \rho, \sigma^2)'$. Then, we define the test statistic

\begin{equation}
{SAD}_n(\hat{\lambda}) = 2n \ 
\underset{\theta_2}{\inf} \ \underset{\nu}{\sup} \ -\mathcal{K}_\psi(\nu,\hat\lambda,\theta_2). \label{Eq. h}
\end{equation}
%
The function $\mathcal{K}_\psi(\nu,(\lambda,\theta_2))$ is the  c.g.f. of the estimating function: 
\begin{equation}
\mathcal{K}_\psi(\nu,\hat\lambda,\theta_2) = n^{-1}\sum_{i = 1}^{n}\ln E_{P_{(\lambda_{0},\theta_2)}} \exp(\nu^{T}\psi_{i}^{(T)}(\hat{\lambda},\theta_2)), \label{Eq. cgfpsi}
\end{equation}
where $\psi_{i}^{(T)}(\lambda,\theta_2):=\sum_{t=1}^{T}(T-1)^{-1}\psi_{i,t}(\lambda,\theta_2)$ and  $\psi_{i,t}$ is as in (4.1). The c.g.f. $\mathcal{K}_\psi$ has a role analogous to the one of the c.g.f. of the $U$-statistic, that we derived in \S \ref{Sec: Methodology}.  We highlight that the expected value in (\ref{Eq. cgfpsi}) is taken w.r.t. the probability $P_{(\lambda_{0},\theta_2)}$, where $\lambda_0$ is specified by the null, while the nuisance parameters are not fixed: the infimum over $\theta_2$ takes care of the nuisance parameters. 
In our inference procedure, we have that
$\hat{\theta}_{n,T}=(\hat\lambda,\hat\theta_2)'$ is the solution to
$
\sum_{i=1}^{n} \psi_{i}^{(T)}(\lambda,\theta_2) = 0.
$
 Under the null hypothesis, the test statistic ${SAD}_n(\hat{\lambda}) $  is  asymptotically $\chi_1^2$ distributed with a relative error of order $O(m^{-1})$ in the normal region. 

\subsubsection{Implementation aspects}

To implement the test (\ref{Eq. h}) for the problem (\ref{TestComp}), in Appendix C.2, we propose an algorithm (see Algorithm 2) and itemize the main steps needed to compute the test statistic. 

\subsubsection{Numerical results}
 
Let us work with a SARAR(1,1) model, having no covariates and known variance $\sigma^2=1$ and $n=24$. It implies that 
 $\theta=(\lambda,\rho)'$ and we consider the problem in (\ref{TestComp}), with $\rho$ being the nuisance
 parameter. We set three different values $\rho=0.25, 0.5, 0.75$ to analyze numerically the impact that
 the spatial dependence in the innovation term has on $SAD_n$. We study the behaviour of the Wald
 test, as obtained using the first-order asymptotic theory and making use of the expression
 of the asymptotic variance as available in Appendix B. We compare the Wald test to $SAD_n$--to implement (\ref{Eq. h}) we make use of the R routine \texttt{nlm}. We consider
 two types of spatial matrix $W_n$, the Rook and the Queen, and we set $W_n\equiv  M_n$. Both test statistics are asymptotically $\chi_1^2$ distributed
under the null hypothesis. To compare them in small samples, we first obtain the 95th and 97.5th quantile of each test statistic; then we compute the corresponding probability as obtained using the $\chi_1^2$. We display the results in Table \ref{TabNuisance}.  We see that the Wald test has severe size distortion. For instance, for $\rho=0.25$, we observe a relative error of about $30\%$, for the quantile of $95\%$, when  $W_n$ is Rook, while the saddlepoint test entails a relative error of about $1.8\%$. Looking at the performance of ${SAD}_n$, we see that it is uniformly more accurate than the Wald test: considering all cases, we observe a maximal relative error of about $2\%$, for the quantile of $95\%$, when  $\rho=0.75$ and $W_n$ is Queen; in the same setting, the Wald test entails a relative error of about $24\%$. Moreover, the size is fairly constant for the different values of $\rho$: it illustrates that the test statistic takes care correctly of the nuisance parameter.

 \begin{table}[h!]
\begin{center}
\begin{tabular}{ccccccccccccccccccccccc}
\toprule
& & \multicolumn{3}{l}{$\rho=0.25$} & &  \multicolumn{3}{l}{$\rho=0.5$} & &  \multicolumn{3}{l}{$\rho=0.75$}  \\
& & $95.00\%$ & $97.50\%$ & & & $95.00\%$ & $97.50\%$ & & & $95.00\%$ & $97.50\%$\\
 \cmidrule{3-5} \cmidrule{7-9} \cmidrule{11-13}\\
Rook & \\
& \text{Wald} & $66.08\%$ & $89.89\%$ && &  $98.33\%$  & $99.41\%$ && & $99.99\%$ & $99.99\%$ && \\
& ${SAD}_n$   & $96.71\%$ & $97.18\%$ && &  $96.66\%$  & $97.18\%$ && & $95.55\%$ & $96.04\%$  &&\\
Queen & \\
& \text{Wald} & $72.71\%$ & $80.20\%$ && & $90.48\%$ & $98.00\%$ && & $98.36\%$ & $99.22\%$ && \\
& ${SAD}_n$   & $94.79\%$ & $96.94\%$ && & $95.90\%$ & $98.20\%$ && & $96.94\%$ & $97.49\%$ && \\ 
\bottomrule

\end{tabular}
\end{center}
\caption{Wald and $SAD_n$ test for the problem (\ref{TestComp})--spatial dependence in the presence of a nuisance parameter, in a SARAR(1,1) model--with no covariates and known variance. The quantiles are obtained using 100 repetitions of each test statistic.}
\label{TabNuisance}
\end{table}

\section{Empirical application} \label{Sec: Real}

\citet{FH80} document empirically that domestic saving rate in a country has a positive correlation with the domestic investment rate. It contrasts with the understanding that, if capital is perfectly mobile between countries, most of any incremental saving is invested to get the highest return regardless of any locations, and that such correlation should actually vanish. \citet{DE10} suggest to use spatial modeling since several papers challenge these findings but under the strong assumption that investment rates are independent across countries.
Such an assumption might influence the conclusions of applied statial economics.

In this empirical exercise, we investigate the presence of spatial autocorrelation in the investment-saving relationship. We consider investment and saving rates for 24 OECD countries between 1960 and 2000 (41 years). Because of macroeconomic reasons (deregulating financial markets), we divide the whole period into shorter sub-periods: 1960-1970, 1971-1985 and 1986-2000, as advocated by \citet{DE10}. 
Since the cross-sectional  size is only $n=24$, the asymptotics may suffer from size distortion as documented in \S\ref{Sec: MC}. 
Therefore, we resort on a saddlepoint test to investigate whether or not there are inferential issues (coming from finite sample distortions and nuisance parameters) in the use of the first-order asymptotic theory. In line with the econometric literature, we specify the following SARAR(1,1) model for the three sub-periods:
\begin{equation}
\begin{aligned}
\text{ Inv}_{nt} &= \lambda_0 W_n \text{ Inv}_{nt} + \beta_0 \text{Sav}_{nt}  + c_{n0} + E\nt,& \\ 
E\nt &=\rho_0 M_n E\nt +V\nt, & \quad t=1,2, \cdots, T & 
\end{aligned}
\label{Eq: Inv}
\end{equation}
where $\text{Inv}_{nt}$ is the $n \times 1$ vector of investment rates for all countries and $\text{Sav}_{nt}$ is the $n \times 1$ vector of saving rates. Each element $v_{it}$ in $V_{nt}$  is i.i.d across $i$ and $t$, having Gaussian distribution with zero mean and variance $\sigma_0^2$. $c_{n0}$ is the vector of fixed effects. 

We assume $W_n=M_n$ and adopt two different weight matrices as in \citet{DE10}. The first one is based on the inverse distance. Each element $\omega_{ij}$ in $W_n $ is $d_{ij}^{-1}$, where $d_{ij}$ is the arc distance between capitals of countries $i$ and $j$. The second is the binary seven nearest neighbors (7NN) weight matrix. More precisely, $\omega_{ij}$=1, if $d_{ij} \leq d_{i}$ and $i \neq j$. Otherwise, $\omega_{ij}=0$, where $d_i$ is the $7_{th}$ order smallest arc-distance between countries $i$ and $j$ such that each country $i$ has exactly 7 neighbors. Both weight matrices are row-normalized. 

%
%
%
%

We estimate the parameters using the MLE described in \S\ref{Sec: Setting}. Table \ref{Table: Estimates} gathers the point estimates (and their standard errors) that agree with the magnitudes found by \citet{DE10}. To investigate the validity of the model (\ref{Eq: Inv}), we test for spatial dependence, working on $\lambda=0$ and/or $\rho=0$. Specifically,  our aim is to detect if and in which period(s) the inference yielded by the first-order asymptotic theory differs from the inference obtained using our saddlepoint test. With this goal, in 
Table \ref{Table: Pvalues} we provide the $p$-values for testing (at the $5\%$ level) three different composite hypotheses: in the first row, we consider the problem of testing for $\lambda=0$; in the second row, we test for $\rho=0$; in the third row, we test for $\lambda=\rho=0$. To perform the tests, we consider the routinely-applied Wald test (as obtained using the first-order asymptotic approximation, ASY) and the saddlepoint test ($SAD_n$}). In each testing procedure, we treat the parameters not specified by the null hypothesis as nuisance parameters. In the $SAD_n$ test, we take care of the nuisance as indicated in (\ref{Eq. h}), while in the ASY test we simply plug-in the MLE estimates for the nuisance parameters---as it is customary in the econometric software based on the first-order asymptotic theory.

In the period 60-70, both ASY and $SAD_n$ yield the same inference, for both the considered types of weight matrix, with conventional significance levels.
The other sub-periods display some discrepancies between the inference obtained via ASY and via $SAD_n$. We do not want to discuss all discrepancies but only briefly comment on some key differences---we highlights the corresponding values in Table \ref{Table: Pvalues}. In the sub-period 71-85 under 7NN $W_n$, the saddlepoint test finds no evidence against no spatial dependence in the investing rates across countries, and vice-versa for the asymptotic approximation. Moreover, the ASY test does not find evidence against $\rho=0$, while the $SAD_n$ test rejects this composite hypothesis. Thus, the $SAD_n$ test indicates a spillover through the contemporary shocks between countries.  This spillover goes through the innovations, i.e., through the unexpected part of the model dynamics, a finding not documentable when one relies on the first-order asymptotic theory. This results suggests that a test statistic designed to perform well in small samples and in the presence of nuisance parameters is able to document  spatial dependence in the disturbances $E_{nt}$. Some differences are  detectable also in the sub-period 86-00, under the inverse distance matrix.

 \begin{table}[htb!]
\begin{center}
\begin{tabular}{cccccccc}

\toprule
& \multicolumn{3}{l}{Weight matrix: inverse distance} & &  \multicolumn{3}{l}{Weight matrix: 7 nearest neighbours} \\
 \cmidrule{2-4} \cmidrule{6-8}
 &1960-1970& 1971-1985& 1986-2000& & 1960-1970& 1971-1985& 1986-2000\\
\midrule
$\beta$ & 0.935(0.05)
 &0.638(0.04)
 & 0.356(0.07)
 & &  0.932(0.05) &0.633(0.04) & 0.368(0.07) \\
$\lambda$ & 0.004(0.10)&0.381(0.11)&	0.430(0.30)

 & & -0.016(0.09)&	0.340(0.10)	&0.437(0.18) \\
$\rho$ &-0.305(0.22)&	0.334(0.16)&	0.222(0.40)

& & -0.219(0.19)&0.258(0.15)	&0.025(0.28)
\\
\bottomrule

\end{tabular}
\caption{SARAR(1,1) model: Maximum likelihood estimates of Parameters $\beta$, $\lambda$, $\rho$. Standard errors are between brackets.}
\label{Table: Estimates}
\end{center}
\end{table}

\begin{table}[hbt!]
\begin{center}
\begin{tabular}{ccccccccc}

\toprule
& &\multicolumn{3}{l}{Weight matrix: inverse distance} & &  \multicolumn{3}{l}{Weight matrix: 7 nearest neighbours} \\
 \cmidrule{3-5} \cmidrule{7-9}
 & &1960-1970& 1971-1985& 1986-2000& & 1960-1970& 1971-1985& 1986-2000\\
\midrule
\multirow {2}{*}{$\lambda=0$}&$SAD_n$ &  1.0000	&0.0096	&\textbf{0.0000} & &  0.9998& \textbf{0.2248}&	0.0000 \\
&ASY &1.0000	&0.0116	&\underline{0.5679}& &0.9987 &	\underline{0.0130}	&0.1123\\
\multirow {2}{*}{$\rho = 0$}&$SAD_n$ & 0.1134	&0.0024	&0.1217
 & & 0.3232&	\textbf{0.0403}&	0.9993 \\
&ASY&0.5890	&0.2261&	0.9578 & & 0.7101&	\underline{0.3898}	&0.9998\\
\multirow {2}{*}{$\lambda =\rho=0$}&$SAD_n$ &0.1414&	0.0000&	0.0000
& & 0.2603&	0.0000	&0.0000
\\
&ASY &0.4615&	0.0000	&0.0000& &0.5042&	0.0000	&0.0000\\
\bottomrule

\end{tabular}

\caption{SARAR(1,1) model: $p$-values of Saddlepoint ($SAD_n$) and Wald (ASY) tests for several composite hypotheses.}

\label{Table: Pvalues}
\end{center}
\end{table}

\begin{center}
{\large\bf SUPPLEMENTARY MATERIAL}
\end{center}
The online supplementary material includes proofs, 
lengthy analytical derivations and additional numerical results for the SAR(1) model. All the codes and data are available in our Github repository. 

{\small
\bibliographystyle{myapalike1}
\bibliography{biblio_Panel}}

\end{document}


\def\spacingset#1{\renewcommand{\baselinestretch}%
{#1}\small\normalsize} \spacingset{1}


\if1\blind
{
  \title{\begin{center}{\bf \Large Supplementary Material for ``Saddlepoint approximations for spatial panel data models''}\end{center}}
  \author{Chaonan Jiang,
    Davide La Vecchia, 
    Elvezio Ronchetti, 
    Olivier Scaillet 
}
  \maketitle
} \fi

\if0\blind
{
  \bigskip
  \bigskip
  \bigskip
\begin{center}{\bf \Large Supplementary Material for ``Saddlepoint approximations for spatial panel data models''}\end{center}
  \medskip
} \fi

\spacingset{1.5} 
\addtolength{\textheight}{.5in}%

\appendix

\noindent Appendix \ref{Section: Appendix1} includes proofs of Lemma 1-2 in A.1-2, 
Proposition 3-4 in A.3-4. \\
Appendix \ref{Section: Appendix2} provides the first, second and third derivatives of the log-likelihood in B.1-3 as well as some additional computations in B.4.\\
Appendix \ref{Section: AppendixAlgo} contains the algorithms which explain how to implement our saddlepoint techniques. \\
Appendix \ref{Section: Appendix3} gives additional numerical results for the SAR(1) model in \ref{Subsec: Asy1} 
and D.3-6, 
and the analytical check of Assumption D(iv) in \ref{Sec: D_check}. 

\section{Proofs of Lemmas and Propositions} \label{Section: Appendix1}

\subsection{Proof of Lemma 1
}  \label{Subsec: Lemma1}

\begin{proof}

Since the MLE is an M-estimator, we derive its second-order von Mises expansion using the results of \cite{vonMises} (see also \citet{F62}, and \citet{CF99}), and we get
\begin{equation}
\vartheta(P_{n,T})-\vartheta(P_{\theta _{0}})=\frac{1}{n}   \sum_{i=1}^n  IF_{i,T}( \psi, P_{\theta _{0}}) +\frac{1}{2n^2}  \sum_{i=1}^n \sum_{j=1}^n \varphi_{i,j,T}(\psi, P_{\theta _{0}}) + O_P(m^{-3/2}),  \label{IIorder}
\end{equation}
where we make use of the fact that $O_P(n^{-3/2})$ is also $O_P(m^{-3/2})$, since $m=n(T-1)$.
The expression of $IF_{i,T}(\psi,P_{\theta _{0}})$ and of $\varphi_{i,j,T}(\psi,P_{\theta _{0}})$ for general $M$-estimators are available in \citet{GR96}.  
Specifically, for the $i$-th observation and for the whole time span, $IF_{i,T}(\psi, P_{\theta _{0}})$
is the Influence Function (IF) of the MLE, having likelihood score $\displaystyle (T-1)^{-1}\sum_{t=1}^{T}\psi_{i,t}(\theta_{0})$, which reads as:
$\displaystyle 
IF_{i,T}(\psi,P_{\theta _{0}})=M_{i,T}^{-1}(\psi,P_{\theta _{0}}) (T-1)^{-1} \sum_{t=1}^T \psi_{i,t}(\theta_{0}),
$
From \citet{W83}, it follows that the second term of the von Mises expansion  in (\ref{IIorder}) is given by (4.7) and (4.8). 

To compute the second-order von Mises expansion, we need the matrices of partial derivatives 
$\frac{\partial^2 \psi_{i,t,l}(\theta)}{\partial\theta \partial\theta'}
\Big\vert_{\theta=\theta_0}$ for $l=1,2,\cdots,d$, whose expressions are provided in Appendix \ref{Section: Appendix2}. 
\end{proof}

\subsection{Proof of Lemma 2
} \label{Subsec: Lemma2}

\begin{proof}

For $m=n(T-1)$, the second-order von Mises expansion for $q[\vartheta(P_{n,T})]$ is:
\begin{eqnarray}
q[\vartheta(P_{n,T})] - q[\vartheta(P_{\theta_0})] &=& \left(\vartheta(P_{n,T})- \vartheta(P_{\theta_0})\right)' \frac{\partial q(\vartheta)}{\partial\vartheta}\Big
\vert_{\theta=\theta_0} \nonumber \\ & & + \frac{1}{2} \left(\vartheta(P_{n,T})- \vartheta(P_{\theta_0})\right)' \frac{\partial^2 q(\vartheta)}{\partial\vartheta
\partial\vartheta'}\Big \vert_{\theta=\theta_0}  \left(\vartheta(P_{n,T})- \vartheta(P_{\theta_0})\right) \nonumber \\
& & O_P(\vert\vert \vartheta(P_{n,T}) - \vartheta(P_{\theta_0})\vert\vert^3).
 \label{Eq: vMq}
\end{eqnarray}

Making use of (\ref{IIorder}), (4.6) and (4.7) 
into (\ref{Eq: vMq}), we get
\begin{eqnarray}
q[\vartheta(P_{n,T})] - q[\vartheta(P_{\theta_0})] &=& \frac{1}{n}  \sum_{i=1}^{n} IF'_{i,T}(\psi,P_{\theta_0})\frac{\partial q(\vartheta)}{\partial
\vartheta}\Big
\vert_{\theta=\theta_0} 
+  \frac{1}{2n^2}   \sum_{i=1}^{n} \sum_{j=1}^{n} \left[ \varphi'_{i,j,T}(\psi,P_{\theta_0}) \frac{\partial q(\vartheta)}{\partial\vartheta}\Big
\vert_{\theta=\theta_0}  \nonumber \right. \\
 &   +& \left. IF'_{i,T}(\psi,P_{\theta_0}) \frac{\partial^2 q(\vartheta)}{\partial\vartheta
\partial\vartheta'}\Big\vert_{\theta=\theta_0} IF_{j,T}(\psi,P_{\theta_0}) \right] 
+ O_P(m^{-3/2}).\label{Eq: qvonMises} 
\end{eqnarray}

Similarly to \citet{GR96}, we delete the diagonal terms from (\ref{Eq: qvonMises}),  and we define the following $U$-statistic 
of order two (see, e.g., \citet{Serf09} or \citet{vdW98}, page 295) by making use of (4.9), (4.10) and (4.11):
\begin{eqnarray} \label{Eq: Ustat}
U_{n,T}(\psi,\theta_0) &=& \frac{2}{n(n-1)}  \sum_{i=1}^{n-1} \sum_{j=i+1}^{n} h_{i,j,T}\left(\psi,P_{\theta_0} \right) \nn \\
 &=& \frac{2}{n}\sum_{i=1}^n g_{i,T}\left(\psi,P_{\theta_0} \right)+\frac{2}{n(n-1)}\sum_{i=1}^{n-1}\sum_{j=i+1}^{n} \gamma_{i,j,T}(\psi,P_{\theta_0}).
\end{eqnarray}

Then, we remark that $q[\vartheta(P_{n,T})] - q[\vartheta(P_{\theta_0})] $ in (\ref{Eq: qvonMises}) is equivalent (up to $O_P(m^{-3/2})$) to $U_{n,T}(\psi,\theta_0)$, namely 
\begin{equation} \label{Eq: Uapprox}
q[\vartheta(P_{n,T})] - q[\vartheta(P_{\theta_0})] = U_{n,T}(\psi,\theta_0) +  O_P(m^{-3/2}),
\end{equation}
which concludes the proof.
\end{proof}

\subsection{Proof of Proposition 3
} \label{App: proof_edge}

To derive the Edgeworth expansion for $\sigma_{n,T}^{-1}\{q[\vartheta(P_{n,T})] - q[\vartheta(P_{\theta_0})]\}$, 
we first introduce two lemmas.
\begin{lemma} \label{Lemma_Mmatrix}
Under Assumptions A-D, for all $i$, $j$ and $i \ne j$,
\begin{equation}
M_{i,T}^{-1}(\psi, P_{\theta_0})M_{j,T}(\psi, P_{\theta_0}) = I_d +O(n^{-1}),
\end{equation}
with $I_d$ representing the $(d\times d)$ identity matrix.
\end{lemma}
\begin{proof}  From the definition of $M_{i,T}(\psi, P_{\theta_0})$ in (4.4), 
we have:
{\small
\begin{eqnarray} \label{Eq: M_trans}
&&M_{i,T}^{-1}(\psi, P_{\theta_0})M_{j,T}(\psi, P_{\theta_0})\nn \\
&=&\left\{M_{j,T}(\psi, P_{\theta_0})-\left[M_{j,T}(\psi, P_{\theta_0})-M_{i,T}(\psi, P_{\theta_0})\right]\right\}^{-1}M_{j,T}(\psi, P_{\theta_0})\nn \\ 
&=& \left(M_{j,T}(\psi, P_{\theta_0})\left\{I_d-M_{j,T}^{-1}(\psi, P_{\theta_0})\left[M_{j,T}(\psi, P_{\theta_0})-M_{i,T}(\psi, P_{\theta_0})\right]\right\}\right)^{-1}M_{j,T}(\psi, P_{\theta_0}) \nn \\ 
&=& \left\{I_d-M_{j,T}^{-1}(\psi, P_{\theta_0})\left[M_{j,T}(\psi, P_{\theta_0})-M_{i,T}(\psi, P_{\theta_0})\right]\right\}^{-1}M_{j,T}^{-1}(\psi, P_{\theta_0}) M_{j,T}(\psi, P_{\theta_0}) \nn \\
&=& \left\{I_d-M_{j,T}^{-1}(\psi, P_{\theta_0})\left[M_{j,T}(\psi, P_{\theta_0})-M_{i,T}(\psi, P_{\theta_0})\right]\right\}^{-1}. 
\end{eqnarray}
}
By a Taylor expansion in (\ref{Eq: M_trans}), we get 
{\small
\begin{eqnarray} \label{Eq: M_taylor}
&&\left\{I_d-M_{j,T}^{-1}(\psi, P_{\theta_0})\left[M_{j,T}(\psi, P_{\theta_0})-M_{i,T}(\psi, P_{\theta_0})\right]\right\}^{-1} \nn \\
&=& I_d + M_{j,T}^{-1}(\psi, P_{\theta_0})\left[M_{j,T}(\psi, P_{\theta_0})-M_{i,T}(\psi, P_{\theta_0})\right]+ \left(M_{j,T}^{-1}(\psi, P_{\theta_0})\left[M_{j,T}(\psi, P_{\theta_0})-M_{i,T}(\psi, P_{\theta_0})\right]\right)^2 \nn \\
&+& \left(M_{j,T}^{-1}(\psi, P_{\theta_0})\left[M_{j,T}(\psi, P_{\theta_0})-M_{i,T}(\psi, P_{\theta_0})\right]\right)^3+ \cdots \nn \\
&=& I_d+ M_{j,T}^{-1}(\psi, P_{\theta_0})\left[M_{j,T}(\psi, P_{\theta_0})-M_{i,T}(\psi, P_{\theta_0})\right]\left\{I_d + M_{j,T}^{-1}(\psi, P_{\theta_0})\left[M_{j,T}(\psi, P_{\theta_0})-M_{i,T}(\psi, P_{\theta_0})\right] \right. \nn \\
&+&\left.\left(M_{j,T}^{-1}(\psi, P_{\theta_0})\left[M_{j,T}(\psi, P_{\theta_0})-M_{i,T}(\psi, P_{\theta_0})\right]\right)^2+\cdots \right\} \nn \\
&=& I_d + M_{j,T}^{-1}(\psi, P_{\theta_0})\left[M_{j,T}(\psi, P_{\theta_0})-M_{i,T}(\psi, P_{\theta_0})\right]\left\{I_d - M_{j,T}^{-1}(\psi, P_{\theta_0})\left[M_{j,T}(\psi, P_{\theta_0})-M_{i,T}(\psi, P_{\theta_0})\right] \right\}^{-1}. \nn \\
\end{eqnarray}}

From Assumption A, we know that $M_{j,T}^{-1}(\psi, P_{\theta_0}) =O(1)$. Under Assumption D({\it iv}), we get 
{\small 
\begin{eqnarray} \label{Eq: matrix_a}
\vert\vert M_{j,T}^{-1}(\psi, P_{\theta_0})\left[M_{j,T}(\psi, P_{\theta_0})-M_{i,T}(\psi, P_{\theta_0})\right]\vert \vert
=O(n^{-1}).
\end{eqnarray}
}

By (\ref{Eq: M_trans}), (\ref{Eq: M_taylor}) and (\ref{Eq: matrix_a}), we finally find
{\small
\begin{eqnarray}
&&\vert\vert M_{i,T}^{-1}(\psi, P_{\theta_0})M_{j,T}(\psi, P_{\theta_0})-I_d\vert\vert \nn \\
&=& \vert\vert M_{j,T}^{-1}(\psi, P_{\theta_0})\left[M_{j,T}(\psi, P_{\theta_0})-M_{i,T}(\psi, P_{\theta_0})\right]\left\{I_d - M_{j,T}^{-1}(\psi, P_{\theta_0})\left[M_{j,T}(\psi, P_{\theta_0})-M_{i,T}(\psi, P_{\theta_0})\right] \right\}^{-1}\vert\vert \nn \\
&\leq& \vert\vert M_{j,T}^{-1}(\psi, P_{\theta_0})\left[M_{j,T}(\psi, P_{\theta_0})-M_{i,T}(\psi, P_{\theta_0})\right]\vert \vert \ \vert\vert \left \{ I_d - M_{j,T}^{-1}(\psi, P_{\theta_0})\left[M_{j,T}(\psi, P_{\theta_0})-M_{i,T}(\psi, P_{\theta_0})\right]\right\}^{-1}\vert\vert \nn \\
&=& \vert\vert M_{j,T}^{-1}(\psi, P_{\theta_0})\left[M_{j,T}(\psi, P_{\theta_0})-M_{i,T}(\psi, P_{\theta_0})\right]\vert \vert \  \vert\vert I_d+ M_{j,T}^{-1}(\psi, P_{\theta_0})\left[M_{j,T}(\psi, P_{\theta_0})-M_{i,T}(\psi, P_{\theta_0})\right]+\cdots \vert \vert  \nn \\
&=& O(n^{-1}).
\end{eqnarray}
}
Thus, 
{\small
\begin{equation}
M_{i,T}^{-1}(\psi, P_{\theta_0})M_{j,T}(\psi, P_{\theta_0}) = I_d +O(n^{-1}).
\end{equation}
}
\end{proof}

\begin{lemma} \label{Lemma_cgf}
Under Assumptions A-D, 
\begin{itemize}
    \item[(i)] $g_{i,T}\left(\psi,P_{\theta_0} \right)$ and $\gamma_{i,j,T}(\psi,P_{\theta_0})$, $1\leq i < j \leq n $ are asymptotically pairwise uncorrelated.
The mean, variance, the standardized third and fourth cumulant of $U_{n,T}$
(defined by (\ref{Eq: Ustat})) are given by the following expressions.
    \item[(ii)] Mean:
        \begin{equation}
            \mu_{n,T}=\mathbb{E}\left[U_{n,T}(\psi,\theta_0)\right]=0.
        \label{Eq: Mean}
        \end{equation}
    \item[(iii)] Variance: 
        \begin{eqnarray} \label{Eq: variance}
            \sigma_{n,T}^2 &=& Var\left[U_{n,T}(\psi,\theta_0)\right]\\ \nn 
            &=& \frac{4}{n}\sigma_{g}^2+\frac{4}{n^2(n-1)^2}  \sum_{i=1}^{n-1} \sum_{j=i+1}^{n} \mathbb{E}\left[\gamma_{i,j,T}^2(\psi,P_{\theta_0})\right]+O(n^{-3}), \nn
        \end{eqnarray}
        where 
        \begin{equation}   \label{Eq: var_g}
            \sigma_{g}^2 = \frac{1}{n}\sum_{i=1}^{n}\mathbb{E}\left[g_{i,T}^2(\psi,P_{\theta_0})\right].
        \end{equation}
    \item[(iv)] Third standardized cumulant is   such that   
$$   \tilde{\kappa}_{n,T}^{(3)} 
             =  \mathbb{E}\left[ U_{n,T}^3(\psi,\theta_0)/\sigma_{n,T}^3\right] = n^{-1/2}\kappa_{n,T}^{(3)}+O(n^{-3/2})
             $$
        where: 
        \begin{equation}
            \kappa_{n,T}^{(3)} =  \sigma_{g}^{-3}  (\overline{g^3}+3 \overline{g_1g_2\gamma_{12}}) ,
            \label{Eq: 3rdk}
        \end{equation}
        \begin{equation}  \label{Eq: g3}
            \overline{g^3}= \frac{1}{n}\sum_{i=1}^{n}\mathbb{E}\left[g_{i,T}^3(\psi,P_{\theta_0})\right],
        \end{equation}
        \begin{equation}   \label{Eq: g1g2gamma12}
            \overline{g_1g_2\gamma_{12}}= \frac{2}{n(n-1)}\sum_{i=1}^{n-1} \sum_{j=i+1}^{n}\mathbb{E}\left[g_{i,T}\left(\psi,P_{\theta_0} \right) g_{j,T}\left(\psi,P_{\theta_0} \right)\gamma_{i,j,T}(\psi,P_{\theta_0})\right].
        \end{equation}
    \item[(v)] The fourth standardized cumulant is   such that:
       $$    
         \tilde{\kappa}_{n,T}^{(4)}   \nn
         =  \mathbb{E}\left[ U_{n,T}^4(\psi,\theta_0)/\sigma_{n,T}^4\right]-3 =  n^{-1}\kappa_{n,T}^{(4)}+O(n^{-2}),
         $$
         where
        \begin{equation}
            \kappa_{n,T}^{(4)} =  \sigma_{g}^{-4}(\overline{g^4}+12 \overline{g_1g_2\gamma_{13}\gamma_{23}}+12 \overline{g_1^2g_2\gamma_{12}})-3,
            \label{Eq: 4rdk}
        \end{equation}   
        \begin{equation}   \label{Eq: g4}
            \overline{g^4}= \frac{1}{n}\sum_{i=1}^{n}\mathbb{E}\left[g_{i,T}^4(\psi,P_{\theta_0})\right],
        \end{equation}
        \begin{equation}  \label{Eq: g1g2ga123}
            \overline{g_1g_2\gamma_{13}\gamma_{23}}= \frac{2}{n(n-1)(n-2)}\sum_{i=1}^{n-1} \sum_{j=i+1}^{n}\sum_{\substack{k=1 \\ k\ne i,j}}^{n}\mathbb{E}\left[g_{i,T}\left(\psi,P_{\theta_0} \right) g_{j,T}\left(\psi,P_{\theta_0} \right)\gamma_{i,k,T}(\psi,P_{\theta_0})\gamma_{j,k,T}(\psi,P_{\theta_0})\right],
        \end{equation}
        \begin{equation}  \label{Eq: g12g2ga12}
            \overline{g_1^2g_2\gamma_{12}}= \frac{1}{n(n-1)}\sum_{i=1}^{n-1} \sum_{j=i+1}^{n}\mathbb{E}\left[\left(g_{i,T}\left(\psi,P_{\theta_0} \right)+g_{j,T}\left(\psi,P_{\theta_0} \right)\right)g_{i,T}\left(\psi,P_{\theta_0} \right) g_{j,T}\left(\psi,P_{\theta_0} \right)\gamma_{i,j,T}(\psi,P_{\theta_0})\right].
        \end{equation}
\end{itemize}
\end{lemma}
 \begin{proof}
\begin{itemize}
    \item[(i)] Under Lemma \ref{Lemma_Mmatrix} and the fact that $IF_{i,T}(\psi, P_{\theta_0})=O_P(1)$, using the definitions of $ M_{i,T}(\psi,P_{\theta _{0}})$ in (4.4), 
    $\gamma_{i,j,T}(\psi,P_{\theta_0})$ in (4.11), 
    and $\varphi_{i,j, T}(\psi,P_{\theta_0})$ in (4.7), 
    we get for the conditional expectation:
    {\small
    \begin{eqnarray}\label{Eq: ConE}
    &&\mathbb{E}\left[\gamma_{i,j,T}(\psi,P_{\theta_0})\Big \vert \frac{1}{T-1}\sum_{t=1}^{T}\psi_{i,t}(\theta_0) \right] \nn \\ \nn
    &=&\mathbb{E}\left[\frac{1}{2} 
    \left( \varphi'_{i,j, T}(\psi,P_{\theta_0})\frac{\partial q(\vartheta)}{\partial\vartheta}\Big\vert_{\theta=\theta_0}  + IF'_{i,T}(\psi,P_{\theta_0}) \frac{\partial^2 q(\vartheta)}{\partial\vartheta\partial\vartheta'}\Big\vert
    _{\theta=\theta_0} IF_{j,T}(\psi,P_{\theta_0})  \right)\Big \vert \frac{1}{T-1}\sum_{t=1}^{T}\psi_{i,t}(\theta_0) \right] \\ \nn
    &=&\frac{1}{2} IF'_{i,T}(\psi, P_{\theta_0})\frac{\partial q(\vartheta)}{\partial\vartheta}\Big\vert_{\theta=\theta_0} \\
    &+& \frac{1}{2} \left(
    M_{i,T}^{-1}(\psi,P_{\theta _{0}})\mathbb{E}\left[(T-1)^{-1}\sum_{t=1}^{T}\frac{\partial \psi_{j,t}(\theta)}{\partial\theta}\Big\vert_{\theta=\theta_0}\right]IF_{i,T}(\psi,P_{\theta_0}) 
    \right)'\frac{\partial q(\vartheta)}{\partial\vartheta}\Big\vert_{\theta=\theta_0} \nn \\ \nn
    &=& \frac{1}{2} IF'_{i,T}(\psi, P_{\theta_0})\frac{\partial q(\vartheta)}{\partial\vartheta}\Big\vert_{\theta=\theta_0}-\frac{1}{2} \left(
    M_{i,T}^{-1}(\psi,P_{\theta _{0}})M_{j,T}(\psi,P_{\theta _{0}})IF_{i,T}(\psi,P_{\theta_0}) 
    \right)'\frac{\partial q(\vartheta)}{\partial\vartheta}\Big\vert_{\theta=\theta_0}\\  \nn
    &=&\frac{1}{2} IF'_{i,T}(\psi, P_{\theta_0})\frac{\partial q(\vartheta)}{\partial\vartheta}\Big\vert_{\theta=\theta_0}-\frac{1}{2} IF'_{i,T}(\psi, P_{\theta_0})\frac{\partial q(\vartheta)}{\partial\vartheta}\Big\vert_{\theta=\theta_0} + O_P(n^{-1})\\ 
    &=&O_P(n^{-1}).
    \end{eqnarray}
    }
    So we deduce that  $g_{i,T}\left(\psi,P_{\theta_0} \right)$ in (4.10) 
    and $\gamma_{i,j,T}(\psi,P_{\theta_0})$  in (4.11), 
    $1\leq i < j \leq n $ are pairwise uncorrelated, up to an $O_P(n^{-1})$ term.
    
    \item[(ii)] From the independence of the estimating function and using (\ref{Eq: Ustat}) and (4.9), 
    we get for the mean $ \mu_{n,T}$ of $U_{n,T}(\psi,\theta_0)$:
    {\small
        \begin{eqnarray}
            \mu_{n,T}&=&\mathbb{E}\left[U_{n,T}(\psi,\theta_0)\right] \nn \\ \nn
            &=&\mathbb{E}\left[\frac{2}{n(n-1)}\sum_{i=1}^{n-1}\sum_{j=i+1}^{n} h_{i,j,T}\left(\psi,P_{\theta_0} \right)\right] \\ \nn
            &=&\frac{1}{n(n-1)}\sum_{i=1}^{n-1}\sum_{j=i+1}^{n}\mathbb{E}\left[ \left\{IF'_{i,T}(\psi,P_{\theta_0}) + IF'_{j,T}(\psi,P_{\theta_0}) +   \varphi'_{i,j,T}(\psi,P_{\theta_0}) \right\}
            \frac{\partial q(\vartheta)}{\partial\vartheta}\Big\vert_{\theta=\theta_0} \right. \\ \nn
            & + &\left. IF'_{i,T}(\psi,P_{\theta_0})  \frac{\partial^2 q(\vartheta)}{\partial\vartheta
            \partial\vartheta'}\Big\vert_{\theta=\theta_0} IF_{j,T}(\psi,P_{\theta_0}) \right] \nn \\&=&0.   
        \end{eqnarray}
        }
    \item[(iii)] From the asymptotic pairwise uncorrelation of $g_{i,T}\left(\psi,P_{\theta_0} \right)$
and $\gamma_{i,j,T}(\psi,P_{\theta_0})$, we know that $$\mathbb{E}\left[g_{i,T}(\psi,P_{\theta_0})\gamma_{i,j,T}(\psi,P_{\theta_0})\right] = O(n^{-1}).$$
 Then by using (\ref{Eq: Ustat}) and (\ref{Eq: var_g}), we get for the variance
 $\sigma_{n,T}^2$ of $U_{n,T}(\psi,\theta_0)$:
 {\small
        \begin{eqnarray} 
            \sigma_{n,T}^2 &=& Var\left[U_{n,T}(\psi,\theta_0)\right] \nn \\ \nn &=&\frac{4}{n^2}\sum_{i=1}^{n}\mathbb{E}\left[g_{i,T}^2(\psi,P_{\theta_0})\right]+  \frac{4}{n^2(n-1)^2}  \sum_{i=1}^{n-1} \sum_{j=i+1}^{n} \mathbb{E}\left[\gamma_{i,j,T}^2(\psi,P_{\theta_0})\right]\\ 
            &+& \frac{8}{n^2(n-1)^2}  \sum_{i=1}^{n-1} \sum_{j=i+1}^{n} \mathbb{E}\left[g_{i,T}(\psi,P_{\theta_0})\gamma_{i,j,T}(\psi,P_{\theta_0})\right] \nn \\        
            &=&\frac{4}{n}\sigma_{g}^2+\frac{4}{n^2(n-1)^2}  \sum_{i=1}^{n-1} \sum_{j=i+1}^{n} \mathbb{E}\left[\gamma_{i,j,T}^2(\psi,P_{\theta_0})\right]+O(n^{-3})
        \end{eqnarray}
        }
     \item[(iv)] Due to the asymptotic pairwise uncorrelation of $g_{i,T}\left(\psi,P_{\theta_0} \right)$
and $\gamma_{i,j,T}(\psi,P_{\theta_0})$, several expectations in the calculation of cumulants are of order $O(n^{-1})$, for example $\mathbb{E}\left[g^2_{i,T}(\psi,P_{\theta_0})\gamma_{i,j,T}(\psi,P_{\theta_0})\right]$, 
$1\leq i < j  \leq n$. Making use of (\ref{Eq: Ustat}), (\ref{Eq: variance}), (\ref{Eq: g3}), (\ref{Eq: g1g2gamma12}), and (\ref{Eq: 3rdk}), we get for the third cumulant of $\sigma_{n,T}^{-1}U_{n,T}(\psi,\theta_0)$:{\small
         \begin{eqnarray}
             \tilde{\kappa}_{n,T}^{(3)} 
             &=&  \mathbb{E}\left[ U_{n,T}^3(\psi,\theta_0)/\sigma_{n,T}^3\right] \nn \\
             &=& \sigma_{n,T}^{-3}\mathbb{E}\left[ \left(   \frac{2}{n}\sum_{i=1}^n g_{i,T}\left(\psi,P_{\theta_0} \right)+\frac{2}{n(n-1)}\sum_{i=1}^{n-1}\sum_{j=i+1}^{n} \gamma_{i,j,T}(\psi,P_{\theta_0}) \right )^3\right]\nn \\
             &=&  \sigma_{n,T}^{-3}\frac{8}{n^3(n-1)^3}  \sum_{i=1}^{n} (n-1)^3\mathbb{E}\left[g_{i,T}^3\left(\psi,P_{\theta_0} \right)\right]+\sigma_{n,T}^{-3}\frac{8}{n^3(n-1)^3}  \sum_{i=1}^{n-1} \sum_{j=i+1}^{n}\mathbb{E}\left[\gamma_{i,j,T}^3(\psi,P_{\theta_0}) \right]\nn \\
             &+& \sigma_{n,T}^{-3}\frac{8}{n^3(n-1)^3}  \sum_{i=1}^{n-1} \sum_{j=i+1}^{n}6(n-1)^2\mathbb{E}\left[g_{i,T}\left(\psi,P_{\theta_0} \right) g_{j,T}\left(\psi,P_{\theta_0} \right)\gamma_{i,j,T}(\psi,P_{\theta_0}) \right] \nn \\
             &+&\sigma_{n,T}^{-3}\frac{8}{n^3(n-1)^3}  \sum_{i=1}^{n-1} \sum_{j=i+1}^{n}3(n-1)\mathbb{E}\left[\left\{g_{i,T}\left(\psi,P_{\theta_0} \right) +g_{j,T}\left(\psi,P_{\theta_0} \right)\right\}\gamma_{i,j,T}^2(\psi,P_{\theta_0}) \right]\nn \\  
             &=&  \sigma_{n,T}^{-3}\frac{8}{n^3(n-1)^3}  \sum_{i=1}^{n} (n-1)^3\mathbb{E}\left[g_{i,T}^3\left(\psi,P_{\theta_0} \right)\right]\nn \\
             &+& \sigma_{n,T}^{-3}\frac{8}{n^3(n-1)^3}  \sum_{i=1}^{n-1} \sum_{j=i+1}^{n}6(n-1)^2\mathbb{E}\left[g_{i,T}\left(\psi,P_{\theta_0} \right) g_{j,T}\left(\psi,P_{\theta_0} \right)\gamma_{i,j,T}(\psi,P_{\theta_0}) \right] +O(n^{-3/2})\nn \\
             &=&  \sigma_{g}^{-3} n^{-1/2} (\overline{g^3}+3 \overline{g_1g_2\gamma_{12}}) +O(n^{-3/2}) \nn \\
             &=&  n^{-1/2}\kappa_{n,T}^{(3)}+O(n^{-3/2}).
         \end{eqnarray}
         }
     \item[(v)]
     Similarly, making use of (\ref{Eq: Ustat}), (\ref{Eq: variance}), (\ref{Eq: g4}), (\ref{Eq: g1g2ga123}), (\ref{Eq: g12g2ga12}), and (\ref{Eq: 4rdk}), we get for the fourth cumulant of $\sigma_{n,T}^{-1}U_{n,T}(\psi,\theta_0)$:
     \begin{footnotesize}
         \begin{eqnarray*}
         & &\tilde{\kappa}_{n,T}^{(4)}   \nn\\
         &=&  \mathbb{E}\left[ U_{n,T}^4(\psi,\theta_0)/\sigma_{n,T}^4\right]-3\\ \nn
         &=& -3 +  \sigma_{n,T}^{-4}\mathbb{E}\left[ \left(   \frac{2}{n}\sum_{i=1}^n g_{i,T}\left(\psi,P_{\theta_0} \right)+\frac{2}{n(n-1)}\sum_{i=1}^{n-1}\sum_{j=i+1}^{n} \gamma_{i,j,T}(\psi,P_{\theta_0})\right )^4\right] \\ \nn
         &=& -3+ \sigma_{n,T}^{-4}\frac{16}{n^4(n-1)^4} \sum_{i=1}^{n}(n-1)^4 \mathbb{E}\left[g_{i,T}^4\left(\psi,P_{\theta_0} \right) \right] + \sigma_{n,T}^{-4}\frac{16}{n^4(n-1)^4}  \sum_{i=1}^{n-1} \sum_{j=i+1}^{n} \mathbb{E}\left[\gamma_{i,j,T}^4(\psi,P_{\theta_0}) \right] \\ \nn
         &+&  \sigma_{n,T}^{-4}\frac{16}{n^4(n-1)^4}  \sum_{i=1}^{n-1} \sum_{j=i+1}^{n} 4(n-1)\mathbb{E}\left[\left(g_{i,T}\left(\psi,P_{\theta_0} \right) +g_{j,T}\left(\psi,P_{\theta_0} \right)\right)\gamma_{i,j,T}^3(\psi,P_{\theta_0}) \right] \\ \nn
         &+&  \sigma_{n,T}^{-4}\frac{16}{n^4(n-1)^4}  \sum_{i=1}^{n-1} \sum_{j=i+1}^{n} 6(n-1)^2\mathbb{E}\left[\left(g_{i,T}^2\left(\psi,P_{\theta_0} \right) +g_{j,T}^2\left(\psi,P_{\theta_0} \right)\right)\gamma_{i,j,T}^2(\psi,P_{\theta_0}) \right] \\ \nn
         &+&   \sigma_{n,T}^{-4}\frac{16}{n^4(n-1)^4}  \sum_{i=1}^{n-1} \sum_{j=i+1}^{n} 12(n-1)^3\mathbb{E}\left[\left(g_{i,T}^2\left(\psi,P_{\theta_0} \right)g_{j,T}\left(\psi,P_{\theta_0} \right) +g_{i,T}\left(\psi,P_{\theta_0} \right)g_{j,T}^2\left(\psi,P_{\theta_0} \right)\right)\gamma_{i,j,T}(\psi,P_{\theta_0}) \right]  \\ \nn 
         &+& \sigma_{n,T}^{-4}\frac{16}{n^4(n-1)^4}  \sum_{i=1}^{n-1} \sum_{j=i+1}^{n} 6(n-1)^2\mathbb{E}\left[g_{i,T}\left(\psi,P_{\theta_0} \right) g_{j,T}\left(\psi,P_{\theta_0} \right)\gamma_{i,j,T}^2(\psi,P_{\theta_0}) \right] \\ \nn
         &+& \sigma_{n,T}^{-4}\frac{16}{n^4(n-1)^4}  \sum_{i=1}^{n-1} \sum_{j=i+1}^{n} 6(n-1)^4\mathbb{E}\left[g_{i,T}^2\left(\psi,P_{\theta_0} \right) g_{j,T}^2\left(\psi,P_{\theta_0} \right) \right] \\ \nn
         &+& \sigma_{n,T}^{-4}\frac{16}{n^4(n-1)^4}  \sum_{i=1}^{n-1} \sum_{j=i+1}^{n} \sum_{\substack{k=1 \\ k\ne i,j}}^{n} 24(n-1)^2\mathbb{E}\left[g_{i,T}\left(\psi,P_{\theta_0} \right) g_{j,T}\left(\psi,P_{\theta_0} \right)\gamma_{i,k,T}(\psi,P_{\theta_0})\gamma_{j,k,T}(\psi,P_{\theta_0}) \right] \\ \nn
         &=& -3+\sigma_{n,T}^{-4}\frac{16}{n^4(n-1)^4} \sum_{i=1}^{n}(n-1)^4 \mathbb{E}\left[g_{i,T}^4\left(\psi,P_{\theta_0} \right) \right] \\ \nn
         &+& \sigma_{n,T}^{-4}\frac{16}{n^4(n-1)^4}  \sum_{i=1}^{n-1} \sum_{j=i+1}^{n} 6(n-1)^4\mathbb{E}\left[g_{i,T}^2\left(\psi,P_{\theta_0} \right) g_{j,T}^2\left(\psi,P_{\theta_0} \right) \right] \\ \nn
         &+& \sigma_{n,T}^{-4}\frac{16}{n^4(n-1)^4}  \sum_{i=1}^{n-1} \sum_{j=i+1}^{n} \sum_{\substack{k=1 \\ k\ne i,j}}^{n} 24(n-1)^2\mathbb{E}\left[g_{i,T}\left(\psi,P_{\theta_0} \right) g_{j,T}\left(\psi,P_{\theta_0} \right)\gamma_{i,k,T}(\psi,P_{\theta_0})\gamma_{j,k,T}(\psi,P_{\theta_0}) \right]  \nn \\ \nn 
         &+& \sigma_{n,T}^{-4}\frac{16}{n^4(n-1)^4}  \sum_{i=1}^{n-1} \sum_{j=i+1}^{n} 12(n-1)^3\mathbb{E}\left[\left(g_{i,T}\left(\psi,P_{\theta_0} \right) +g_{j,T}\left(\psi,P_{\theta_0} \right)\right)g_{i,T}\left(\psi,P_{\theta_0} \right)g_{j,T}\left(\psi,P_{\theta_0} \right)\gamma_{i,j,T}(\psi,P_{\theta_0}) \right] \nn \\
         &+& O(n^{-2}) \nn \\ 
         &=&  \sigma_{g}^{-4} n^{-1} (\overline{g^4}+12 \overline{g_1g_2\gamma_{13}\gamma_{23}}+12 \overline{g_1^2g_2\gamma_{12}})-3n^{-1} +O(n^{-2}) \nn \\
         &=& n^{-1}\kappa_{n,T}^{(4)}+O(n^{-2}).
         \end{eqnarray*}
    \end{footnotesize}
  \end{itemize}
\end{proof}
Now we can prove Proposition (3). 
Let $\Psi_{n,T}$ be the characteristic function (c.f.) of $\sigma_{n,T}^{-1}\{q[\vartheta(P_{n,T})] - q[\vartheta(P_{\theta_0})]\}$,
\begin{equation}
   \Psi_{n,T}(z)=\mathbb{E}\left[\exp\left(\iota t \sigma_{n,T}^{-1}\{q[\vartheta(P_{n,T})] - q[\vartheta(P_{\theta_0})]\}\right)\right], 
\end{equation}
where $\iota^2=-1$. Making use of $\kappa_{n,T}^{(3)}$ in (\ref{Eq: 3rdk}) and $\kappa_{n,T}^{(4)}$ in (\ref{Eq: 4rdk}), we define 
\begin{equation} \label{Eq: AppCF}
\Psi_{n,T}^{*}(z)=\left\{1+n^{-1/2}\kappa_{n,T}^{(3)}\frac{(\iota z)^3}{6}+n^{-1}\kappa_{n,T}^{(4)}\frac{(\iota z)^4}{24}+n^{-1}\left(\kappa_{n,T}^{(3)}\right)^2\frac{(\iota z)^6}{72}\right\}e^{-z^{2}/2}.
\end{equation} as the approximate c.f.. To prove (4.12) 
in Proposition 3, 
we use Esseen smoothing lemma as in \citet{F71_book} and show that there exist sequences $\{Z_n\}$ and $\{\varepsilon'_n\}$ such that $n^{-1}Z_n \rightarrow \infty$, $\varepsilon'_n \rightarrow 0$, and
\begin{equation}
    \int_{-Z_n}^{Z_n}\Big\vert \frac{\Psi_{n,T}(z)-\Psi_{n,T}^{*}(z)}{z} \Big\vert dz \leq \varepsilon'_nn^{-1}.
\end{equation}
We proceed along the same lines as in \citet{BGVZ86}. We work on a compact subset of $\mathbb{R}$ and we consider the c.f. for small $\vert z \vert$. Then, we prove Lemma \ref{OurLemma21}, which
essentially shows the validity of the Edgeworth by means of the Esseen lemma. With this regard, we flag that we need to prove the Esseen lemma within our setting (we are dealing with independent but not identically distributed random variables) and we cannot invoke directly the results in the the paper by \citet{BGVZ86}. To this end, we prove a new result similar to the one in Lemma 2.1 of the last paper, adapting their proof to our context---see 
Lemma \ref{OurLemma21} below. Finally, the application of the derived results concludes the proof.


\begin{lemma} \label{OurLemma21}
  Under Assumptions A-D, there exists a sequence $\varepsilon^{''}_{n} \downarrow 0$ such that for
  \begin{equation}
      z_n= n^{(r-1)/r}(\log n)^{-1}, 
  \end{equation}
  \begin{equation} \label{InE: ES}
      \int_{-z_n}^{z_n}\Big\vert \frac{\Psi_{n,T}(z)-\Psi_{n,T}^{*}(z)}{z} \Big\vert dz \leq \varepsilon''_nn^{-1}.
  \end{equation}
\end{lemma}
\begin{proof} 
For the sake of readability, we split the proof in five steps. At the beginning of each step, we explain the goal of the derivation.
\begin{itemize}
\item[Step 1.] We approximate the characteristic function (c.f.) of $\hat{\theta}_{n,T}$ via the c.f. of $U_{n,T}$, up to the suitable order. This yields (\ref{Eq: cfExpan1}).

   Let us decompose the $U$-statistic $U_{n,T}$ in (\ref{Eq: Ustat}) as $U_{n,T} = U_{1,n,T} + U_{2,n,T}$ with  $U_{1,n,T} =  \frac{2}{n}\sum_{i=1}^n g_{i,T}\left(\psi,P_{\theta_0} \right)$ and $U_{2,n,T} =  \frac{2}{n(n-1)}\sum_{i=1}^{n-1}\sum_{j=i+1}^{n} \gamma_{i,j,T}(\psi,P_{\theta_0})$.
   Making use of (2.6), (2.7) as in \citet{BGVZ86}, (\ref{Eq: Uapprox}) and the $U$-statistic decomposition, we can write
   {\small
   \begin{eqnarray} 
       \Psi_{n,T}(z) & = & \mathbb{E}\left[\exp\left(\iota z\sigma_{n,T}^{-1}U_{1,n,T}\right)\left(1+\iota z\sigma_{n,T}^{-1}U_{2,n,T}-\frac{1}{2}z^2\sigma_{n,T}^{-2}U_{2,n,T}^2\right)\right] \notag \\ & & \qquad +O(E\vert z \sigma_{n,T}^{-1}U_{2,n,T} \vert ^{2+\delta})+O(n^{-3/2}\vert z \sigma_{n,T}^{-1}\vert ), \label{Eq: cfExpan}
   \end{eqnarray}}
   for $\delta \in (0,1]$. Let 
   {\small
   \begin{equation}
        \Psi_{g,i,T}(z) = \mathbb{E}\left[ \exp\left( \iota z\sigma_{n,T}^{-1}\frac{2}{n}g_{i,T}(\psi,P_{\theta_0}) \right)\right] 
   \end{equation}}
   be the c.f. of $ \sigma_{n,T}^{-1}\frac{2}{n}g_{i,T}(\psi,P_{\theta_0})$. In view of (\ref{Eq: variance}), and the fact that $\mathbb{E}\left[\vert U_{2,n,T}\vert ^{2+\delta}\right]=O(n^{2+\delta})$ (see \citet{CJ78}), we rewrite (\ref{Eq: cfExpan}) as 
  \begin{footnotesize}
   \begin{eqnarray}  \label{Eq: cfExpan1}
   &&\Psi_{n,T}(z)\nn \\ &=& \prod_{i=1}^{n} \Psi_{g,i,T}(z) \nn \\  
   &+&\sum_{j=1}^{n-1}\sum_{k=j+1}^{n}\left\{\prod_{\substack{i=1 \\ i\ne j,k}}^{n} \Psi_{g,i,T}(z)\right\}\left\{\iota z\sigma_{n,T}^{-1} \mathbb{E}\left[\exp\left(\iota z\sigma_{n,T}^{-1}\frac{2}{n}\left\{g_{j,T}(\psi,P_{\theta_0})+g_{k,T}(\psi,P_{\theta_0})\right\}\right)\frac{2\gamma_{j,k,T}(\psi,P_{\theta_0})}{n(n-1)}\right] \right.\nn \\
   &-& \left. \frac{1}{2}z^2\sigma_{n,T}^{-2} \mathbb{E}\left[\exp\left(\iota z\sigma_{n,T}^{-1}\frac{2}{n}\left\{g_{j,T}(\psi,P_{\theta_0})+g_{k,T}(\psi,P_{\theta_0})\right\}\right)\frac{4}{n^2(n-1)^2}\gamma_{j,k,T}^2(\psi,P_{\theta_0})\right] \right\}\nn \\
 &-& \sum_{j=1}^{n-2}\sum_{k=j+1}^{n-1}\sum_{m=k+1}^{n}\left\{\prod_{\substack{i=1 \\ i\ne j,k,m}}^{n} \Psi_{g,i,T}(z)\right\}z^2\sigma_{n,T}^{-2} \frac{4}{n^2(n-1)^2} \nn \\ 
&\times& \mathbb{E}\left[\exp\left(\iota z\sigma_{n,T}^{-1}\frac{2}{n}\left\{g_{j,T}(\psi,P_{\theta_0})+g_{k,T}(\psi,P_{\theta_0})+g_{m,T}(\psi,P_{\theta_0})\right \}\right) \left\{ \gamma_{j,k,T}(\psi,P_{\theta_0})\gamma_{j,m,T}(\psi,P_{\theta_0})\right. \right. \nn \\
&+& \left. \left. \gamma_{j,m,T}(\psi,P_{\theta_0})\gamma_{k,m,T}(\psi,P_{\theta_0})+\gamma_{j,k,T}(\psi,P_{\theta_0})\gamma_{k,m,T}(\psi,P_{\theta_0}) \right\}\right. \bigg]\nn \\
     &-&  \sum_{j=1}^{n-3}\sum_{k=j+1}^{n-2}\sum_{m=k+1}^{n-1}\sum_{l=m+1}^{n}\left\{\prod_{\substack{i=1 \\ i\ne j,k,m,l}}^{n} \Psi_{g,i,T}(z)\right\}z^2\sigma_{n,T}^{-2}\frac{4}{n^2(n-1)^2} \nn \\
&\times&\mathbb{E}\left[\exp\left(\frac{2\iota z}{n\sigma_{n,T}}\left\{g_{j,T}(\psi,P_{\theta_0})+g_{k,T}(\psi,P_{\theta_0})+g_{m,T}(\psi,P_{\theta_0})+g_{l,T}(\psi,P_{\theta_0})\right \}\right)\left\{\gamma_{j,k,T}(\psi,P_{\theta_0})\gamma_{m,l,T}(\psi,P_{\theta_0})\right. \right. \nn \\
&+& \left. \left. \gamma_{j,m,T}(\psi,P_{\theta_0})\gamma_{k,l,T}(\psi,P_{\theta_0})+\gamma_{j,l,T}(\psi,P_{\theta_0})\gamma_{k,m,T}(\psi,P_{\theta_0}) \right\}\right. \bigg]\nn \\
   &+& O(\vert n^{-1/2} z  \vert ^{2+\delta}+\vert n^{-1}z \vert ).
   \end{eqnarray}
 \end{footnotesize}  
\item[Step 2.] To match the expression of $\Psi^*_{n,T}(z)$ as in (\ref{Eq: AppCF}), we need an expansion for each of the terms in (\ref{Eq: cfExpan1}). To this end, we work on the exponential terms in (\ref{Eq: cfExpan1}). Here, we focus on the first exponential term and get (\ref{Eq: 1stExp}). We can repeat the computations for the other terms, and  those tedious developments follow similar arguments.
   We expand the first exponential term in (\ref{Eq: cfExpan1}) by using (2.7) in \citet{BGVZ86}, (\ref{Eq: variance}) and Assumption D. Thus,  we obtain
 \begin{footnotesize}
   \begin{eqnarray} \label{Eq: 1stExp}
        & &\mathbb{E}\left[\exp\left(\iota z\sigma_{n,T}^{-1}\frac{2}{n}\left\{g_{j,T}(\psi,P_{\theta_0})+g_{k,T}(\psi,P_{\theta_0})\right\}\right)\frac{2}{n(n-1)}\gamma_{j,k,T}(\psi,P_{\theta_0})\right] \nn \\
        &=& \mathbb{E}\left[ \left(\exp\left\{\iota z\sigma_{n,T}^{-1}\frac{2}{n}g_{j,T}(\psi,P_{\theta_0})\right\} -1-\iota z\sigma_{n,T}^{-1}\frac{2}{n}g_{j,T}(\psi,P_{\theta_0}) \right)  \right. \nn \\
        &\times& \left. \left(\exp\left\{\iota z\sigma_{n,T}^{-1}\frac{2}{n}g_{k,T}(\psi,P_{\theta_0})\right\} -1-\iota z\sigma_{n,T}^{-1}\frac{2}{n}g_{k,T}(\psi,P_{\theta_0}) \right) \frac{2}{n(n-1)}\gamma_{j,k,T}(\psi,P_{\theta_0})\right] \nn \\
        &+&  \mathbb{E}\left[ \iota z\sigma_{n,T}^{-1} \frac{4g_{k,T}(\psi,P_{\theta_0})\gamma_{j,k,T}(\psi,P_{\theta_0})}{n^2(n-1)} \left(\exp\left\{\iota z\sigma_{n,T}^{-1}\frac{2}{n}g_{j,T}(\psi,P_{\theta_0})\right\}- \sum_{\nu=0}^{2} \frac{\left\{\iota z\sigma_{n,T}^{-1}\frac{2}{n}g_{j,T}(\psi,P_{\theta_0})\right\}^{\nu}}{\nu!}  \right)\right]\nn \\
        &+&   \mathbb{E}\left[ \iota z\sigma_{n,T}^{-1} \frac{4g_{j,T}(\psi,P_{\theta_0}) \gamma_{j,k,T}(\psi,P_{\theta_0})}{n^2(n-1)} \left(\exp\left\{\iota z\sigma_{n,T}^{-1}\frac{2}{n}g_{k,T}(\psi,P_{\theta_0})\right\}- \sum_{\nu=0}^{2} \frac{\left\{\iota z\sigma_{n,T}^{-1}\frac{2}{n}g_{k,T}(\psi,P_{\theta_0})\right\}^{\nu}}{\nu!} \right)\right]\nn \\
        &-& \mathbb{E}\left[ z^2 \sigma_{n,T}^{-2}\frac{4}{n^2}g_{j,T}(\psi,P_{\theta_0})g_{k,T}(\psi,P_{\theta_0})\frac{2}{n(n-1)}\gamma_{j,k,T}(\psi,P_{\theta_0})\right] \nn \\
        &-& \mathbb{E}\left[ \left(\iota z^3 \sigma_{n,T}^{-3}\frac{4}{n^3}g_{j,T}^2(\psi,P_{\theta_0})g_{k,T}(\psi,P_{\theta_0}) +\iota z^3 \sigma_{n,T}^{-3}\frac{4}{n^3}g_{j,T}(\psi,P_{\theta_0})g_{k,T}^2(\psi,P_{\theta_0})\right) \frac{2}{n(n-1)}\gamma_{j,k,T}(\psi,P_{\theta_0})\right] \nn \\
        &=& -z^2 \sigma_{n,T}^{-2}\frac{8}{n^3(n-1)}\mathbb{E}\left[ g_{j,T}(\psi,P_{\theta_0})g_{k,T}(\psi,P_{\theta_0})\gamma_{j,k,T}(\psi,P_{\theta_0})\right] \nn \\
        &-& \iota z^3 \sigma_{n,T}^{-3}\frac{8}{n^4(n-1)} \mathbb{E}\left[ \left\{g_{j,T}^2(\psi,P_{\theta_0})g_{k,T}(\psi,P_{\theta_0}) +g_{j,T}(\psi,P_{\theta_0})g_{k,T}^2(\psi,P_{\theta_0})\right\} \gamma_{j,k,T}(\psi,P_{\theta_0})\right] \nn \\
        &+& O(n^{-4}z^{4}+ n^{-2}\vert n^{-1/2}z \vert ^{3+\delta}),
   \end{eqnarray}
 \end{footnotesize}  
with $\delta \in (0,1]$.
   Similarly, we expand all the other exponentials in (\ref{Eq: cfExpan1}) and after some algebraic simplifications, we get 
  \begin{footnotesize}
   \begin{eqnarray}  
        &&\Psi_{n,T}(z) \nn \\  
        &=& \left\{\prod_{i=1}^{n} \Psi_{g,i,T}(z) \right\}
        +\sum_{j=1}^{n-1}\sum_{k=j+1}^{n}\left\{\prod_{\substack{i=1 \\ i\ne j,k}}^{n} \Psi_{g,i,T}(z)\right\}\left(\frac{-8\iota z^3 \mathbb{E}\left[ g_{j,T}(\psi,P_{\theta_0})g_{k,T}(\psi,P_{\theta_0})\gamma_{j,k,T}(\psi,P_{\theta_0})\right] }{n^3(n-1)\sigma_{n,T}^{3}}\right. \nn \\
        &+&\left. z^4 \sigma_{n,T}^{-4}\frac{8 \mathbb{E}\left[ \left\{g_{j,T}^2(\psi,P_{\theta_0})g_{k,T}(\psi,P_{\theta_0}) +g_{j,T}(\psi,P_{\theta_0})g_{k,T}^2(\psi,P_{\theta_0})\right\} \gamma_{j,k,T}(\psi,P_{\theta_0})\right]}{n^4(n-1)} \right. \nn \\
        &-& \left. \frac{2}{n^2(n-1)^2}z^2\sigma_{n,T}^{-2}\mathbb{E}\left[\gamma_{j,k,T}^2(\psi,P_{\theta_0})\right] \right) \nn \\   
      &+& \sum_{j=1}^{n-1}\sum_{k=j+1}^{n}\sum_{\substack{m=1 \\ m\ne j,k}}^{n}\left\{\prod_{\substack{i=1 \\ i\ne j,k,m}}^{n} \Psi_{g,i,T}(z)\right\}\frac{16z^4\mathbb{E}\left[g_{j,T}(\psi,P_{\theta_0})g_{k,T}(\psi,P_{\theta_0})\gamma_{j,m,T}(\psi,P_{\theta_0})\gamma_{k,m,T}(\psi,P_{\theta_0})\right] }{n^4(n-1)^2\sigma_{n,T}^{4} }\nn \\     
        &-&  \sum_{j=1}^{n-3}\sum_{k=j+1}^{n-2}\sum_{m=k+1}^{n-1}\sum_{l=m+1}^{n}\left\{\prod_{\substack{i=1 \\ i\ne j,k,m,l}}^{n} \Psi_{g,i,T}(z)\right\}z^6\sigma_{n,T}^{-6}\frac{64}{n^6(n-1)^2}  \nn \\
        &\times& \left\{\mathbb{E}\left[ g_{j,T}(\psi,P_{\theta_0})g_{k,T}(\psi,P_{\theta_0})\gamma_{j,k,T}(\psi,P_{\theta_0})\right]\mathbb{E}\left[ g_{m,T}(\psi,P_{\theta_0})g_{l,T}(\psi,P_{\theta_0})\gamma_{m,l,T}(\psi,P_{\theta_0})\right] \right. \nn \\
        &+& \left. \mathbb{E}\left[ g_{j,T}(\psi,P_{\theta_0})g_{m,T}(\psi,P_{\theta_0})\gamma_{j,m,T}(\psi,P_{\theta_0})\right]\mathbb{E}\left[ g_{k,T}(\psi,P_{\theta_0})g_{l,T}(\psi,P_{\theta_0})\gamma_{k,l,T}(\psi,P_{\theta_0})\right] \right. \nn \\
        &+& \left. \mathbb{E}\left[ g_{j,T}(\psi,P_{\theta_0})g_{l,T}(\psi,P_{\theta_0})\gamma_{j,l,T}(\psi,P_{\theta_0})\right]\mathbb{E}\left[ g_{k,T}(\psi,P_{\theta_0})g_{m,T}(\psi,P_{\theta_0})\gamma_{k,m,T}(\psi,P_{\theta_0})\right] \right\} \nn \\
        &+& O( \prod_{i=1}^{n} \Psi_{g,i,T}(z) \vert z\vert \mathcal{P}(\vert z \vert)n^{-1-\delta/2}+\vert n^{-1/2} z  \vert ^{2+\delta}+\vert n^{-1}z \vert), \label{Eq: CF_U_Exp}
   \end{eqnarray}
 \end{footnotesize} 
   where $\mathcal{P}$ is a fixed polynomial.
\item[Step 3.] 
We need to derive the expansions (up to a suitable order) of the products of $\Psi_{g,i,T}(z)$ represented as the four curly brackets in (\ref{Eq: CF_U_Exp}), to have similar expressions as the terms in $\Psi^*_{n,T}(z)$; see  (\ref{Eq: AppCF}). To achieve it, we introduce the approximate c.f. of $\sigma_{g}^{-1}\sum_{i=1}^ng_{i,T}(\psi,P_{\theta_0})$ in (\ref{Eq: Psii}) and find a connection to the c.f. of $\sigma_{n,T}^{-1}\sum_{i=1}^ng_{i,T}(\psi,P_{\theta_0})$ so that we can get the expressions of the four curly brackets in (\ref{Eq: CF_U_Exp}).

   Let 
   \begin{equation}  \label{Eq: CF_g}
       \Psi_{i,T}(z) =\mathbb{E}\left[\exp\left( \iota z\sigma_{g}^{-1}g_{i,T}(\psi,P_{\theta_0}) \right)\right]
   \end{equation}
   denote the c.f. of $\sigma_{g}^{-1}g_{i,T}(\psi,P_{\theta_0})$, where $\sigma_{g}^2$ is defined in (\ref{Eq: var_g}). For sufficient small $\varepsilon'>0$ and for $\vert z\vert \leq \varepsilon' n^{1/2}$, we get for the c.f. of $\sigma_{g}^{-1}\sum_{i=1}^{n}g_{i,T}(\psi,P_{\theta_0})$ 
   \begin{equation} \label{Eq: Psii}
       \prod_{i=1}^n\Psi_{i,T}(n^{-1/2}z) = e^{-z^2/2}\left[ 1-\frac{\iota \tilde{\kappa}_3}{6}n^{-1/2}z^3+\frac{\tilde{\kappa}_4}{24}n^{-1}z^4-\frac{\tilde{\kappa}_3^2}{72}n^{-1}z^6\right]+o(n^{-1}\vert z\vert e^{-z^2/4}),
    \end{equation}
   where $\tilde{\kappa}_3=n^{-1}\sigma_{g}^{-3}\sum_{i=1}^n \mathbb{E}\left[g_{i,T}^3(\psi,P_{\theta_0}) \right]$ and $\tilde{\kappa}_4=n^{-1}\sigma_{g}^{-4}\sum_{i=1}^n \mathbb{E}\left[g_{i,T}^4(\psi,P_{\theta_0}) \right]-3$.
   
   Since $\Psi_{g,i,T}(z)=\Psi_{i,T}(\sigma_g\sigma_{n,T}^{-1}\frac{2}{n}z)$, we can investigate the behaviour of the four curly brackets in (\ref{Eq: CF_U_Exp}), namely  {\footnotesize
   \begin{eqnarray} \label{Eq: gP}
       \prod^{n}_{i=1} \Psi_{g,i,T}(z) &=& \prod_{i=1}^{n} \Psi_{i,T}(n^{-1/2}z)+e^{-z^2/2}\left[\frac{1}{n(n-1)^2}\sigma^{-2}_{g}\sum_{u=1}^{n-1}\sum_{v=u+1}^{n}\mathbb{E}\left[\gamma_{u,v,T}^2(\psi,P_{\theta_0})\right]\right]z^2 \nn \\
       &+&o(n^{-1}\vert z\vert e^{-z^2/4}),
   \end{eqnarray}
   \begin{eqnarray}  \label{Eq: gP2}
       \prod^{n}_{\substack{i=1 \\ i\ne j,k}} \Psi_{g,i,T}(z)& =& \prod_{i=1}^{n} \Psi_{i,T}(n^{-1/2}z)\nn \\ &+&e^{-z^2/2}\left[\sum_{u=1}^{n-1}\sum_{v=u+1}^{n}\frac{\mathbb{E}\left[\gamma_{u,v,T}^2(\psi,P_{\theta_0})\right]}{n(n-1)^2\sigma^{2}_{g}}+\frac{\mathbb{E}\left[g_{j,T}^2(\psi,P_{\theta_0})\right]+\mathbb{E}\left[g_{k,T}^2(\psi,P_{\theta_0})\right]}{n\sigma_{g}^{2}}\right] z^2\nn \\ 
       &+& o(n^{-1}\vert z\vert e^{-z^2/4}),
   \end{eqnarray}
   \begin{eqnarray}  \label{Eq: gP3}
      && \prod^{n}_{\substack{i=1 \\ i\ne j,k,m}} \Psi_{g,i,T}(z)\nn \\&=& \prod_{i=1}^{n} \Psi_{i,T}(n^{-1/2}z)\nn \\ 
       &+&e^{-z^2/2}\left[\sum_{u=1}^{n-1}\sum_{v=u+1}^{n}\frac{\mathbb{E}\left[\gamma_{u,v,T}^2(\psi,P_{\theta_0})\right]}{n(n-1)^2\sigma^{2}_{g}} 
       +  \frac{\mathbb{E}\left[g_{j,T}^2(\psi,P_{\theta_0})\right]+\mathbb{E}\left[g_{k,T}^2(\psi,P_{\theta_0})\right]+\mathbb{E}\left[g_{m,T}^2(\psi,P_{\theta_0})\right]}{n\sigma_{g}^{2}}\right] z^2\nn \\ 
       &+& o(n^{-1}\vert z\vert e^{-z^2/4}),
   \end{eqnarray}
   \begin{eqnarray}    \label{Eq: gP4}
      \prod_{\substack{i=1 \\ i\ne j,k,m,l}}^{n} \Psi_{g,i,T}(z) 
       &=& \prod_{i=1}^{n} \Psi_{i,T}(n^{-1/2}z)\nn \\ 
       &+&e^{-z^2/2}\left[\sum_{u=1}^{n-1}\sum_{v=u+1}^{n}\frac{\mathbb{E}\left[\gamma^2_{u,v,T}(\psi,P_{\theta_0})\right]}{n(n-1)^2\sigma^{2}_{g}}\right. \nn\\
      & +&\left. \frac{\mathbb{E}\left[g_{j,T}^2(\psi,P_{\theta_0})\right]+\mathbb{E}\left[g_{k,T}^2(\psi,P_{\theta_0})\right]+\mathbb{E}\left[g_{m,T}^2(\psi,P_{\theta_0})\right]+\mathbb{E}\left[g_{l,T}^2(\psi,P_{\theta_0})\right]}{n\sigma_{g}^{2}}\right]z^2 \nn \\ 
       &+& o(n^{-1}\vert z\vert e^{-z^2/4}),
   \end{eqnarray}
   }
   for $\vert z\vert \leq \varepsilon' n^{1/2}$. 
\item[Step 4.] We combine the remainders and derive an expression for $\Psi_{n,T}(z)$ such that $\Psi^*_{n,T}(z)$ is the leading term and we characterize the order of the remainder. This yields (\ref{Eq: Psi_nt}).

Substitution of (\ref{Eq: Psii}), (\ref{Eq: gP}), (\ref{Eq: gP2}), (\ref{Eq: gP3}), (\ref{Eq: gP4}), (\ref{Eq: variance}), and (\ref{Eq: AppCF}) into (\ref{Eq: CF_U_Exp}) shows that for $\vert z\vert \leq \varepsilon' n^{1/2}$,
   \begin{equation}  \label{Eq: Psi_nt}
       \Psi_{n,T}(z) = \Psi_{n,T}^{*}(z)+o(n^{-1}\vert z\vert \mathcal{P}(\vert z \vert) e^{-z^2/4})+O(\vert n^{-1/2} z  \vert ^{2+\delta}),
   \end{equation}
   the same as (2.13) in \citet{BGVZ86}. 
\item[Step 5.] Moving along the lines of (2.13) in  \citet{BGVZ86}, we prove (\ref{InE: ES}).
\end{itemize}
\end{proof}

\subsection{Proof of Proposition 4
} \label{App: proof_sadd}

\begin{proof}

We derive (4.13) 
by the tilted-Edgeworth technique; see e.g. \citet{BNSC89} for a book-length presentation. Our proof follows from standard arguments, like e.g. those in \citet{F82}, \citet{ER86}, and \citet{GR96}. The main difference between the available proofs and ours is related to the fact that we need to use our approximate c.g.f., as obtained 
via the (approximate) cumulants in (\ref{Eq: Mean}), (\ref{Eq: variance}), (\ref{Eq: 3rdk}) and (\ref{Eq: 4rdk}). To this end, we set 
\begin{equation}
\tilde{\mathcal{K}}_{n,T}(\nu)= \mu_{n,T} \nu + \frac{1}{2} n  \sigma_{n,T}^{2} \nu^2 +  \frac{1}{6} n^2  \varkappa_{n,T}^{(3)} \sigma_{n,T}^{3} \nu^3 
  + \frac{1}{24} n^3  \varkappa_{n,T}^{(4)}  \sigma_{n,T}^{4} \nu^4,
\label{Eq: rtilde}
\end{equation}
where we use the cumulants
$\varkappa_{n,T}^{(3)} = n^{-1/2} \kappa_{n,T}^{(3)}$,  $\varkappa_{n,T}^{(4)} = n^{-1} \kappa_{n,T}^{(4)}$,
with $\varkappa_{n,T}^{(3)}$ and $\varkappa_{n,T}^{(4)}$ being of order $O(m^{-1})$, as derived in Lemma \ref{Lemma_cgf}. Then, following the argument of Remark 2 in \citet{ER86}, we obtain  the required result
$f_{n,T}(z) = 
p_{n,T}(z)\left[1 + O\left(m^{-1}\right)\right]$; see also \citet{F82}, p. 677. 
Finally, 
a straightforward application of Lugannani-Rice formula   yields (4.15); 
see 
\citet{LR80} and \citet{GR96}. 

\end{proof}


\section{The first and second derivatives of the log-likelihood} \label{Section: Appendix2}

\citet{LY10} have already provided a few calculations for the first-order asymptotics. To go further, our online materials give additional and more explicit mathematical expressions for the higher-order terms needed for the saddlepoint approximation.

\noindent We recall the following notations, which are frequently used:
\begin{eqnarray*}
S_n(\lambda) &=& I_n -\lambda W_n \\
\Rn &=& I_n -\rho M_n \\
G_n(\lambda) &=& W_n S_n^{-1} \\
\Hn &=& M_n R_n^{-1} \\
\ddot{W}_n &=& R_n W_n R_n^{-1}, \\
\ddot{G}_n(\lambda_0) &=& \ddot{W}_n(I_n-\lambda_0 \ddot{W}_n)^{-1}= R_nG_nR_n^{-1}, \\
\ddot{X}_{nt} &=& R_n \tXn  \\
H_n^s &=& H_n^{'}+H_n,\\ 
G_n^s &=& G_n^{'}+G_n \\
\hbar\nt(\zeta) &=& \frac{1}{m} \sum_{t=1}^{T} \left( \ddot{X}\nt, \ddot{G}_n(\lambda_0) \ddot{X}\nt \beta_0\right)' \left( \ddot{X}\nt, \ddot
{G}_n(\lambda_0) \ddot{X}\nt \beta_0 \right)
\end{eqnarray*}



\subsection{The first derivative of the log-likelihood} 

\label{AppendixA} 

\subsubsection{Common terms}
 First, consider the following elements which are common to many partial derivatives that we are going to compute. To this end, we set $\xi=(\beta',\lambda,\rho)'$ and we compute:
\begin{itemize}
\item the matrix
\begin{eqnarray}
\partial_{\lambda} S_n(\lambda)=-W_n \label{Eq. dSlambda}\\
\partial_{\rho} R_n(\rho)=-M_n \label{Eq. dRrho}
\end{eqnarray}
\item the vector 
\begin{equation*}
\partial_{\xi}\tVn= (\partial_{\beta'}\tVn,\partial_{\lambda}\tVn,\partial_{\rho}\tVn), 
\end{equation*}
where 
\begin{itemize}

\item[] 
\begin{eqnarray} 
\partial_{\beta'}\tVn &=& \partial_{\beta'} \left\{\SnY \right\} = -\Rn \tXn \label{Eq. dVbeta}
\end{eqnarray}

\item[] and 
\begin{eqnarray} 
\partial_{\lambda}\tVn &=& \partial_{\lambda} \left\{\Rn S_n(\lambda)\tYn \right\} = -\Rn W_n \tilde{Y}_{nt} \label{Eq. dVlambda}
\end{eqnarray}

\item[] and making use of (\ref{Eq. dRrho}), we have
\begin{eqnarray} 
\partial_{\rho}\tVn &=& \partial_{\rho} \left\{\SnY \right\} \notag \\  
&=&  -M_n[S_n(\lambda)\tilde{Y}_{nt}-\tilde{X}_{nt}\beta]\nn \\ &=&  -M_nR_n^{-1}(\rho)\Rn[S_n(\lambda)\tilde{Y}_{nt}-\tilde{X}_{nt}\beta] \nn \\ &=& -H_n(\rho)\tVn  \label{Eq. dVrho},
\end{eqnarray}

\end{itemize}

\item the vector
\begin{equation*}
 \partial_{\xi} G_n(\lambda) = (\partial_{\beta'}G_{n}(\lambda),\partial_{\lambda}G_n(\lambda),\partial_{\rho}G_n(\lambda)), 
\end{equation*}
where 
\begin{itemize}

\item[] 
\begin{eqnarray} 
\partial_{\beta'} G_n(\lambda) = 0 \label{Eq. dGnbeta}\\
\partial_{\rho} G_n(\lambda) =0      \label{Eq. dGnrho}
\end{eqnarray}

\item[] and 
\begin{eqnarray} 
\partial_{\lambda} G_n(\lambda) &=&  \partial_{\lambda}(W_nS_n^{-1})= W_n \partial_{\lambda}S_n^{-1} \nn \\&=& W_n (- S_n^{-1}\underbrace{\partial_{\lambda}(S_n)}_{\ref{Eq. dSlambda}}S_n^{-1} ) \nn \\
 &=& (W_nS_n^{-1})^2 = G_n^2 \label{Eq. dGnlambda}
\end{eqnarray}

\end{itemize}
\item the vector 
\begin{equation*}
 \partial_{\xi} H_n(\rho) = (\partial_{\beta'}H_{n}(\rho),\partial_{\lambda}H_{n}(\rho),\partial_{\rho}H_{n}(\rho)), 
\end{equation*}
where 
\begin{itemize}
\item[] 
\begin{eqnarray} 
\partial_{\beta'}H_n(\rho) = 0 \label{Eq. dHnbeta}\\
\partial_{\lambda}H_{n}(\rho) = 0 \label{Eq. dHnlambda}
\end{eqnarray}

\item[] and 
\begin{eqnarray} 
\partial_{\rho}H_{n}(\rho) &=& \partial_{\rho}(M_nR_n^{-1})= M_n \partial_{\rho}R_n^{-1} \nn \\&=& M_n (- R_n^{-1}\underbrace{\partial_{\rho}(R_n)}_{\ref{Eq. dRrho}}R_n^{-1} ) \nn \\
&=&(M_n R_n^{-1}(\rho))^2=H_n^2     \label{Eq. dHnrho},
\end{eqnarray}

\end{itemize}
 
\item the vector
\begin{equation*}
\partial_{\xi}\qGn = (\partial_{\beta'}\qGn,\partial_{\lambda}\qGn,\partial_{\rho}\qGn), 
\end{equation*}
where 
\begin{itemize}
\item[] 
\begin{eqnarray} 
\partial_{\lambda}\qGn &=& \partial_{\lambda}\{ G_n(\lambda)G_n(\lambda) \} = \underbrace{\partial_{\lambda} G_n(\lambda)}_{\ref{Eq. dGnlambda}} G_n(\lambda) + G_n(\lambda) \underbrace{\partial_{\lambda} G_n(\lambda)}_{\ref{Eq. dGnlambda}}  \nn \\ 
&= & \qGn \Gn + \Gn \qGn \notag\\ &=&2G_n^3(\lambda)\label{Eq. dqGnlambda}
\end{eqnarray}

\item[] $\partial_{\beta'}\qGn =0 \ \text{and}  \  \partial_{\rho}\qGn = 0$

\end{itemize}
\item
\begin{equation*}
\partial_{\xi} H^2_n(\rho) = (\partial_{\beta'}H^2_{n}(\rho),\partial_{\lambda}H^2_{n}(\rho),\partial_{\rho}H^2_{n}(\rho)), 
\end{equation*}
 where 
\begin{itemize}
\item[] $\partial_{\beta'} H^2_n(\rho) = 0, \quad \partial_{\lambda} H^2_n(\rho) = 0$, 
\item[] 
\begin{eqnarray} 
\partial_{\rho} H^2_n(\rho) &=& \partial_{\rho} \{H_n(\rho)H_n(\rho)\}= \underbrace{\partial_{\rho} H_n(\rho) }_{\ref{Eq. dHnrho}}H_n(\rho) + H_n(\rho) \underbrace{\partial_{\rho} H_n(\rho)}_{\ref{Eq. dHnrho}}  \nn \\
&=&  H_n(\rho) ^2 H_n(\rho) + H_n(\rho)  H_n(\rho)^2 \nn 
 \\& =&2H_n^3(\rho) \label{Eq. dqHnrho}
\end{eqnarray}

\end{itemize}

\end{itemize}

\subsubsection{Component-wise calculation of  the log-likelihood} 
 \begin{equation*} 
 \frac{\partial\ell_{n,T}(\theta)}{\partial {\theta} }
 = \{ \partial_{\beta'}\ell_{n,T}(\theta),\partial_{\lambda}\ell_{n,T}(\theta),\partial_{\rho}\ell_{n,T}(\theta),\partial_{\sigma^2}\ell_{n,T}(\theta)   \}
\end{equation*}
\begin{itemize}
\item \begin{eqnarray}
\partial_{\beta'}\ell_{n,T}(\theta) &=& \partial_{\beta'}\{-\frac{1}{2\sigma^2} \sum_{t=1}^{T} \rVn\tVn \} \nn \\
&=& -\frac{1}{2\sigma^2} \sum_{t=1}^{T}(\underbrace{\partial_{\beta'}\tVn}_{\ref{Eq. dVbeta}})'\tVn -\frac{1}{2\sigma^2} \sum_{t=1}^{T}\left( \rVn \underbrace{\partial_{\beta'}\tVn}_{\ref{Eq. dVbeta}}\right)^{'} \nn \\ &=& \frac{1}{2\sigma^2} \sum_{t=1}^{T}(R_n(\rho)\tXn)'\tVn + \frac{1}{2\sigma^2}\sum_{t=1}^{T}\left(\rVn R_n(\rho)\tXn \right)^{'}\nn \\
&=& \frac{1}{\sigma^2} \sum_{t=1}^{T}(R_n(\rho)\tXn)'\tVn
\end{eqnarray}

\item \begin{eqnarray}
\partial_{\lambda}\ell_{n,T} (\theta)&=& \partial_{\lambda}\{(T-1)\ln|S_n(\lambda)|-\frac{1}{2\sigma^2} \sum_{t=1}^{T} \rVn\tVn \} \nn \\ &=& (T-1)\tr(S_n^{-1}(\lambda)\underbrace{\partial_{\lambda}S_n(\lambda)}_{\ref{Eq. dSlambda}}) \nn \\
&-& \frac{1}{2\sigma^2}\sum_{t=1}^{T}\{(\underbrace{\partial_{\lambda}\tVn}_{\ref{Eq. dVlambda}})'\tVn 
+ \rVn \underbrace{\partial_{\lambda}\tVn}_{\ref{Eq. dVlambda}}\} \nn \\
&= & -(T-1)\tr(S_n^{-1}(\lambda)W_n)\nn \\
&+& \frac{1}{2\sigma^2}\sum_{t=1}^{T}\{ (\Rn W_n \tYn)'\tVn + \rVn \Rn W_n \tYn  \} \nn \\
&=& -(T-1)\tr(G_n(\lambda))+ \frac{1}{\sigma^2}\sum_{t=1}^{T} (\Rn W_n \tYn)'\tVn
\end{eqnarray}

\item \begin{eqnarray}
\partial_{\rho}\ell_{n,T}(\theta) &=& \partial_{\rho}\{(T-1)\ln|\Rn|-\frac{1}{2\sigma^2} \sum_{t=1}^{T} \rVn\tVn \} \nn \\ &=& (T-1)\tr(R_n^{-1}(\rho)\underbrace{\partial_{\rho}\Rn}_{\ref{Eq. dRrho}})\nn \\
&-& \frac{1}{2\sigma^2}\sum_{t=1}^{T}\{(\underbrace{\partial_{\rho}\tVn}_{\ref{Eq. dVrho}})'\tVn + \rVn \underbrace{\partial_{\rho}\tVn}_{\ref{Eq. dVrho}}\} \nn \\ &=& -(T-1)\tr(R_n^{-1}(\rho)M_n)\nn \\
& +& \frac{1}{2\sigma^2}\sum_{t=1}^{T}\{ (H_n(\rho) \tVn)'\tVn + \rVn  H_n(\rho) \tVn \} \nn \\ &=& -(T-1)\tr(H_n(\rho))+\frac{1}{\sigma^2}\sum_{t=1}^{T}(H_n(\rho) \tVn)'\tVn
\end{eqnarray}

\item \begin{eqnarray}
\partial_{\sigma^2}\ell_{n,T} (\theta) &=& \partial_{\sigma^2}\{ -\frac{n(T-1)}{2}\ln(2\pi\sigma^2)- \frac{1}{2\sigma^2} \sum_{t=1}^{T} \rVn\tVn \} \nn \\ 
&=& -\frac{n(T-1)}{2\sigma^2}+ \frac{1}{2\sigma^4}\sum_{t=1}^{T} \rVn\tVn 
\end{eqnarray}
\end{itemize}
\begin{equation} 
 \frac{\partial \ell_{n,T}(\theta)}{\partial \theta} = \left( \begin{array}{cc} 
     \frac{1}{\sigma^2} \sum_{t=1}^{T}(R_n(\rho)\tXn)'\tVn \\ 
     -(T-1)\tr(G_n(\lambda))+ \frac{1}{\sigma^2}\sum_{t=1}^{T} (\Rn W_n \tYn)'\tVn \\
     -(T-1)\tr(H_n(\rho))+\frac{1}{\sigma^2}\sum_{t=1}^{T}(H_n(\rho) \tVn)'\tVn \\ 
     -\frac{n(T-1)}{2\sigma^2}+ \frac{1}{2\sigma^4}\sum_{t=1}^{T} \rVn\tVn 
      \end{array} \right)
\end{equation}
\begin{equation*}
\frac{\partial\ell_{n,T}(\theta)}{\partial {\theta} } = \frac{1}{(T-1)} \sum_{t=1}^{T}  \psi(({Y}\nt, {X}\nt), \theta) = 0.
\end{equation*}
where $\psi(({Y}\nt, {X}\nt), \theta_{n,T})$ represents the likelihood score function and its expression is 
\begin{equation} \label{Eq: score}
\psi({Y}\nt, {X}\nt, \theta) = \left( \begin{array}{cc} 
     \frac{(T-1)}{\sigma^2}  (\Rnr \tilde{X}\nt)' \tilde{V}_{nt}(\zeta) \\ 
     \frac{(T-1)}{\sigma^2}  (\Rnr W_n  \tilde{Y}\nt)' \tilde{V}_{nt}(\zeta) -  \frac{(T-1)^2}{T}\tr(G_n(\lambda)) \\
     \frac{(T-1)}{\sigma^2}   (\Hn  \tilde{V}\nt(\zeta))' \tilde{V}_{nt}(\zeta) - \frac{ (T-1)^2}{T}\tr(H_n(\rho)) \\ 
     \frac{(T-1)}{2\sigma^4}   \left(\tilde{V}'\nt(\zeta)\tilde{V}\nt(\zeta)  -  \frac{n(T-1)}{T} \sigma^2 \right)
      \end{array} \right)
\end{equation}

\subsection{The second derivative  the log-likelihood} 

\label{AppendixB}
\begin{itemize}
\item The first row of $\dfrac{\partial^2 \ell_{n,T}(\theta_0)}{\partial \theta \partial \theta'} $ is $\dfrac{\partial \{\frac{1}{\sigma^2} \sum_{t=1}^{T}(R_n(\rho)\tXn)'\tVn\} }{\partial {\theta'}} $
\begin{eqnarray}
\partial_{\beta'} ( \frac{1}{\sigma^2}\sum_{t=1}^{T}(R_n(\rho)\tXn)'\tVn) &=&\frac{1}{\sigma^2} \sum_{t=1}^{T}(R_n(\rho)\tXn)'\underbrace{\partial_\beta' \tVn}_{\ref{Eq. dVbeta}} \nn \\
&=& -\frac{1}{\sigma^2} \sum_{t=1}^{T}(R_n(\rho)\tXn)'\Rn \tXn
\end{eqnarray}

\begin{eqnarray}
\partial_\lambda ( \frac{1}{\sigma^2}\sum_{t=1}^{T}(R_n(\rho)\tXn)'\tVn) &=&\frac{1}{\sigma^2} \sum_{t=1}^{T}(R_n(\rho)\tXn)'\underbrace{\partial_\lambda \tVn}_{\ref{Eq. dVlambda}} \nn \\
&=& -\frac{1}{\sigma^2} \sum_{t=1}^{T}(R_n(\rho)\tXn)'\Rn W_n \tYn 
\end{eqnarray}
\begin{eqnarray}
\partial_\rho ( \frac{1}{\sigma^2}\sum_{t=1}^{T}(R_n(\rho)\tXn)'\tVn) &=&\frac{1}{\sigma^2} \sum_{t=1}^{T}\{ (\underbrace{\partial _\rho  R_n(\rho)}_{\ref{Eq. dRrho}}\tXn)'\tVn+(R_n(\rho)\tXn)'\underbrace{\partial_\rho \tVn}_{\ref{Eq. dVrho}} \}\nn \\
&=& -\frac{1}{\sigma^2} \sum_{t=1}^{T}\{(M_n\tXn)'\tVn+ (\Rn \tXn)' H_n(\rho)\tVn \} \nn \\ 
\end{eqnarray}

\begin{eqnarray}
\partial_{\sigma^2} ( \frac{1}{\sigma^2}\sum_{t=1}^{T}(R_n(\rho)\tXn)'\tVn) &=&-\frac{1}{\sigma^4} \sum_{t=1}^{T}(R_n(\rho)\tXn)'\tVn
\end{eqnarray}

\item The second row of $\dfrac{\partial^2 \ell_{n,T}(\theta_0)}{\partial \theta \partial \theta'} $ is $\dfrac{\partial \{-(T-1)\tr(G_n(\lambda))+ \frac{1}{\sigma^2}\sum_{t=1}^{T} (\Rn W_n \tYn)'\tVn\}}{\partial \theta'}$. The matrix $\dfrac{\partial^2 \ell_{n,T}(\theta_0)}{\partial \theta \partial \theta'} $ is symmetric. The first element is the transpose of the second one in the first row. So 
\begin{eqnarray*}
\partial_{\beta'} ( -(T-1)\tr(G_n(\lambda))+ \frac{1}{\sigma^2}\sum_{t=1}^{T} (\Rn W_n \tYn)'\tVn) = -\frac{1}{\sigma^2} \sum_{t=1}^{T}(\Rn W_n \tYn)' R_n(\rho)\tXn 
\end{eqnarray*}
\begin{eqnarray}
\partial_\lambda ( -(T-1)\tr(G_n(\lambda))&+& \frac{1}{\sigma^2}\sum_{t=1}^{T} (\Rn W_n \tYn)'\tVn) \nn \\ &=&-(T-1) \tr(\underbrace{ \partial_\lambda G_n(\lambda)}_{\ref{Eq. dGnlambda}}) +\frac{1}{\sigma^2} \sum_{t=1}^{T}(\Rn W_n \tYn)'\underbrace{\partial_\lambda \tVn}_{\ref{Eq. dVlambda}} \nn \\
&=& -(T-1)\tr(G_n^2(\lambda))-\frac{1}{\sigma^2} \sum_{t=1}^{T}(\Rn W_n \tYn)'\Rn W_n \tYn \nn \\
\end{eqnarray}
\begin{eqnarray}
\partial_\rho ( -(T-1)\tr(G_n(\lambda))&+& \frac{1}{\sigma^2}\sum_{t=1}^{T} (\Rn W_n \tYn)'\tVn) \nn \\
&=&\frac{1}{\sigma^2} \sum_{t=1}^{T}\{ (\underbrace{\partial _\rho  R_n(\rho)}_{\ref{Eq. dRrho}}W_n \tYn)'\tVn+(R_n(\rho)W_n \tYn)'\underbrace{\partial_\rho \tVn}_{\ref{Eq. dVrho}} \}\nn \\
&=& -\frac{1}{\sigma^2} \sum_{t=1}^{T}\{(M_n W_n \tYn)'\tVn+ (\Rn W_n \tYn)' H_n(\rho)\tVn \} \nn \\ 
\end{eqnarray}

\begin{eqnarray}
\partial_{\sigma^2} ( -(T-1)\tr(G_n(\lambda))&+& \frac{1}{\sigma^2}\sum_{t=1}^{T} (\Rn W_n \tYn)'\tVn) \nn \\ &=&-\frac{1}{\sigma^4} \sum_{t=1}^{T}(\Rn W_n \tYn)'\tVn
\end{eqnarray}

\item The third row of $\dfrac{\partial^2 \ell_{n,T}(\theta_0)}{\partial \theta \partial \theta'} $ is 
$$\dfrac{\partial\{ -(T-1)\tr(H_n(\rho))+\frac{1}{\sigma^2}\sum_{t=1}^{T}(H_n(\rho) \tVn)'\tVn \} }{\partial {\theta'}}.$$ 
We could get the first two elements from the transpose of the third ones in the first two rows. So we only need to calculate the following two derivatives.
\begin{eqnarray}
&& \partial_\rho (-(T-1)\tr(H_n(\rho))+\frac{1}{\sigma^2}\sum_{t=1}^{T}(H_n(\rho) \tVn)'\tVn ) \nn \\
&&=-(T-1) \tr(\underbrace{ \partial_\rho H_n(\rho)}_{\ref{Eq. dHnrho}}) +\frac{1}{\sigma^2} \sum_{t=1}^{T}\{ (\underbrace{\partial _\rho  H_n(\rho)}_{\ref{Eq. dHnrho}}\tVn)'\tVn \nn \\
&&+(H_n(\rho) \underbrace{\partial_\rho \tVn}_{\ref{Eq. dVrho}})'\tVn+(H_n(\rho) \tVn)'\underbrace{\partial_\rho \tVn}_{\ref{Eq. dVrho}} \}\nn \\
&&=-(T-1)\tr(H_n^2(\rho)) +\frac{1}{\sigma^2} \sum_{t=1}^{T}\{(H_n^2(\rho)\tVn)'\tVn \nn \\
&&-(H_n^2(\rho)\tVn)'\tVn- (H_n(\rho) \tVn)' H_n(\rho)\tVn \} \nn \\ 
&& = -(T-1)\tr(H_n^2(\rho))-\frac{1}{\sigma^2}\sum_{t=1}^{T}(H_n(\rho) \tVn)' H_n(\rho)\tVn 
\end{eqnarray}

\begin{eqnarray}
\partial_{\sigma^2} ( -(T-1)\tr(H_n(\rho))&+&\frac{1}{\sigma^2}\sum_{t=1}^{T}(H_n(\rho) \tVn)'\tVn ) \nn \\ &=&-\frac{1}{\sigma^4} \sum_{t=1}^{T}(H_n(\rho) \tVn)'\tVn
\end{eqnarray}

\item The fourth row of $\dfrac{\partial^2 \ell_{n,T}(\theta_0)}{\partial \theta \partial \theta'} $ is $\dfrac{\partial\{  -\frac{n(T-1)}{2\sigma^2}+ \frac{1}{2\sigma^4}\sum_{t=1}^{T} \rVn\tVn  \} }{\partial {\theta'}} $. We only need to calculate the derivative in respect with $\sigma^2$.

\begin{eqnarray}
\partial_{\sigma^2} (  -\frac{n(T-1)}{2\sigma^2} + \frac{1}{2\sigma^4}\sum_{t=1}^{T} \rVn\tVn )  &=&\frac{n(T-1)}{2\sigma^4}-\frac{1}{\sigma^6} \sum_{t=1}^{T}\rVn\tVn \nn \\
\end{eqnarray}

\end{itemize}

We report the matrix

$$-\frac{\partial^2 \ell_{n,T}(\theta)}{\partial \theta \partial \theta'}=$$
\begingroup
    \fontsize{7pt}{7pt}\selectfont
\begin{eqnarray*}
\left[\begin{array}{cccccc}
\frac{1}{\sigma^{2}} \sum_t (\Rn \tXn)'(\Rn \tXn)  &  * &   * &  * & \\ 
\frac{1}{\sigma^{2}} \sum_t (\Rn W_n \tYn)'(\Rn \tXn) & \frac{1}{\sigma^{2}}\sum_t (\Rn W_n \tYn)'(\Rn W_n \tYn) + (T-1) \text{tr}(\qGn) &   * &  * & \\ 
\frac{1}{\sigma^{2}}\sum_t \{(\Hn \tVn)'(\Rn \tXn) + \tVn' M_n \tXn  \}& \frac{1}{\sigma^{2}}\sum_t \{(\Rn W_n \tYn)'(\Hn \tVn ) +  
(M_n W_n \tYn)' \tVn\} & 0 & 0 \\
\frac{1}{\sigma^{4}} \sum_t \tVn'\Rn\tXn &  \frac{1}{\sigma^{4}} \sum_t (\Rn W_n \tYn)' \tVn & 0&0
\end{array}\right]    \nn \\
\end{eqnarray*}
\endgroup 

\begingroup
    \fontsize{8pt}{8pt}\selectfont
\begin{eqnarray}
+ \left[\begin{array}{lllccc}
0_{k \times k} & 0_{k \times 1} & 0_{k \times 1} & 0_{k \times 1} \\
0_{1 \times k} & 0 & 0 & 0 \\
0_{1 \times k} & 0 & \sigma^{-2} \sum_t (\Hn \tVn)'\Hn\tVn + (T-1)\text{tr}\{ \Hn^2 \} & * \\
0_{1 \times k} & 0 & \sigma^{-4} \sum_t (\Hn \tVn)'\tVn & -\frac{m}{2\sigma^{4}} + \sigma^{-6} \sum_t \tilde{V}_{nt}'(\xi)\tVn  \\
\end{array}\right] & .\nn \\
\label{Eq. secDerln}
\end{eqnarray}
\endgroup

\subsection{The third derivative of the log-likelihood} 

\label{AppendixC}

Assuming, for the third derivative w.r.t . $\theta$ of $\ell_{n,T}(\theta)$, that derivation and integration can be exchanged (namely the dominated 
convergence theorem holds, 
component-wise for in $\theta$, for the third derivative of the log-likelihood), we 
derive $\frac{\partial^2 \ell_{n,T}(\theta)}{\partial \theta \partial \theta'}$ to compute the term $\Gamma(i,j,T,\theta_0)$, for each $i$ and $j
$. To this end, we have:
\begin{itemize}
\item \begin{eqnarray*} 
&&\frac{\partial \left\{ \sigma^{-2} \sum_t (\Rn \tXn)'(\Rn \tXn)\right\}}{\partial\theta} = \\ &&\left[\begin{array}{ccccc}
A^{(1,1,\beta')}(\rho, \sigma^2)  & A^{(1,1,\lambda)}(\rho, \sigma^2) &
A^{(1,1,\rho)}(\rho, \sigma^2)    & A^{(1,1,\sigma^2)}(\rho, \sigma^2) 
\end{array}\right] 
\end{eqnarray*}
where 
\beq
A^{(1,1,\beta')}(\rho, \sigma^2) = 0_{k\times k}, \quad A^{(1,1,\lambda)}(\rho, \sigma^2) = 0,
\eeq
\beq
A^{(1,1,\rho)} (\rho, \sigma^2) =  -2 \sigma^{-2} \sum_{t=1}^{T} \left\{ \tXn' M'_n   \tXn - \rho  \tXn' M'_n M_n  \tXn  \right\}=  -2 \sigma^{-2} \sum_{t=1}^{T}  (M_n   \tXn)'\Rn\tXn   ,
\eeq
and
\beq
A^{(1,1,\sigma^2)} (\rho, \sigma^2)= -\sigma^{-4} \sum_{t=1}^{T} \left\{ (\Rn \tXn)'  (\Rn \tXn)      \right\}.
\eeq
\item 
\begin{eqnarray*} 
&&\frac{\partial \left\{  \sigma^{-2} \sum_t (\Rn W_n \tilde{Y}_{nt})'(\Rn \tXn)\right\}}{\partial\theta} = \\ &&\left[\begin{array}{ccccc}
A^{(2,1,\beta')}(\rho, \sigma^2)  & A^{(2,1,\lambda)}(\rho, \sigma^2) &
A^{(2,1,\rho)} (\rho, \sigma^2)   & A^{(2,1,\sigma^2)}(\rho, \sigma^2) 
\end{array}\right] 
\end{eqnarray*}
where 
\beq
A^{(2,1,\beta')}(\rho, \sigma^2) = 0_{k\times k}, \quad A^{(2,1,\lambda)}(\rho, \sigma^2) = 0,
\eeq
\beq
A^{(2,1,\rho)}(\rho, \sigma^2)  =  \sigma^{-2} \sum_{t=1}^{T} \left\{ (-M_n W_n \tYn)' \Rn \tXn -(\Rn W_n \tYn)'M_n\tXn\right\},
\eeq
and
\beq
A^{(2,1,\sigma^2)}(\rho, \sigma^2) = -\sigma^{-4} \sum_{t=1}^{T} \left\{ (\Rn W_n \tYn)'(\Rn \tXn)     \right\}.
\eeq

\item 
\begin{eqnarray*} 
&&\frac{\partial \left\{ \sigma^{-2} \sum_t (\Rn W_n \tYn)'(\Rn W_n \tYn) + (T-1) \text{tr}(\qGn)\right\} }{\partial\theta} = \\
&&\left[\begin{array}{cccc}
A^{(2,2,\beta')}(\lambda,\rho, \sigma^2)  & A^{(2,2,\lambda)} (\lambda,\rho, \sigma^2) & A^{(2,2,\rho)}(\lambda,\rho, \sigma^2)   & A^{(2,2,\sigma^2)}(\lambda,\rho, \sigma^2) 
\end{array}\right] 
\end{eqnarray*}
where 
\beq
A^{(2,2,\beta')}(\lambda,\rho, \sigma^2) = 0_{k\times k}
\eeq
\beq
A^{(2,2,\lambda)} (\lambda,\rho, \sigma^2) = (T-1) \text{tr} \left\{  \partial_\lambda(\qGn) \right\}= (T-1) \text{tr} \left\{ 2G_n^3(\lambda) \right\},
\eeq
where we make use of  (\ref{Eq. dqGnlambda}), moreover, using (\ref{Eq. dRrho})
\bea
A^{(2,2,\rho)}(\lambda,\rho, \sigma^2) &=& \frac{1}{\sigma^{2}} \sum_{t=1}^{T} \{ (\partial_\rho \Rn W_n \tilde{Y}_{nt})'(\Rn W_n \tYn) \nn \\ &+& (\Rn W_n \tilde{Y}_{nt})'(\partial_\rho \Rn W_n \tYn) \} \nn \\
&=&-\frac{1}{\sigma^{2}} \sum_{t=1}^{T} \{ (M_n W_n \tilde{Y}_{nt})'(\Rn W_n \tYn) \nn \\ &+& (\Rn W_n \tilde{Y}_{nt})'(M_n W_n \tYn)  \} \nn \\
&=& -\frac{2}{\sigma^{2}} \sum_{t=1}^{T}(M_n W_n \tilde{Y}_{nt})'(\Rn W_n \tYn)
\eea
and
\beq
A^{(2,2,\sigma^2)}(\lambda,\rho, \sigma^2)  =   -\frac{1}{\sigma^{4}} \sum_{t=1}^{T} (\Rn M_n \tYn)' \Rn W_n \tYn  ,
\eeq
\item
\begin{eqnarray*} 
\frac{\partial \left\{ \sigma^{-2} \sum_{t=1}^{T} \left( (\Hn \tVn)'(\Rn \tXn) + \sigma^{-2} \sum_{t=1}^{T} \tilde{V}_{nt}'(\xi) M_n \tXn \right) \right\} }{\partial\theta} = \\ 
\left[\begin{array}{cccc}
A^{(3,1,\beta')}(\lambda,\rho, \sigma^2)  & A^{(3,1,\lambda)} (\lambda,\rho, \sigma^2) &
A^{(3,1,\rho)}(\beta',\lambda,\rho, \sigma^2)   & A^{(3,1,\sigma^2)}(\beta',\lambda,\rho, \sigma^2) 
\end{array}\right] 
\end{eqnarray*}
where 
\begin{eqnarray}
A^{(3,1,\beta')}(\lambda,\rho, \sigma^2) &=&  \sigma^{-2} \sum_t \left\{(\Hn \underbrace{ \partial_{\beta^{'}}\tVn}_{\text{\ \ see \ \ } (\ref{Eq. dVbeta})})'(\Rn \tXn) + \underbrace{ \partial_{\beta^{'}}\tilde{V}_{nt}'(\xi)}_{\text{\ \ see \ \ } (\ref{Eq. dVbeta})}    M_n \tXn \right\}  \nn \\
&=& -\frac{1}{\sigma^2} \sum_{t=1}^{T}\left\{(H_n(\rho)\Rn \tXn )'\Rn \tXn+ (\Rn \tXn)'M_n\tXn\right\} \nn \\
\end{eqnarray}

\begin{eqnarray}
A^{(3,1,\lambda)}(\lambda,\rho, \sigma^2) &=& \frac{1}{\sigma^{2}} \sum_{t=1}^{T} \left\{\left( \Hn \underbrace{\partial_\lambda \tVn}_{\text{\ \ see \ \ }( \ref{Eq. dVlambda})} \right)' \Rn \tXn + \left(\underbrace{\partial_\lambda \tVn}_{\text{\ \ see \ \ }( \ref{Eq. dVlambda})}\right)' M_n \tXn \right\} \nn \\
&=&-\frac{1}{\sigma^{2}} \sum_{t=1}^{T} \left\{ (\Hn \Rn W_n\tYn )'\Rn \tXn+ (\Rn W_n\tYn )' M_n \tXn \right\} \nn \\
\end{eqnarray}

where $\partial_\lambda \tVn=-\Rn W_n \tilde{Y}_{nt}$ as in (\ref{Eq. dVlambda}), and
\bea
&&A^{(3,1,\rho)}(\beta',\lambda,\rho, \sigma^2)\nn \\ 
&&= \frac{1}{\sigma^{2}} \sum_{t=1}^{T}
\left\{\left( \underbrace{\partial_\rho \Hn}_{\text{\ \ see \ \ } (\ref{Eq. dHnrho})}  \tVn \right)' \Rn \tXn   
+  \left(\Hn \underbrace{\partial_\rho \tVn}_{\text{\ \ see \ \ } (\ref{Eq. dVrho})}  \right)' \Rn \tXn\right\} \nn  \\ 
&&+\frac{1}{\sigma^{2}}\sum_{t=1}^{T}  \left\{\left( \Hn  \tVn \right)' \underbrace{\partial_\rho\Rn}_{\text{\ \ see \ \ } (\ref{Eq. dRrho})} \tXn   
+\underbrace{ \partial_\rho \tilde{V}_{nt}'(\xi)}_{\text{\ \ see \ \ } (\ref{Eq. dVrho})} M_n \tXn \right\}\nn \\
&&= \frac{1}{\sigma^{2}} \sum_{t=1}^{T} \left(\Hn^2 \tVn \right)'\Rn \tXn - \frac{1}{\sigma^{2}} \sum_{t=1}^{T} \left(\Hn^2 \tVn \right)'\Rn \tXn \nn \\
&&-\frac{1}{\sigma^{2}} \sum_{t=1}^{T} \left(\Hn\tVn \right)'M_n \tXn - \frac{1}{\sigma^{2}} \sum_{t=1}^{T} \left(\Hn \tVn \right)'M_n \tXn \nn \\ 
&&=- \frac{2}{\sigma^{2}} \sum_{t=1}^{T} \left(\Hn \tVn \right)'M_n \tXn 
\eea
\bea
A^{(3,1,\sigma^2)}(\beta',\lambda,\rho, \sigma^2) = -\frac{1}{\sigma^{4}}\sum_{t=1}^{T} \left\{   (\Hn \tVn)'(\Rn \tXn)+ \tilde{V}_{nt}'(\xi) M_n \tXn \right\} \nn \\
\eea
\item 
\begin{eqnarray*} 
\frac{\partial \left\{ \sigma^{-2} \sum_{t=1}^{T} (\Rn W_n \tYn)'(\Hn \tVn ) + \sigma^{-2} \sum_{t=1}^{T}  
(M_n W_n \tYn)' \tVn     \right\} }{\partial\theta} = \\
\left[\begin{array}{ccccc}
A^{(3,2,\beta')}(\lambda,\rho, \sigma^2)  & A^{(3,2,\lambda)} (\lambda,\rho, \sigma^2) &
A^{(3,2,\rho)}(\beta', \lambda,\rho, \sigma^2)   & A^{(3,2,\sigma^2)}(\beta', \lambda,\rho, \sigma^2) 
\end{array}\right] 
\end{eqnarray*}
where 
\begin{eqnarray}
A^{(3,2,\beta')}(\lambda,\rho, \sigma^2) &=& \frac{1}{\sigma^{2}} \sum_{t=1}^{T} \left\{\left( (\Rn W_n \tYn)'(\Hn  \underbrace{\partial_{\beta^{'}}\tVn }_{\text
{\ \ see \ \  (\ref{Eq. dVbeta})}} ) \right)^{'}+  \left(
(M_n W_n \tYn)' \underbrace{\partial_{\beta^{'}}\tVn }_{\text{\ \ see \ \ } (\ref{Eq. dVbeta})}  \right)^{'}   \right\} \nn \\ &=& - \frac{1}{\sigma^{2}} \sum_{t=1}^{T} \left\{(\Hn \Rn \tXn)' \Rn W_n \tYn+(\Rn \tXn)' M_n W_n \tYn  \right\} \nn \\ 
\end{eqnarray}

\begin{eqnarray}
A^{(3,2,\lambda)}(\lambda,\rho, \sigma^2) &=&  \frac{1}{\sigma^{2}}  \sum_{t=1}^{T} \left\{\left( \Rn W_m \tYn \right)' \Hn \underbrace{\partial_\lambda 
\tVn}_{\text{\ \ see \ \ } (\ref{Eq. dVlambda})}+ \left (M_n W_n \tYn \right)' \underbrace{\partial_\lambda \tVn}_{\text{\ \ see \ \ } (\ref{Eq. 
dVlambda})} \right\} \nn \\
&=& -\frac{1}{\sigma^{2}}  \sum_{t=1}^{T}  \left( \Rn W_m \tYn \right)' \Hn \Rn W_n \tYn 
\nn \\ & -& \frac{1}{\sigma^{2}}\sum_{t=1}^{T} \left (M_n W_n \tYn \right)'\Rn W_n \tYn  \nn \\
&=& -\frac{2}{\sigma^{2}} \sum_{t=1}^{T}(M_n W_n \tilde{Y}_{nt})'(\Rn W_n \tYn)
\end{eqnarray}

and
\bea
&&A^{(3,2,\rho)}(\beta',\lambda,\rho, \sigma^2) \nn \\ 
&&= \frac{1}{\sigma^{2}}\sum_{t=1}^{T} \left\{
\left( \underbrace{\partial_\rho \Rn}_{\text{\ \ see \ \ } (\ref{Eq. dRrho})}  W_n \tYn \right)' \Hn \tVn  + 
\left(\Rn W_n \tYn  \right)' \underbrace{\partial_\rho \Hn}_{\text{\ \ see \ \ } (\ref{Eq. dHnrho})} \tVn \right\} \nn  \\ 
&&+  \frac{1}{\sigma^{2}} \sum_{t=1}^{T} \left\{ \left(\Rn W_n \tYn  \right)' \Hn \underbrace{ \partial_\rho \tVn}_{\text{\ \ see \ \ } (\ref{Eq. 
dVrho})}  + \left(M_n W_n \tYn \right)' \underbrace{ \partial_\rho \tVn}_{\text{\ \ see \ \ } (\ref{Eq. dVrho})} \right\} \nn \\ && =- \frac{1}{\sigma^{2}} \sum_{t=1}^{T}\left\{\left(M_n W_n \tYn  \right)'\Hn \tVn - \left(\Rn W_n \tYn  \right)'\Hn^2 \tVn \right\} \nn \\
&& - \frac{1}{\sigma^{2}} \sum_{t=1}^{T}\left\{ \left(\Rn W_n \tYn  \right)'\Hn^2 \tVn  +\left(M_n W_n \tYn \right)'\Hn \tVn   \right\} \nn \\
&& = - \frac{2}{\sigma^{2}} \sum_{t=1}^{T}\left(M_n W_n \tYn \right)'\Hn \tVn 
\eea

\bea
A^{(3,2,\sigma^2)}(\beta',\lambda,\rho, \sigma^2) &=& -\frac{1}{\sigma^{4}} \sum_{t=1}^{T} \left\{(\Rn W_n \tYn)'(\Hn \tVn ) + 
(M_n W_n \tYn)' \tVn      \right\}. \nn \\
\eea
\item 
\begin{eqnarray*}
&&\frac{\partial \left\{\sigma^{-4} \sum_t \tVn'\Rn\tXn\right\}}{\partial\theta} = \\ &&\left[\begin{array}{ccccc}
A^{(4,1,\beta')}(\lambda,\rho, \sigma^2)  & A^{(4,1,\lambda)} (\lambda,\rho, \sigma^2) &
A^{(4,1,\rho)}(\beta',\lambda,\rho, \sigma^2)    & A^{(4,1,\sigma^2)}(\beta',\lambda,\rho, \sigma^2) 
\end{array}\right], 
\end{eqnarray*}
where 
\begin{eqnarray}
A^{(4,1,\beta')}(\lambda,\rho, \sigma^2)
&=&  \frac{1}{\sigma^{4}} \sum_{t=1}^{T} \left\{( \underbrace{\partial_{\beta^{'}}\tilde{V}_{nt}(\xi) }_{\text{\ \ see \ \  (\ref{Eq. dVbeta})}})'\Rn\tXn \right\} \nn \\
&=& -\frac{1}{\sigma^{4}} \sum_{t=1}^{T} (\Rn\tXn)'\Rn\tXn
\end{eqnarray}

\begin{eqnarray}
A^{(4,1,\lambda)}(\lambda,\rho, \sigma^2) &=& \frac{1}{\sigma^{4}}  \sum_{t=1}^{T} \left\{(\underbrace{\partial_\lambda \tilde{V}_{nt}(\xi)}_{\text{\ \ see \ \ } (\ref{Eq. dVlambda})})' \Rn\tXn  \right\} \nn \\
&=& -\frac{1}{\sigma^{4}}  \sum_{t=1}^{T} (\Rn W_n \tYn )'\Rn \tXn 
\end{eqnarray}

and
\bea
A^{(4,1,\rho)}(\beta',\lambda,\rho, \sigma^2) &=& \frac{1}{\sigma^{4}}  \sum_{t=1}^{T} \left\{ (\underbrace{\partial_\rho \tilde{V}_{nt}(\xi)}_{\text{\ \ see \ \ } (\ref{Eq. dVrho})})' \Rn\tXn+  \tilde{V}_{nt}'(\xi)\underbrace{\partial_{\rho}{\Rn}}_{\text{\ \ see \ \ } (\ref{Eq. dRrho})}\tXn\right\} \nn \\
&=& -\frac{1}{\sigma^{4}}  \sum_{t=1}^{T} \left\{ \left(\Hn \tVn\right)'\Rn\tXn+ \rVn M_n \tXn  \right\} \nn \\
\eea
\bea
A^{(4,1,\sigma^2)}(\beta',\lambda,\rho, \sigma^2) &=& -2\sigma^{-6} \sum_{t=1}^{T} \left\{ \tilde{V}_{nt}'(\xi) \Rn \tXn \right\},
\eea
\item 
\begin{eqnarray*}
&&\frac{\partial \left\{\sigma^{-4} \sum_{t=1}^{T} (\Rn W_n \tYn)' \tVn \right\}}{\partial\theta} = \\ &&\left[\begin{array}{ccccc}
A^{(4,2,\beta')}(\lambda,\rho, \sigma^2)  & A^{(4,2,\lambda)} (\lambda,\rho, \sigma^2) &
A^{(4,2,\rho)}(\beta',\lambda,\rho, \sigma^2)    & A^{(4,2,\sigma^2)}(\beta',\lambda,\rho, \sigma^2) 
\end{array}\right], 
\end{eqnarray*}
where 
\begin{eqnarray}
A^{(4,2,\beta')}(\lambda,\rho, \sigma^2) &=& \frac{1}{\sigma^{4}}  \sum_{t=1}^{T}\left((\Rn W_n \tYn)' \underbrace{\partial_{\beta'}\tVn }_{\text{\ \ see \ \ } 
(\ref{Eq. dVbeta})} \right)^{'}\nn \\
&=&-\frac{1}{\sigma^{4}}  \sum_{t=1}^{T}(\Rn \tXn)'   \Rn W_n \tYn
\end{eqnarray}

\begin{eqnarray}
A^{(4,2,\lambda)}(\lambda,\rho, \sigma^2) &=& \frac{1}{\sigma^{4}}  \sum_{t=1}^{T}\left\{\left( \Rn W_n \tYn \right)'  \underbrace{\partial_\lambda \tVn}_{\text{\ \ see \ \ } (\ref{Eq. dVlambda})} \right\} \nn \\
&=&-\frac{1}{\sigma^{4}}  \sum_{t=1}^{T} \left( \Rn W_n \tYn \right)' \Rn W_n \tYn 
\end{eqnarray}

and
\begin{eqnarray}
A^{(4,2,\rho)}(\beta',\lambda,\rho, \sigma^2) &=& \frac{1}{\sigma^{4}} \sum_{t=1}^{T} \left\{
\left( \underbrace{\partial_\rho \Rn}_{\text{\ \ see \ \ } (\ref{Eq. dRrho})}  W_n \tYn \right)'\tVn  + 
\left(\Rn W_n \tYn  \right)' \underbrace{\partial_\rho \tVn}_{\text{\ \ see \ \ } (\ref{Eq. dVrho})}  \right\} \nn  \\
&=& - \frac{1}{\sigma^{4}} \sum_{t=1}^{T}  \left\{\left(M_n W_n \tYn  \right)'\tVn + \left(\Rn W_n \tYn  \right)'  \Hn \tVn \right\} \nn \\
\end{eqnarray}

\bea
A^{(4,2,\sigma^2)}(\beta',\lambda,\rho, \sigma^2) &=& -2\sigma^{-6} \sum_{t=1}^{T} \left\{(\Rn W_n \tYn)' \tVn\right\}.
\eea
\item
\begin{eqnarray*}
&&\frac{\partial \left\{\sigma^{-2} \sum_{t}^{T} (\Hn \tVn)'\Hn\tVn + (T-1)\text{tr}\{ \qHn \} \right\}}{\partial\theta} = \\ &&\left[\begin{array}
{ccccc}
B^{(3,3,\beta')}(\beta',\lambda,\rho, \sigma^2)  & B^{(3,3,\lambda)} (\beta',\lambda,\rho, \sigma^2) &
B^{(3,3,\rho)}(\beta',\lambda,\rho, \sigma^2)    & B^{(3,3,\sigma^2)}(\beta',\lambda,\rho, \sigma^2) 
\end{array}\right], 
\end{eqnarray*}
where, making use of $\partial_{\beta^{'}}\Hn=0$,
\bea
B^{(3,3,\beta')}(\beta',\lambda,\rho, \sigma^2) &=& \frac{1}{\sigma^{2}} \sum_{t=1}^{T} \left\{ \left(\Hn  \underbrace{\partial_{\beta'}\tVn }_{\text{\ \ see \ 
\ } (\ref{Eq. dVbeta})}\right)'\Hn \tVn \right\} \nn \\
&+& \frac{1}{\sigma^{2}}\sum_{t=1}^{T} \left\{ \left( \Hn \tVn \right)' \Hn  \underbrace{\partial_{\beta'}\tVn }_{\text{\ \ see \ \ } (\ref{Eq. dVbeta})} ^{'}
\right\} \nn \\ 
&=& -\frac{1}{\sigma^{2}} \sum_{t=1}^{T} \left( \Hn \Rn \tXn \right)'\Hn \tVn \nn \\
&-& \frac{1}{\sigma^{2}} \sum_{t=1}^{T} \left(\Hn \Rn \tXn  \right)' \Hn \tVn \nn\\
&=& -\frac{2}{\sigma^{2}} \sum_{t=1}^{T} \left( M_n \tXn \right)'\Hn \tVn
\eea
\bea
B^{(3,3,\lambda)}(\beta',\lambda,\rho, \sigma^2) &=&\frac{1}{\sigma^{2}}  \sum_{t=1}^{T}\left\{ \left( \Hn \underbrace{\partial_\lambda \tVn}_{\text{\ \ see \ 
\ } (\ref{Eq. dVlambda})}     \right)' \Hn\tVn      \right\} \nn \\
&+& \frac{1}{\sigma^{2}} \sum_{t=1}^{T}\left\{  \left( \Hn \tVn\right)' \Hn  \underbrace{\partial_\lambda \tVn}_{\text{\ \ see \ \ } (\ref{Eq. 
dVlambda})}  \right\} \nn \\
&=& -\frac{1}{\sigma^{2}} \sum_{t=1}^{T}\left\{  \left(\Hn \Rn W_n \tYn\right)'\Hn \tVn \right\} \nn\\
&-& \frac{1}{\sigma^{2}} \sum_{t=1}^{T}\left\{\left(\Hn \tVn \right)'\Hn \Rn W_n \tYn \right\}
\eea
\bea
B^{(3,3,\rho)}(\beta',\lambda,\rho, \sigma^2) &=& \sigma^{-2} \sum_{t}^{T} (\underbrace{\partial_{\rho}\Hn}_{\text{\ \ see \ \ } (\ref{Eq. dqHnrho})} 
\tVn)'\Hn\tVn \nn \\
&+& \sigma^{-2} \sum_{t}^{T} (\Hn \underbrace{\partial_{\rho}\tVn}_{\text{\ \ see \ \ } (\ref{Eq. dVrho})})'\Hn\tVn  \nn \\
&+& \sigma^{-2} \sum_{t}^{T} (\Hn \tVn)'\underbrace{\partial_{\rho}\Hn}_{\text{\ \ see \ \ } (\ref{Eq. dqHnrho})} \tVn  \nn \\
&+& \sigma^{-2} \sum_{t}^{T} (\Hn \tVn)'\Hn\underbrace{\partial_{\rho}\tVn}_{\text{\ \ see \ \ } (\ref{Eq. dVrho})}  \nn \\
&+& (T-1)\text{tr}\{ \underbrace{\partial_{\rho} \qHn}_{\text{\ \ see \ \ } (\ref{Eq. dqHnrho})} \} \nn \\
&=& \frac{1}{\sigma^2}\sum_{t=1}^{T} (\qHn \tVn)' \Hn \tVn \nn \\
&-&\frac{1}{\sigma^2}\sum_{t=1}^{T} (\qHn \tVn)' \Hn \tVn \nn \\
&+& \frac{1}{\sigma^2}\sum_{t=1}^{T} (\Hn \tVn)' \qHn \tVn \nn \\
&-&  \frac{1}{\sigma^2}\sum_{t=1}^{T} (\Hn \tVn)' \qHn \tVn \nn \\
&+& (T-1)\text{tr}\{ 2\Hn^3 \} \nn \\
&=&  (T-1)\text{tr}\{ 2\Hn^3 \} 
\eea
\bea
B^{(3,3,\sigma^2)}(\beta',\lambda,\rho, \sigma^2) = &=& -\sigma^{-4} \sum_{t=1}^{T} \left\{(\Hn \tVn)' \Hn \tVn\right\}.
\eea
\item
\begin{eqnarray*}
&&\frac{\partial \left\{ \sigma^{-4} \sum_{t=1}^{T} (\Hn \tVn)' \tVn  \right\}}{\partial\theta} = \\ &&\left[\begin{array}{ccccc}
B^{(4,3,\beta')}(\beta',\lambda,\rho, \sigma^2)  & B^{(4,3,\lambda)} (\beta',\lambda,\rho, \sigma^2) &
B^{(4,3,\rho)}(\beta',\lambda,\rho, \sigma^2)    & B^{(4,3,\sigma^2)}(\beta',\lambda,\rho, \sigma^2) 
\end{array}\right], 
\end{eqnarray*}
where
\bea
B^{(4,3,\beta')}(\beta',\lambda,\rho, \sigma^2)  & =& \frac{1}{\sigma^{4} }\sum_{t=1}^{T} \left\{ \left(
\Hn \underbrace{\partial_{\beta'}\tVn }_{\text{\ \ see \ \ } (\ref{Eq. dVbeta})} \right)' \tVn + \left(\left(\Hn \tVn\right)'  \underbrace{\partial_
{\beta'}\tVn }_{\text{\ \ see \ \ } (\ref{Eq. dVbeta})} \right)^{'}    \right\}\nn \\ 
&=& -\frac{1}{\sigma^{4} }\sum_{t=1}^{T} \left\{ \left(
\Hn \Rn \tXn \right)'\tVn +  \left(
\Rn \tXn \right)'\Hn \tVn \right\} \nn \\
\eea
\bea
B^{(4,3,\lambda)}(\beta',\lambda,\rho, \sigma^2)  &=& \frac{1}{\sigma^{4} } \sum_{t=1}^{T} \left\{ \left(
\Hn \underbrace{\partial_{\lambda}\tVn }_{\text{\ \ see \ \ } (\ref{Eq. dVlambda})} \right)' \tVn + \left(\Hn \tVn\right)'  \underbrace
{\partial_{\lambda}\tVn }_{\text{\ \ see \ \ } (\ref{Eq. dVlambda})}     \right\} \nn \\
&=& \frac{1}{\sigma^{4} } \sum_{t=1}^{T} \left\{ \left(
\Hn \Rn W_n \tYn \right)' \tVn \right\}\nn \\
&+&\frac{1}{\sigma^{4} } \sum_{t=1}^{T} \left\{\left(\Hn \tVn\right)' \Rn W_n \tYn \right\} \nn \\
\eea
\bea
B^{(4,3,\rho)}(\beta',\lambda,\rho, \sigma^2)  &=& \frac{1}{\sigma^{4} } \sum_{t=1}^{T} \left\{\left( \underbrace{\partial_\rho \Hn}_{\text{\ \ see \ \ } (\ref
{Eq. dRrho})} \tVn + \Hn \underbrace{\partial_{\rho}\tVn}_{\text{\ \ see \ \ } (\ref{Eq. dVrho})}  \right)' \tVn   \right\} \nn \\
&+& \frac{1}{\sigma^{4} }  \sum_{t=1}^{T} \left\{ \left( \Hn  \tVn \right)' \underbrace{\partial_{\rho}\tVn}_{\text{\ \ see \ \ } (\ref{Eq. dVrho})}   
\right\} \nn \\
&=& \frac{1}{\sigma^{4} }  \sum_{t=1}^{T} \left\{ \left( \qHn  \tVn - \qHn \tVn \right)'\tVn \right\} \nn \\
&-& \frac{1}{\sigma^{4} }  \sum_{t=1}^{T} \left\{ \left( \Hn  \tVn\right)'\Hn \tVn \right\} \nn \\
&=& -\frac{1}{\sigma^{4} }  \sum_{t=1}^{T} \left\{ \left( \Hn  \tVn\right)'\Hn \tVn \right\} 
\eea
\bea
B^{(4,3,\sigma^2)}(\beta',\lambda,\rho, \sigma^2)  &=&  -2\sigma^{-6} \sum_{t=1}^{T} (\Hn \tVn)' \tVn.
\eea
\item
\begin{eqnarray*}
&&\frac{\partial \left\{-\frac{m}{2\sigma^{4}} + \sigma^{-6} \sum_t \tilde{V}_{nt}'(\xi)\tVn \right\}}{\partial\theta}  = \\ &&\left[\begin{array}{ccccc}
B^{(4,4,\beta')}(\beta',\lambda,\rho, \sigma^2)  & B^{(4,4,\lambda)} (\beta',\lambda,\rho, \sigma^2) &
B^{(4,4,\rho)}(\beta',\lambda,\rho, \sigma^2)    & B^{(4,4,\sigma^2)}(\beta',\lambda,\rho, \sigma^2) 
\end{array}\right], 
\end{eqnarray*}
where
\bea
B^{(4,4,\beta')}(\beta',\lambda,\rho, \sigma^2)  & =& \frac{1}{\sigma^{6}} \sum_{t=1}^{T} \left\{ \left(\underbrace{\partial_{\beta'}\tVn }_{\text{\ \ see \ \ } (\ref{Eq. dVbeta})} \right)' \tVn+ \left( \tilde{V}_{nt}'(\xi) \left(\underbrace{\partial_{\beta'}\tVn }_{\text{\ \ see \ \ } (\ref{Eq. dVbeta})} \right)\right)^{'}  \right\} \nn \\
&=& -\frac{1}{\sigma^{6}} \sum_{t=1}^{T} \left\{ \left(\Rn \tXn \right)' \tVn +  \left(\Rn \tXn\right)' \tVn \right\} \nn \\
&=&-\frac{2}{\sigma^{6}} \sum_{t=1}^{T}  \left(\Rn \tXn \right)' \tVn 
\eea
\bea
B^{(4,4,\lambda)} (\beta',\lambda,\rho, \sigma^2) &=& \frac{1}{\sigma^{6}} \sum_{t=1}^{T} \left\{ \left(\underbrace{\partial_{\lambda}\tVn }_{\text{\ \ see \ \ } (\ref{Eq. dVlambda})} \right)' \tVn+  \tilde{V}_{nt}'(\xi) \left(\underbrace{\partial_{\lambda}\tVn }_{\text{\ \ see \ \ } (\ref{Eq. dVlambda})} \right)  \right\} \nn \\ 
&=& -\frac{1}{\sigma^{6}} \sum_{t=1}^{T} \left\{ \left(\Rn W_n \tYn \right)' \tVn +\rVn \Rn W_n \tYn \right\} \nn \\
\eea
\bea
B^{(4,4,\rho)} (\beta',\lambda,\rho, \sigma^2) &=& \frac{1}{\sigma^{6}}\sum_{t=1}^{T} \left\{ \left(\underbrace{\partial_{\rho}\tVn }_{\text{\ \ see \ \ } (\ref{Eq. dVrho})} \right)' \tVn+  \tilde{V}_{nt}'(\xi) \left(\underbrace{\partial_{\rho}\tVn }_{\text{\ \ see \ \ } (\ref{Eq. dVrho})} \right)  \right\} \nn \\
&=&- \frac{1}{\sigma^{6}}\sum_{t=1}^{T} \left\{ \left(\Hn \tVn \right)'\tVn + \rVn \Hn \tVn \right\}  \nn \\
\eea
\bea
B^{(4,4,\sigma^2)}(\beta',\lambda,\rho, \sigma^2) &=& m\sigma^{-6}-3\sigma^{-8}\sum_{t=1}^{T}\tilde{V}_{nt}'(\xi)
\tVn
\eea
\end{itemize}

\subsection{Component-wise calculation of $- \mathbb{E} \left( \frac{1}{m} \frac{\partial^2 \ell_{n,T}(\theta_0)}{\partial \theta \partial \theta'}  \right)$} 

\label{AppendixD} 
From the model setting, we have $V_{nt} = (V_{1t},V_{2t},\cdots V_{nt})$ and $V_{it}$ is i.i.d. across i and t with zero mean and variance $\sigma_0^2$. So $\mathbb{E}(V_{nt}) = 0_{n \times 1}$, $Var(V_{nt})=\sigma_0^2I_n$. 
Knowing that $\tVn = V_{nt}-\sum_{t=1}^{T}V_{nt}/T$, \\
$\mathbb{E}(\tVn)=0_{n \times 1}$, 
$Var(\tVn)= Var((1-1/T)V_{nt}+1/T \sum_{j=1,j\neq t}^{T} V_{nj})=\frac{T-1}{T}\sigma_0^2 I_n$.\\
 $\tVn=\Rn[S_n(\lambda) \tYn-\tXn \beta]$, so $\mathbb{E}(\tYn)= S_n^{-1}(\lambda)\tXn \beta$.\\
 Some other notations: 
$\ddot{X}_{nt} = R_n\tXn, H_n=M_n R_n^{-1}, G_n=W_n S_n^{-1},\ddot{G}_n=R_n G_n R_n^{-1}$.
\begin{itemize}

\item \bea
\mathbb{E}\left[\frac{1}{m\sigma_0^{2}} \sum_t (R_n(\rho_0) \tXn)'(R_n(\rho_0) \tXn)\right]= \frac{1}{m\sigma_0^{2}} \sum_{t=1}^{T}\ddot{X}_{nt}^{'}\ddot{X}_{nt} 
\eea
\begin{eqnarray}
\mathbb{E}\left[ \frac{1}{m\sigma_0^{2}} \sum_t (R_n(\rho_0) W_n \tYn)'(R_n(\rho_0) \tXn) \right] &=& \frac{1}{m\sigma_0^{2}}\sum_t \left(R_n(\rho_0) W_n E(\tYn)\right)'(R_n(\rho_0) \tXn) \nonumber \\   
&=& \frac{1}{m\sigma_0^{2}}\sum_t (R_n(\rho_0) \underbrace{W_n S_n^{-1}}_{G_n}\tXn\beta_0)'\ddot{X}_{nt} \nonumber \\
&=& \frac{1}{m\sigma_0^{2}} \sum_t (R_n(\rho_0) G_nR_n^{-1}R_n\tXn\beta_0)'\ddot{X}_{nt} \nonumber \\ 
&=& \frac{1}{m\sigma_0^{2}}\sum_t (\ddot{G}_n(\lambda_0)\ddot{X}_{nt}\beta_0)'\ddot{X}_{nt}
\end{eqnarray}
\begin{eqnarray}
&&\mathbb{E}\left[\frac{1}{m\sigma_0^{2}} \sum_t \left\{(H_n(\rho_0) \tilde{V}_{nt}(\xi_0))'(R_n(\lambda_0) \tXn) + \tilde{V}_{nt}^{'}(\xi_0) M_n \tXn\right\}\right] \nonumber \\
&&= \frac{1}{m\sigma_0^{2}} \sum_t \left\{(H_n(\rho_0) \mathbb{E}\left[\tilde{V}_{nt}(\xi_0)\right])'(R_n(\lambda_0) \tXn) + \mathbb{E}\left[\tilde{V}_{nt}^{'}(\xi_0)\right] M_n \tXn\right\}=0_{1 \times k} \nn \\
\end{eqnarray}
\begin{eqnarray}
\mathbb{E}\left[\frac{1}{m\sigma_0^{4}} \sum_t \tilde{V}_{nt}^{'}(\xi_0)R_n(\lambda_0)\tXn \right] &=& \frac{1}{m\sigma_0^{4}} \sum_t \mathbb{E}\left[\tilde{V}_{nt}^{'}(\xi_0)\right] R_n(\lambda_0)\tXn =0_{1 \times k} \nn \\
\end{eqnarray}
So we prove the first column of  $\Sigma_{0,n,T}$, that is:\\
\begin{eqnarray*}
\frac{1}{m\sigma_0^2} \left( \begin{array}{c}    \sum_{t}\ddot{X}_{nt}^{'}\ddot{X}_{nt}    \\
                                                \sum_t (\ddot{G}_n(\lambda_0)\ddot{X}_{nt}\beta_0)'\ddot{X}_{nt} \\
                                                0_{1  \times k}  \\    
                                                0_{1  \times k}    
 \end{array}\right)
\end{eqnarray*}

\item \begin{eqnarray}
& &\mathbb{E}\left[ \frac{1}{m\sigma_0^{2}} \sum_t \left(R_n(\rho_0) W_n \tYn)'(R_n(\rho_0)W_n \tYn\right) + \frac{T-1}{m} \text{tr}(G_n^2(\lambda_0))\right] \nonumber \\
& &= \mathbb{E}\left[ \frac{1}{m\sigma_0^{2}} \sum_t \left(R_n W_n S_n^{-1}(R_n^{-1}\tilde{V}_{nt}+\tXn \beta_0)\right)' \left(R_nW_n S_n^{-1}(R_n^{-1}\tilde{V}_{nt}+\tXn \beta_0)\right)\right] \nn \\
&& + \frac{1}{n} \text{tr}(G_n^2(\lambda_0)) \nonumber\\
& &=\mathbb{E}\left[ \left(\frac{1}{m\sigma_0^{2}} \sum_t \left(\ddot{G}_n\tilde{V}_{nt}(\xi_0)+\ddot{G}_n\ddot{X}_{nt}\beta_0 \right)' \left(\ddot{G}_n\tilde{V}_{nt}(\xi_0)+\ddot{G}_n\ddot{X}_{nt}\beta_0\right)\right)\right]+\frac{1}{n} \text{tr}(G_n^2(\lambda_0)) \nonumber\\
& &=\mathbb{E}\left[\frac{1}{m\sigma_0^{2}}\sum_t \left(\ddot{G}_n(\lambda_0)\tilde{V}_{nt}(\xi_0)\right)'\left(\ddot{G}_n(\lambda_0)\tilde{V}_{nt}(\xi_0)\right)\right]\nn \\
&&+\frac{1}{m\sigma_0^{2}} \sum_t\left(\ddot{G}_n(\lambda_0)\ddot{X}_{nt}\beta_0\right)'\left(\ddot{G}_n(\lambda_0)\ddot{X}_{nt}\beta_0\right) \nonumber \\
& &+\frac{1}{m\sigma_0^{2}}\sum_t \left(\ddot{G}_n(\lambda_0)\ddot{X}_{nt}\beta_0\right)'\ddot{G}_n(\lambda_0)\underbrace{\mathbb{E}\left[\tilde{V}_{nt}(\xi_0)\right]}_{0_{n \times 1}}+\frac{1}{n} \text{tr}(G_n^2(\lambda_0)) \nonumber \\
& &=\frac{1}{n} \text{tr}(\ddot{G}_n^{'}(\lambda_0)\ddot{G}_n(\lambda_0))+\frac{1}{n} \text{tr}(R_n^{-1}(\rho_0)R_n(\rho_0)G_n^2(\lambda_0))
\nn \\&&+\frac{1}{m\sigma_0^{2}} \sum_t\left(\ddot{G}_n(\lambda_0)\ddot{X}_{nt}\beta_0\right)'\left(\ddot{G}_n(\lambda_0)\ddot{X}_{nt}\beta_0\right) \nonumber \\ 
& &=\frac{1}{n} \text{tr}\left(\ddot{G}_n^{'}(\lambda_0)\ddot{G}_n(\lambda_0)+R_n(\rho_0)G_n(\lambda_0)R_n^{-1}(\rho_0)R_n(\rho_0)G_n(\lambda_0)R_n^{-1}\right) \nonumber \\
& &+\frac{1}{m\sigma_0^{2}} \sum_t\left(\ddot{G}_n(\lambda_0)\ddot{X}_{nt}\beta_0\right)'\left(\ddot{G}_n(\lambda_0)\ddot{X}_{nt}\beta_0\right) \nonumber \\ 
& &=\frac{1}{n} \text{tr}\left(\ddot{G}_n^{'}(\lambda_0)\ddot{G}_n(\lambda_0)+\ddot{G}_n(\lambda_0)\ddot{G}_n(\lambda_0)\right)+\frac{1}{m\sigma_0^{2}} \sum_t\left(\ddot{G}_n(\lambda_0)\ddot{X}_{nt}\beta_0\right)'\left(\ddot{G}_n(\lambda_0)\ddot{X}_{nt}\beta_0\right) \nonumber \\ 
& &=\frac{1}{n} \text{tr}\left(\ddot{G}_n^{S}(\lambda_0)\ddot{G}_n(\lambda_0)\right)+\frac{1}{m\sigma_0^{2}} \sum_t\left(\ddot{G}_n(\lambda_0)\ddot{X}_{nt}\beta_0\right)'\left(\ddot{G}_n(\lambda_0)\ddot{X}_{nt}\beta_0\right)
\end{eqnarray}

\begin{eqnarray}
& &\mathbb{E}\left[ \frac{1}{m \sigma_0^{2}} \sum_t (R_n(\rho_0) W_n \tYn)'(H_n(\lambda_0)  \tilde{V}_{nt}(\xi_0)) + \frac{1}{m\sigma^{2} }\sum_t  (M_n W_n \tYn)' \tilde{V}_{nt}(\xi_0) \right] \nonumber \\
& &=  \mathbb{E}\left[ \frac{1}{m \sigma_0^{2}} \sum_t  \left(R_n(\rho_0) W_n S_n^{-1}(\lambda_0)(R_n^{-1}(\rho_0)\tilde{V}_{nt}(\xi_0)+\tXn \beta_0)\right)'(H_n(\lambda_0)  \tilde{V}_{nt}(\xi_0)) \right] \nonumber\\
& &+ \mathbb{E}\left [  \frac{1}{m \sigma_0^{2}} \sum_t \left(M_n W_nS_n^{-1}(\lambda_0)(R_n^{-1}(\rho_0)\tilde{V}_{nt}(\xi_0)+\tXn \beta_0)\right)'\tilde{V}_{nt}(\xi_0) \right] \nonumber \\
& &= \mathbb{E}\left[ \frac{1}{m \sigma_0^{2}} \sum_t \left( \left(\ddot{G}_n(\lambda_0)\tilde{V}_{nt}(\xi_0)\right)'(H_n(\lambda_0)  \tilde{V}_{nt}(\xi_0))\right)\right]\nn \\ 
&&+\mathbb{E}\left[\frac{1}{m \sigma_0^{2}} \sum_t  \left(\left(\ddot{G}_n(\lambda_0)\ddot{X}_{nt} \beta_0)\right)'(H_n(\lambda_0)  \tilde{V}_{nt}(\xi_0))\right) \right] \nonumber\\
& &+\mathbb{E}\left[ \frac{1}{m \sigma_0^{2}} \sum_t \left( \left(H_n(\rho_0)\ddot{G}_n(\lambda_0)\tilde{V}_{nt}(\xi_0)     \right)'\tilde{V}_{nt}(\xi_0)+\left(M_nG_n(\lambda_0)\tXn\beta_0\right)'\tilde{V}_{nt}(\xi_0) \right) \right]   \nonumber \\
& &=\frac{1}{n} \text{tr}\left(H_n^{'}(\rho_0)\ddot{G}_n(\lambda_0)\right) +\frac{1}{n} \text{tr}\left(H_n(\rho_0)\ddot{G}_n(\lambda_0)\right) =\frac{1}{n} \text{tr}\left(H_n^{S}(\rho_0)\ddot{G}_n(\lambda_0)\right)
\end{eqnarray}

\begin{eqnarray}
& &\mathbb{E}\left[\frac{1}{m\sigma_0^{4}} \sum_t \left(R_n(\rho_0) W_n \tYn   \right)' \tilde{V}_{nt}(\xi_0)    \right]   \nonumber \\
& &= \mathbb{E}\left[\frac{1}{m\sigma_0^{4}} \sum_t \left(R_n(\rho_0) W_n S_n^{-1}(\lambda_0)(R_n^{-1}(\rho_0)\tilde{V}_{nt}(\xi_0)+\tXn \beta_0)   \right)' \tilde{V}_{nt}(\xi_0)    \right] \nonumber \\
& &= \mathbb{E}\left[\frac{1}{m\sigma_0^{4}} \sum_t \left(\underbrace{R_n(\rho_0) W_n S_n^{-1}(\lambda_0)R_n^{-1}(\rho_0)}_{\ddot{G}_n(\lambda_0)}\tilde{V}_{nt}(\xi_0) \right)'\tilde{V}_{nt}(\xi_0)\right] \nonumber \\
& &+  \mathbb{E}\left[\frac{1}{m\sigma_0^{4}} \sum_t \left(R_n(\rho_0) W_n\tXn \beta_0 \right)' \tilde{V}_{nt}(\xi_0)    \right] \nonumber \\
& & = \frac{1}{n\sigma_0^{2}}\text{tr}(\ddot{G}_n(\lambda_0))
\end{eqnarray}

The second column of $\Sigma_{0,n,T}$ is:\\
\begin{eqnarray*}
\left( \begin{array}{c}    \frac{1}{m\sigma_0^2}\sum_t \ddot{X}_{nt}^{'}(\ddot{G}_n(\lambda_0)\ddot{X}_{nt}\beta_0)\\
\frac{1}{n} \text{tr}\left(\ddot{G}_n^{S}(\lambda_0)\ddot{G}_n(\lambda_0)\right)+\frac{1}{m\sigma_0^{2}} \sum_t\left(\ddot{G}_n(\lambda_0)\ddot{X}_{nt}\beta_0\right)'\left(\ddot{G}_n(\lambda_0)\ddot{X}_{nt}\beta_0\right)\\
\frac{1}{n} \text{tr}\left(H_n^{S}(\rho_0)\ddot{G}_n(\lambda_0)\right)\\
\frac{1}{n\sigma_0^{2}}\text{tr}(\ddot{G}_n(\lambda_0))
\end{array}\right)
\end{eqnarray*}


\item \begin{eqnarray}   
&&\mathbb{E}\left[\frac{1}{m\sigma_0^{2} }\sum_t \left(H_n(\rho_0)\tilde{V}_{nt}(\xi_0) \right)' H_n(\rho_0)\tilde{V}_{nt}(\xi_0)+ \frac{T-1}{m}\text{tr}\left(H_n^2 (\rho_0)\right)\right] \nonumber \\
&&= \frac{1}{n}\text{tr}\left(H_n^{'}(\rho_0)H_n(\rho_0)+H_n^2 (\rho_0)\right) = \frac{1}{n}\text{tr}\left(H_n^{S}(\rho_0)H_n(\rho_0)\right)
\end{eqnarray}
\begin{eqnarray}  
&&\mathbb{E}\left[\frac{1}{m\sigma_0^{4} }\sum_t \left(H_n(\rho_0)\tilde{V}_{nt}(\xi_0) \right)' \tilde{V}_{nt}(\xi_0)\right]
= \frac{1}{n\sigma_0^{2}}\text{tr}(H_n(\rho_0)) 
\end{eqnarray}

The third column of $\Sigma_{0,n,T}$ is:\\
\begin{eqnarray*}
\left( \begin{array}{c} 0_{k \times 1}\\
    \frac{1}{n} \text{tr}\left(H_n^{S}(\rho_0)\ddot{G}_n(\lambda_0)\right)\\                         
 \frac{1}{n}\text{tr}\left(H_n^{S}(\rho_0)H_n(\rho_0)\right)\\
\frac{1}{n\sigma_0^{2}}\text{tr}(H_n(\rho_0)) 
\end{array}\right)
\end{eqnarray*}


\item \begin{eqnarray}    
&& \mathbb{E}\left[-\frac{m}{2m\sigma_0^{4}} + \frac{1}{m\sigma_0^{6} }\sum_t \tilde{V}_{nt}'(\xi_0) \tilde{V}_{nt}(\xi_0) \right]=-\frac{1}{2\sigma_0^{4}}+ \frac{1}{m\sigma_0^{6} }T\frac{T-1}{T}\sigma_0^2\cdot n =\frac{1}{2\sigma_0^{4}} \nn \\
\end{eqnarray}

The fourth column of  $\Sigma_{0,n,T}$ is:\\
\begin{eqnarray*}
\left( \begin{array}{c}  0_{k \times 1}\\
                                 \frac{1}{n\sigma_0^{2}}\text{tr}(\ddot{G}_n(\lambda_0))\\
                                 \frac{1}{n\sigma_0^{2}}\text{tr}(H_n(\rho_0))\\
                                 \frac{1}{2\sigma_0^{4}}                                  
\end{array}\right).
\end{eqnarray*}

Thus, we have:
{\footnotesize
\begin{eqnarray*}
\Sigma_{0,T} = &&\left( \begin{array}{cccc} \frac{1}{m\sigma_0^2} \sum_{t}\ddot{X}_{nt}^{'}\ddot{X}_{nt} & \frac{1}{m\sigma_0^2}\sum_t \ddot{X}_{nt}^{'}(\ddot{G}_n(\lambda_0)\ddot{X}_{nt}\beta_0)&0_{k \times 1}&0_{k \times 1} \\
                                                \frac{1}{m\sigma_0^2}\sum_t (\ddot{G}_n(\lambda_0)\ddot{X}_{nt}\beta_0)'\ddot{X}_{nt}&\frac{1}{m\sigma_0^{2}} \sum_t\left(\ddot{G}_n(\lambda_0)\ddot{X}_{nt}\beta_0\right)'\left(\ddot{G}_n(\lambda_0)\ddot{X}_{nt}\beta_0\right) &0&0\\
                                                0_{1  \times k} &0&0& 0\\    
                                                0_{1  \times k}  &  0&0&0  
\end{array}\right) \\
&&+\left( \begin{array}{cccc}0_{k \times k}& 0_{k \times 1}&0_{k \times 1}&0_{k \times 1} \\
                                              0_{1 \times k}  &\frac{1}{n} \text{tr}\left(\ddot{G}_n^{S}(\lambda_0)\ddot{G}_n(\lambda_0)\right)& \frac{1}{n} \text{tr}\left(H_n^{S}(\rho_0)\ddot{G}_n(\lambda_0)\right)&\frac{1}{n\sigma_0^{2}}\text{tr}(\ddot{G}_n(\lambda_0)) \\
                                                0_{1  \times k} &\frac{1}{n} \text{tr}\left(H_n^{S}(\rho_0)\ddot{G}_n(\lambda_0)\right)&\frac{1}{n}\text{tr}\left(H_n^{S}(\rho_0)H_n(\rho_0)\right)&  \frac{1}{n\sigma_0^{2}}\text{tr}(H_n(\rho_0))\\    
                                                0_{1  \times k}  &  \frac{1}{n\sigma_0^{2}}\text{tr}(\ddot{G}_n(\lambda_0))&\frac{1}{n\sigma_0^{2}}\text{tr}(H_n(\rho_0))&\frac{1}{2\sigma_0^{4}}   
\end{array}\right)
\end{eqnarray*}
}
\end{itemize}

\section{Algorithms}  \label{Section: AppendixAlgo}

\subsection{Algorithm 1} \label{Subsec: Algm1}

Most of the quantities related to the saddlepoint density approximation $p_{n,T}$ and the tail area in (4.15) 
are available in closed-form. 
Here, we provide an algorithm (see Algorithm 1) in which we itemize the main computational steps needed to implement the saddlepoint tail area approximation, 
for a given transformation $q$ 
and for a given reference parameter $\theta_0$---it is, e.g., the parameter characterizing the null hypothesis in a simple hypothesis testing, where the tail area probability is an approximate $p$-value.  

\begin{algorithm}[tbh] 
\caption{Density and tail-area saddlepoint approximation} \label{Algm1}
\SetAlgoLined
\KwIn{a sample of $Y\nt$ and $X\nt$; the reference parameter $\theta_0$.}
\KwOut{density and tail area saddlepoint approximation.}
\begin{algorithmic}[1]
\STATE Given a sample of $Y\nt$ and $X\nt$, first compute the transformed values $\tilde{Y}\nt$ and $\tilde{X}\nt$, as in (3.3)
\STATE Compute $\mu_{n,T}$, $\sigma^2_g$,  $\varkappa^{(3)}_{n,T}$ and $\varkappa^{(4)}_{n,T}$, using the formulae available in 
Appendix A.3 and A.4
\STATE Combine the expressions of the approximate cumulants into the analytical expression of the approximate c.g.f.\ $\tilde{\mathcal{K}}_{n,T}(\nu)$ given by (A.42) 
in Appendix A.4 
\STATE Define a grid $\mathcal{Z}$ of points $\{z_j, j=1,...,j_{\max}\}$. Select the min (i.e. $z_1$) and max (i.e. $z_{j_{\max}}$) value of this grid according to the min and max values taken by  the parameter of interest 
 
\STATE \For{$j \gets 1$ \textbf{to} $j_{\max}$}{
Compute the saddlepoint density approximation $p_{n,T}(z_j)$ as in (4.13)
}
\STATE Compute the tail area using formula (4.15) 
or by numerical integration of the density $p_{n,T}$
\end{algorithmic}
\end{algorithm}



Some additional comments are in order. Step 2  requires the approximation of some expected values, like, e.g., $\mathbb{E}
\left[g^2_{i,T}(\psi,P_{\theta_0}) \right]$, for $g_{i,T}(\psi,P_{\theta_0}) $ as in (4.10), 
which numerical methods can provide. For instance, we can rely on numerical integration with respect to the underlying  Gaussian distribution $P_{\theta_0}$, or on the 
Laplace method, or on the approximation of the integrals by Riemann sums, using simulated data.
In our experience, the latter approximation represents a good compromise between accuracy, simplicity to code, and 
computational burden. Moreover for Step 5, one needs to solve Eq. (4.14) 
for each grid point. Some well-known methods are available. For instance, \citet{K06} p. 84 suggests the use of 
Newton-Raphson derivative-based methods.  Specifically, for a given starting value $\nu_0$ (which is an approximate solution to the saddlepoint equation), a first-order Taylor expansion of the saddlepoint equation yields $\tilde{\mathcal{K}'}_{n,T}(\nu_0) + \tilde{\mathcal{K}''}_{n,T}(\nu_0) (\nu-\nu_0) \approx z $, whose solution is 
$
\nu = \nu_0 + {(z-\tilde{\mathcal{K}'}_{n,T}(\nu_0))}/{\tilde{\mathcal{K}''}_{n,T}(\nu_0)}.$
We apply this solution to update the approximate solution $\nu_0$, yielding a new approximation $\nu$ to the saddlepoint. We iterate the procedure until the approximate solution is accurate enough---e.g., we can set a tolerance value (say, {\it tol}) and iterate the procedure till $\vert z-\tilde{\mathcal{K}'}_{n,T}(\nu) \vert < \text{{\it tol}}$. An alternative option is the secant method; see \citet{K06} p. 86. As noticed by \citet{GR96}, due to the approximate nature of the c.g.f.\ in (A.42), 
 the saddlepoint equation  can admit multiple solutions in some areas of the density. To solve this problem, one may use the modified  c.g.f. proposed by \citet{W92general}.

\subsection{Algorithm 2} \label{Subsec: Algm2}

To implement the test (6.2) 
for the problem (6.1), 
we propose the following algorithm: 

\begin{algorithm}[H] \label{AlgmTest}
\SetAlgoLined
\KwIn{a sample of $Y_{nt}$ and $X_{nt}$; the testing problem in (6.1).}
\KwOut{saddlepoint test based on the test statistic ${SAD}_n(\hat{\lambda})$}
\begin{algorithmic}[1]
\STATE Given a sample of $Y_{nt}$ and $X_{nt}$, first compute the transformed values $\tilde{Y}_{nt}$ and $\tilde{X}_{nt}$, as in (3.3)
\STATE Compute the maximum likelihood estimate $(\hat\lambda, \hat{\theta}_2)$
\STATE \For{$i \gets 1$ \textbf{to} $n$}{
Plug-in $\hat\lambda$ in $\psi_{i}^{(T)}(\lambda,\theta_2)$ and compute $\mathcal{K}_{\psi_i}^{(T)}(\nu, \hat\lambda,\theta_2)$
}
\STATE Compute  $\mathcal{K}_\psi(\nu,\hat\lambda,\theta_2)$ as in (6.3)
\STATE Obtain the value of ${SAD}_n(\hat{\lambda})$ and compare it to the quantile of $\chi_1^2$--- a chi-square with one degree-of-freedom.
\end{algorithmic}
\caption{Saddlepoint test in the presence of nuisance parameters}
\end{algorithm}

As far as Step 3 is concerned, we suggest to apply a standard MC integration to approximate numerically the  c.g.f. $\mathcal{K}_\psi$---this is in line with our suggestion for Algorithm \ref{Algm1}. Specifically, for each location $i$, we may approximate  
 $
    \mathcal{K}_{\psi_i}^{(T)}(\nu, \lambda,\theta_2) = \ln E_{P_{(\lambda_{0}, \theta_{2})}}[\exp
    \{\nu^{T}\psi_i^{(T)}(\lambda, \theta_2)\}]
 $ 
by simulating from the corresponding SARAR model under the null. To this end, for a user-specified MC size (\text{MC.size}), in each $j$-th MC run, we simulate the variables $\{Y_{nt}^{(j)}\}_{t=1,\cdots,T}$, under the probability $P_{(\lambda_{0}, \theta_{2})}$. We remark that $P_{(\lambda_{0}, \theta_{2})}$ is given by the composite null hypothesis, where the nuisance parameters are not specified. Thus, $\theta_2$ is not fixed  in the MC runs: the numerical optimization needed to compute the infimum over $\theta_2$ takes care of that aspect. Then,  we set 
     \begin{eqnarray*}
    {\mathcal{K}}_{\psi_{i}^{(T)}} (\nu,\hat\lambda,\theta_2) 
    &\approx& \ln \left ( \frac{1}{\text{MC.size}}\sum_{j= 1}
    ^{\text{MC.size}}\exp\{\nu^{T}\psi_{i,j}^{(T)}(\hat{\lambda},\theta_2)\} \right),
    \end{eqnarray*}
    where $\psi_{i,j}^{(T)}$ is the estimating function at location $i$ as evaluated in the $j$-th MC run. Finally, we build the test using the MC approximated c.g.f. ${\mathcal{K}}_\psi(\nu,\hat\lambda,\theta_2) = -n^{-1}\sum_{i = 1}^{n} {\mathcal{K}}_{\psi_{i}^{(T)}}(\nu,\hat\lambda,\theta_2)$.

\section{Additional results for the SAR(1) model}  \label{Section: Appendix3}

\subsection{First-order asymptotics} \label{Subsec: Asy1}
As it is customary in the statistical/econometric software, we consider three different spatial weight matrices: Rook matrix, Queen matrix, and Queen matrix with torus.  In Figure \ref{Fig: Wn1}, we display the geometry of $Y_{nt}$ as implied by each considered spatial matrix for $n=100$: the plots highlight that different matrices imply different spatial relations. For instance, we see that the Rook matrix implies fewer links than the Queen matrix. Indeed, the Rook criterion defines neighbours by the existence of a common edge between two spatial units, whilst the Queen criterion is less rigid and defines neighbours as spatial units sharing an edge or a vertex. Besides, we may interpret $\{Y_{nt}\}$ as a $n$-dimensional random field on the network graph which describes the known underlying spatial structure. Then, $W_n$ represents the weighted adjacency matrix (in the spatial econometrics literature, $W_n$ is called contiguity matrix). In Figure \ref{Fig: Wn1}, we display the geometry of a random field on a regular lattice (undirected graph).

\begin{figure}[htbp]
\begin{center}
\begin{tabular}{l}
\hspace{2.cm} Rook \hspace{3.75cm} Queen   \hspace{3.8cm} Queen torus   \\

\includegraphics[width=0.325\textwidth, height=0.25\textheight]{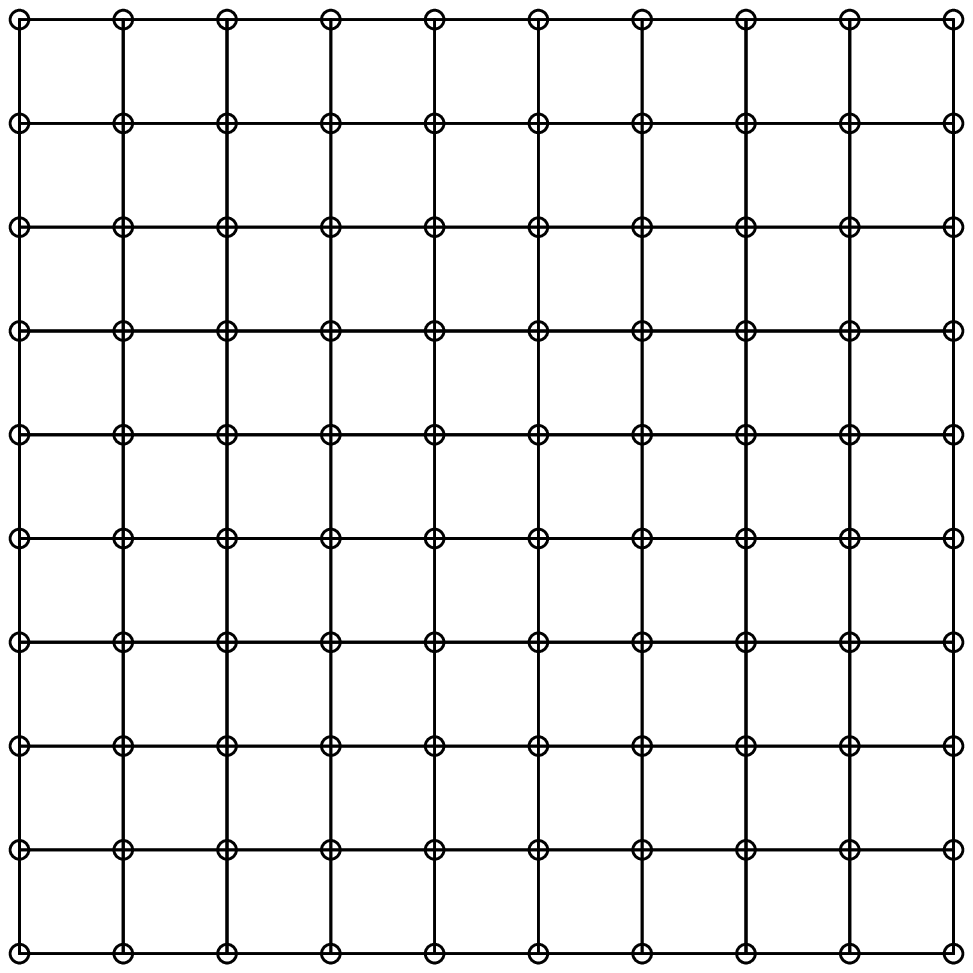}
\includegraphics[width=0.325\textwidth, height=0.25\textheight]{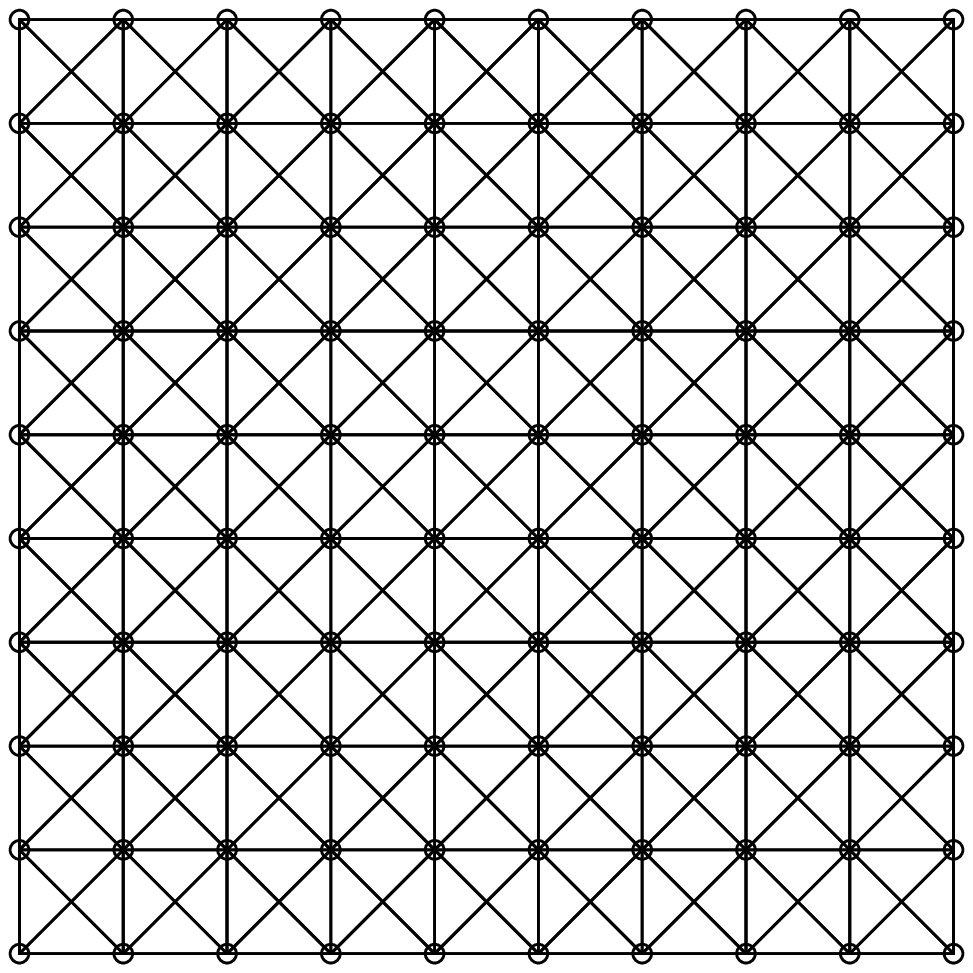}
\includegraphics[width=0.325\textwidth, height=0.25\textheight]{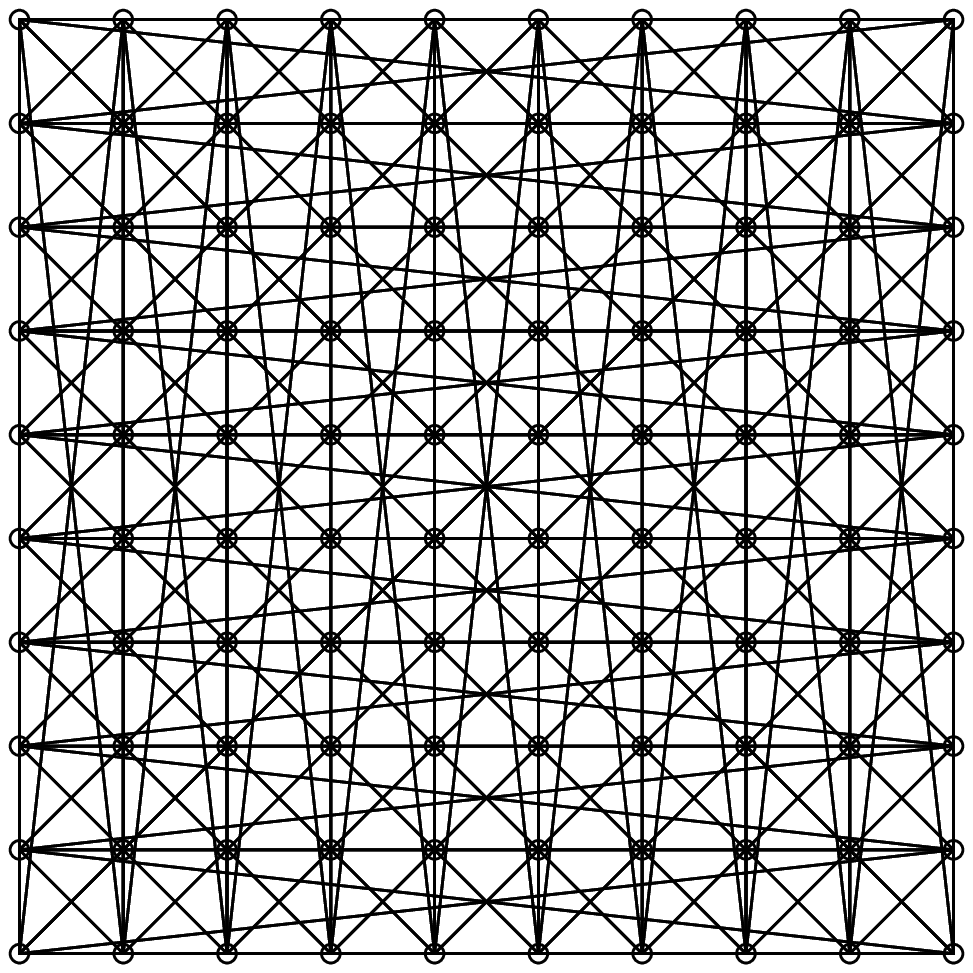}
\end{tabular}
\caption{Different types of neighboring structure for $Y_{nt}$, as implied by different types of $W_n$ matrix, for $n=100$.} 
    \label{Fig: Wn1}
\end{center}
\end{figure}

We complement the motivating example of our research illustrating the low accuracy of the routinely applied first-order asymptotics in the setting of the spatial autoregressive process of order one, henceforth SAR(1):

\begin{equation}
\begin{aligned}
Y\nt &=& \lambda_0 W_n Y\nt + c_{n0} + V\nt,   \quad {\text{for}} \quad t=1,2, 
\end{aligned}
\label{Eq: RR}
\end{equation}

where $V\nt =(v_{1t},v_{2t},..,v\nt)'$ are $n \times 1$ vectors, and $v_{it}\sim \mathcal{N}(0,1)$, i.i.d.\ across $i
$ and $t$. The model is a special case of the general model in (2.1), 
the spatial autoregressive process with spatial autoregressive error (SARAR) of \citet{LY10}. 
Since $c_{n0}$ creates an incidental parameter issue, we 
eliminate it by the standard differentiation procedure. 
Given that we have only two periods, the transformed (differentiated) model is formally equivalent to the cross-sectional SAR(1) model, 
in which $c_{n0} \equiv 0$, a priori;  see \citet{RR15} for a related discussion. 


In the MC exercise, we set $\lambda_0 = 0.2$, and we estimate it through Gaussian  
likelihood maximisation.  The resulting $M$-estimator (the maximum likelihood estimator) 
is consistent and asymptotically normal; see  \S4.1. 
We consider the same types
of $W_n$ as in the motivating example of \S 2.
In Figure \ref{Fig_SAR}, we display the MC results.
Via QQ-plot, we compare the distribution of $\hat\lambda$ to the Gaussian asymptotic distribution (as implied by the first-order asymptotic theory).

\begin{figure}[htb!]
\begin{center}
\begin{tabular}{l}
\hspace{3.7cm} Rook \hspace{3cm} Queen   \hspace{2cm} Queen torus   \\
\begin{turn}
{90} \hspace{1.5cm}  $n = 24$ 
\end{turn}
\hspace{0.3cm} 
\includegraphics[width=0.9\textwidth, height=0.2\textheight]{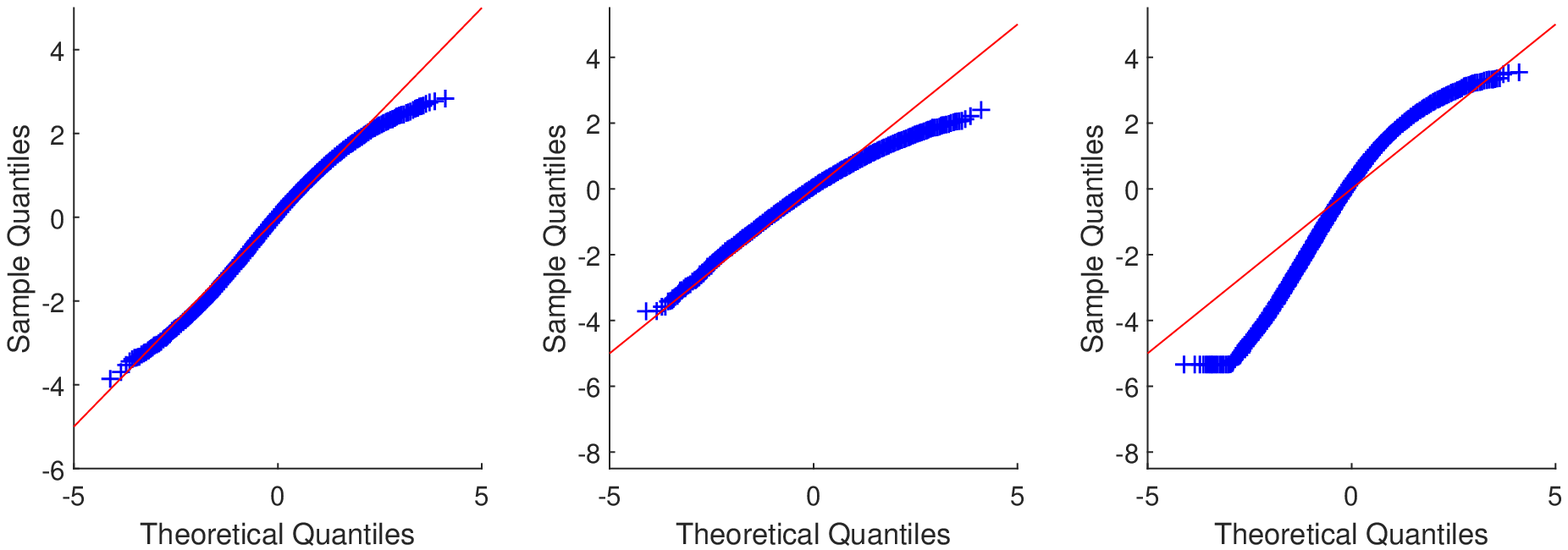} \\
\begin{turn}
{90} \hspace{1.5cm} $n=100$
\end{turn}
\hspace{0.3cm}
\includegraphics[width=0.9\textwidth, height=0.2\textheight]{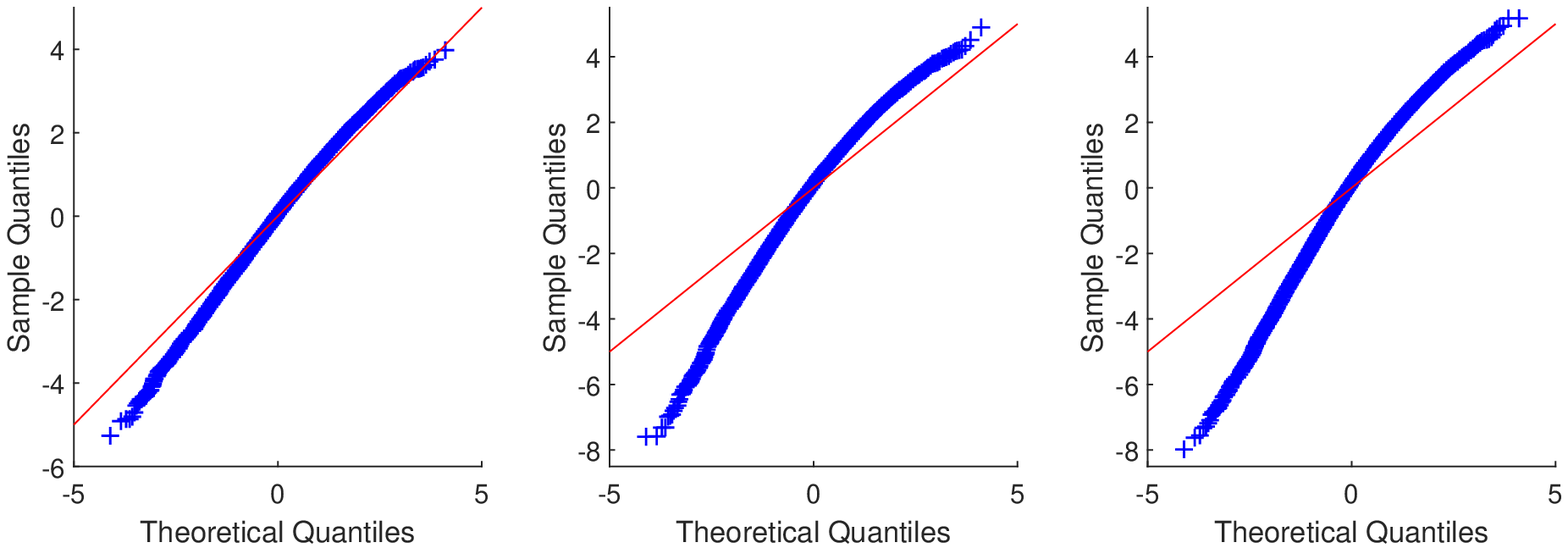} 
\end{tabular}
\caption{SAR(1) model: QQ-plot vs normal of the MLE $\hat\lambda$, for different sample sizes ($n=24$ and $n=100$), $\lambda_0=0.2$ and different types of $W_n$ matrix.}
    \label{Fig_SAR}
\end{center}
\end{figure}

The plots show that, for both sample sizes $n=24$ and $=100$, the Gaussian approximation can be either too thin or too thick in the tails with respect to the ``exact'' distribution (as obtained via simulation). 
The more complex is
the geometry of $W_n$ (e.g., $W_n$ is Queen), the more pronounced are the departures from the Gaussian approximation.

\subsection{Assumption D(iv) and its relation to the spatial weights}  \label{Sec: D_check}
We can obtain $M_{i,T}(\psi,P_{\theta _{0}})$ from (4.4) 
on a SAR(1) model as in (\ref{Eq: RR}), 
$$M_{i,T}(\psi,P_{\theta _{0}})=(T-1)( \tilde{g}_{ii}+g_{ii}),$$ 
where $\tilde{g}_{ii}$ and $g_{ii}$ are the $i_{th}$ element of the diagonals of $G_nG'_n$ and $G^2_n(\lambda_0)$, respectively. Note that $S_n(\lambda_0)=I_n-\lambda_0 W_n$ and $G_n(\lambda_0) = W_nS_n^{-1}(\lambda_0)$.

 To check Assumption D(iv): $\vert\vert M_{i,T}(\psi,P_{\theta _{0}})-M_{j,T}(\psi,P_{\theta _{0}})\vert\vert=O(n^{-1})$, first let us find the expression for the difference between $M_{i,T}(\psi,P_{\theta _{0}})$ and $M_{j,T}(\psi,P_{\theta _{0}})$,
\begin{eqnarray}
M_{i,T}(\psi,P_{\theta _{0}})-M_{j,T}(\psi,P_{\theta _{0}}) &=& (T-1)\l[\l(\tilde{g}_{ii}-\tilde{g}_{jj}\r)+\l(g_{ii}-g_{jj}\r)\r] ,
\end{eqnarray}
where $\tilde{g}_{ii}$ and $\tilde{g}_{ii}$ are $i_{th}$ and $j_{th}$ elements of the diagonal of $G_nG'_n$. $g_{ii}$ snd $g_{jj}$ the $i_{th}$ and $j_{th}$ elements of the diagonal of $G^2_n(\lambda_0)$. Then, we rewrite the expression of $G_n^2(\lambda_0)$ to check its diagonal:
\begin{eqnarray}
G_n^2(\lambda_0) &=& W^2_nS_n^{-2}(\lambda_0) = W_n^2\l(I_n-\lambda_0W_n\r)^{-2} \nn \\
&=&W^2_n\l(I_n+\lambda_0W_n+\lambda_0^2W_n^2+\lambda_0^3W_n^3+\cdots\r)^2 \nn \\
&=& \sum_{k=2}^{\infty}a_kW_n^k,
\end{eqnarray}
where $a_k$ is the coefficient of $W_n^k$.
By comparing the entries of diagonals of $W_n^k$, we can make an assumption on the differences between any two entries of the diagonal of $W_n^k$ that should be at most of order $n^{-1}$, for all integers $k \ge 2 $, such that
$||{g}_{ii}-{g}_{jj}|| = O(n^{-1})$, for any $1\le i$, $j \le n$ and $i\ne j$.

We do the same for $G_nG'_n$:
\begin{eqnarray}
G_nG'_n &=& W_nS_n^{-1}(\lambda_0) \l(W_nS_n^{-1}(\lambda_0) \r)'= W_n\l(I_n-\lambda_0W_n\r)^{-1}\l(W_n\l(I_n-\lambda_0W_n\r)^{-1}\r)' \nn \\
&=&W_n\l(I_n+\lambda_0W_n+\lambda_0^2W_n^2+\lambda_0^3W_n^3+\cdots\r)\l(I_n+\lambda_0W_n+\lambda_0^2W_n^2+\lambda_0^3W_n^3+\cdots\r)'W'_n\nn \\
&=& \sum_{k_1=1}^{\infty}\sum_{k_2=1}^{\infty}a_{k_1k_2} W_n^{k_1}\l(W_n^{k_2}\r)',
\end{eqnarray}
where $a_{k_1k_2}$ is the coefficient of $W_n^{k_1}\l(W_n^{k_2}\r)'$.
Then, we can make a second assumption on the differences between any two entries of the diagonal of $W_n^{k_1}\l(W_n^{k_2}\r)'$ that should be at most of order $n^{-1}$, for all positive integers $k_1$ and $k_2$, such that $||\tilde{g}_{ii}-\tilde{g}_{jj}|| = O(n^{-1})$, for any for any $1\le i$, $j \le n$ and $i\ne j$. From those two assumptions on the behaviour of diagonal entries, we notice that $W_n$ plays a leading role.

To exemplify those two conditions, we consider numerical examples for different choices of $W_n$. The sample size $n$ is $24$ with the benchmark $1/n = 1/24 = 0.041$. When $W_n $ is Rook, the elements of $W_n$ are $\{1/2,1/3,1/4,0\}$. On top of that, we can find the differences between the entries of the diagonals of $W_n^k$ and $W_n^{k_1}\l(W_n^{k_2}\r)'$ are small. For instance, the entries of the diagonal of $W_n^2$ are $\{0.27,0.29,0.31,0.33,0.36\}$. The elements of diagonal of $W_n^3$ are all $0$. For $W_n^2(W_n^2)'$, the entries of its diagonal are \{0.17,0.21,0.22,0.37\}. We can obtain a similar result for Queen weight matrix. For example, the largest difference between diagonal entries of $W_n^2$ is $0.044$. When $W_n $ is Queen with torus, $W_n$ is a symmetric matrix and its elements are $0$ or $1/8$. The entries of the diagonals of $W_n^k$ and $W_n^{k_1}\l(W_n^{k_2}\r)'$ are always identical to each other. Thus, the differences are always $0$. 
Since the entries of $W_n$ are all positive and less than $1$, when the exponent becomes larger,  all the values of the powers of $W_n$ and $W'_n$ will go close to $0.00$. Thus, large powers of $W_n$ definitely satisfy the two assumptions.

\subsection{First-order asymptotics vs saddlepoint approximation}   \label{Subsec: Density}

For the SAR(1) model, we analyse the behaviour of the MLE of $\lambda_0$,
whose PP-plots are available in Figure \ref{Fig: PP}. For each type of
$W_n$, for  $n=100$, the plots show that the saddlepoint approximation is closer to the ``exact'' probability than the first-order asymptotics approximation. 
For $W_n$ Rook, the saddlepoint approximation improves on the routinely-applied first-order asymptotics. In Figure 
\ref{Fig: PP}, the accuracy gains are evident also for $W_n$ Queen and  Queen with torus, where the first-order asymptotic theory displays large errors essentially over the whole support (specially in the tails). On the contrary, 
the saddlepoint approximation is  close to the 45-degree line. 


\begin{figure}[hbt!]
\begin{center}
\begin{tabular}{ccc}
 Rook & Queen  & Queen torus   \\
\begin{turn}
{90} \hspace{2.5cm}  $n = 100$ 
\end{turn}
\includegraphics[width=0.3\textwidth, height=0.265\textheight]{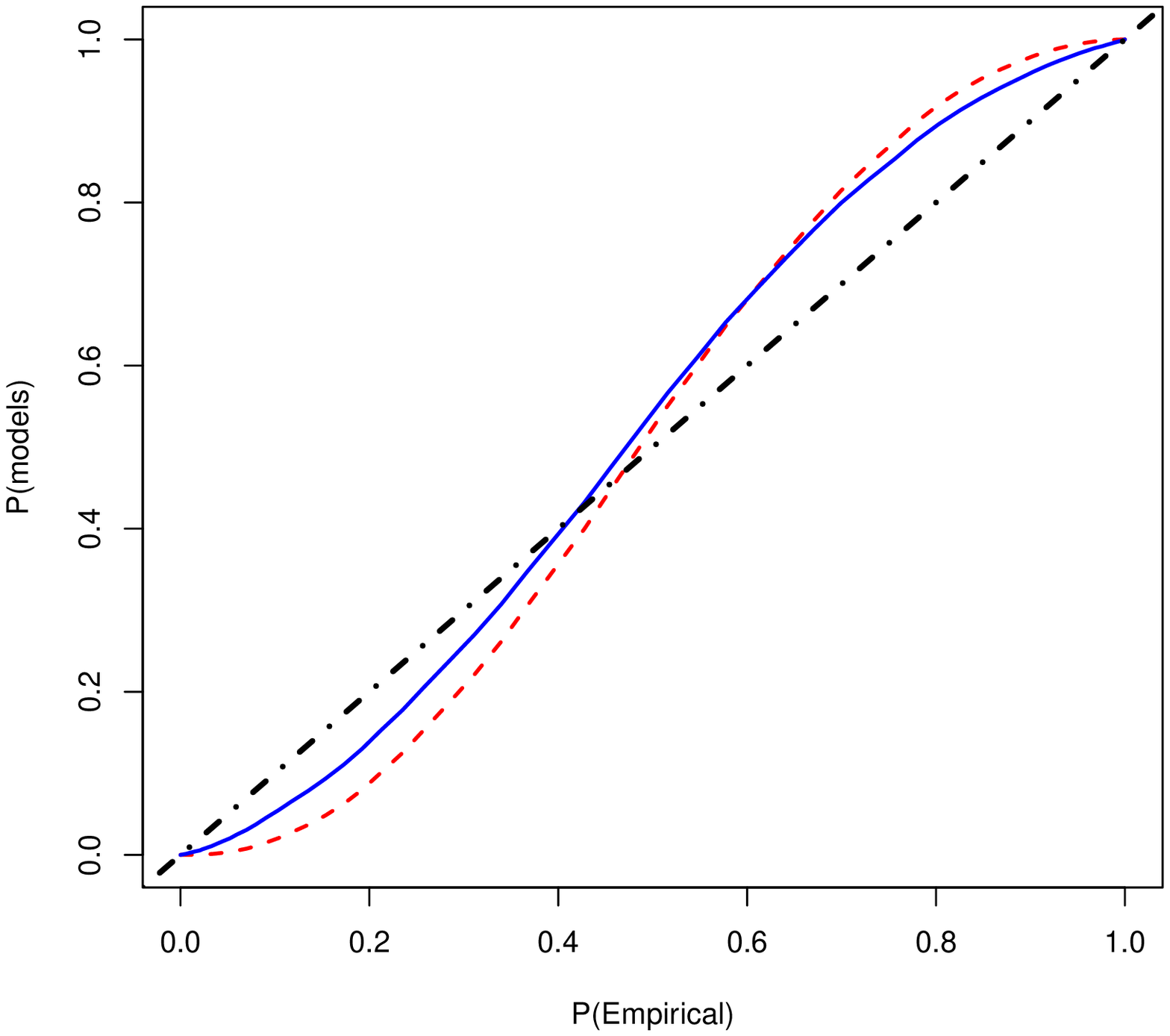} &
\includegraphics[width=0.3\textwidth, height=0.265\textheight]{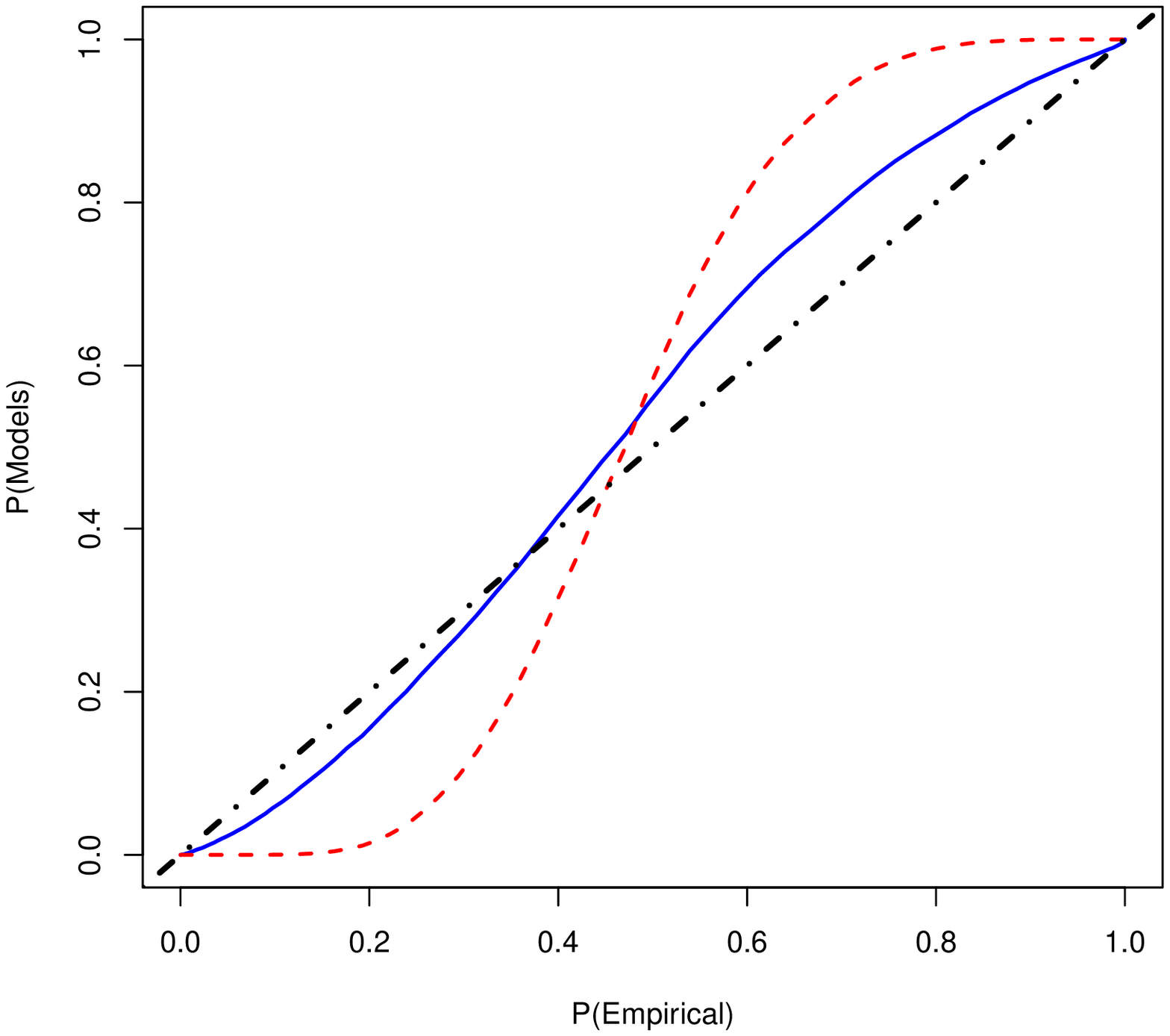} &
\includegraphics[width=0.3\textwidth, height=0.265\textheight]{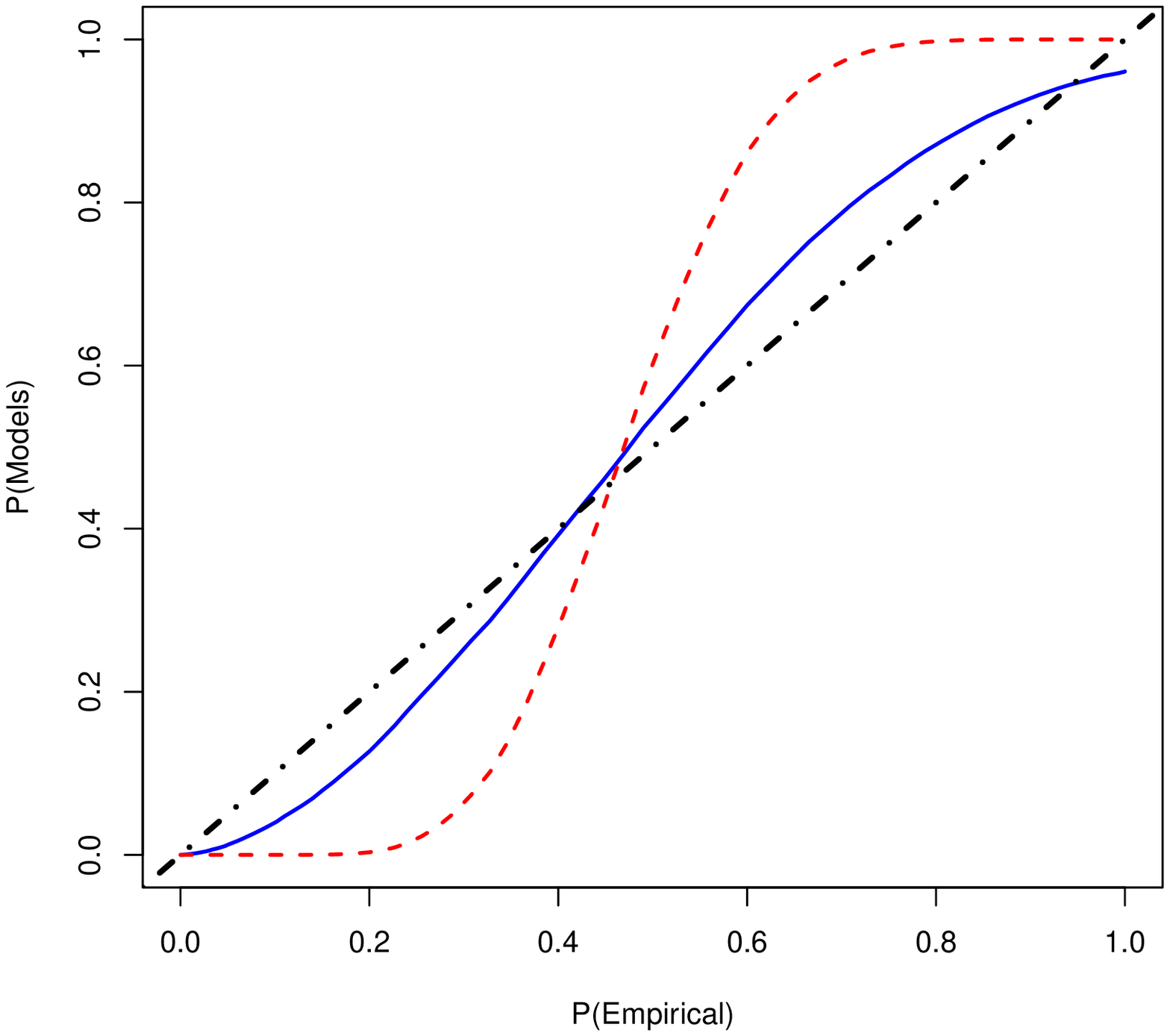}
\end{tabular}
\caption{SAR(1) model: PP-plots for saddlepoint (continuous line) vs asymptotic normal (dotted line) probability approximation, for the MLE $\hat\lambda$, for $n=100$, $\lambda_0=0.2$, and different $W_n$.}
    \label{Fig: PP}
\end{center}
\end{figure}

Density plots show the same information as PP-plots displayed in Figure 3 (of the paper) and Figure \ref{Fig: PP} (in this Appendix), 
but provide a better visualization of the behavior of the considered approximations. Thus, we compute the Gaussian density implied by the asymptotic theory, 
and we compare it to our saddlepoint density approximation. In Figure \ref{Fig: Densities}, we plot the histogram of the ``exact'' estimator density (as 
obtained using 25,000 Monte Carlo runs) to which we superpose both the Gaussian and the saddlepoint density approximation. $W_n$ are
Rook and Queen. The plots illustrate that the saddlepoint technique provides an approximation to the true density which is more accurate
than the one obtained using the first-order asymptotics. 
%
%
\begin{figure}[htbp]
\begin{center}
\begin{tabular}{cc}
Rook  \& $n=24$ & Queen  \&  $n=24$  \\
\includegraphics[width=0.4\textwidth, height=0.25\textheight]{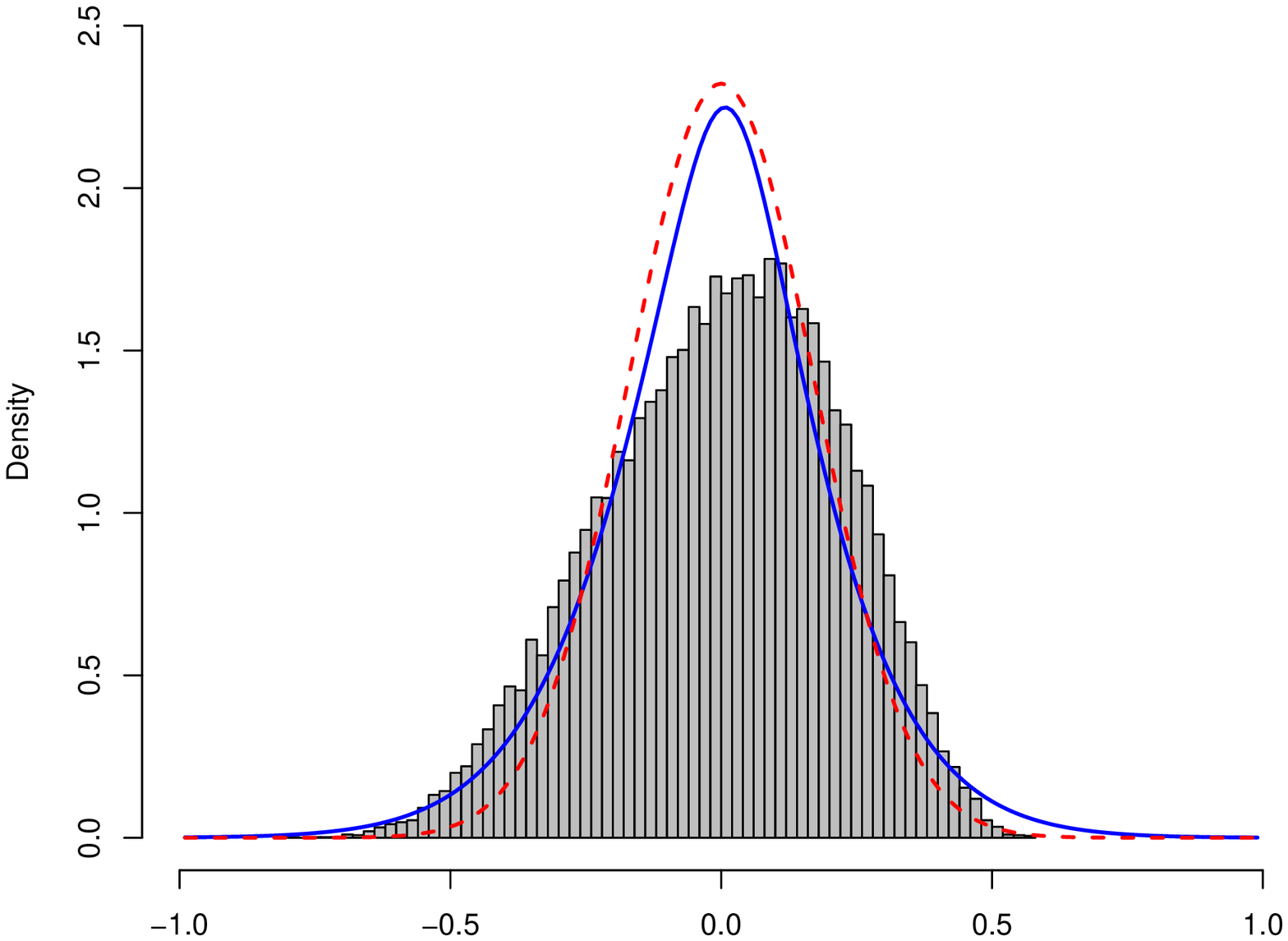} &
 \includegraphics[width=0.4\textwidth, height=0.25\textheight]{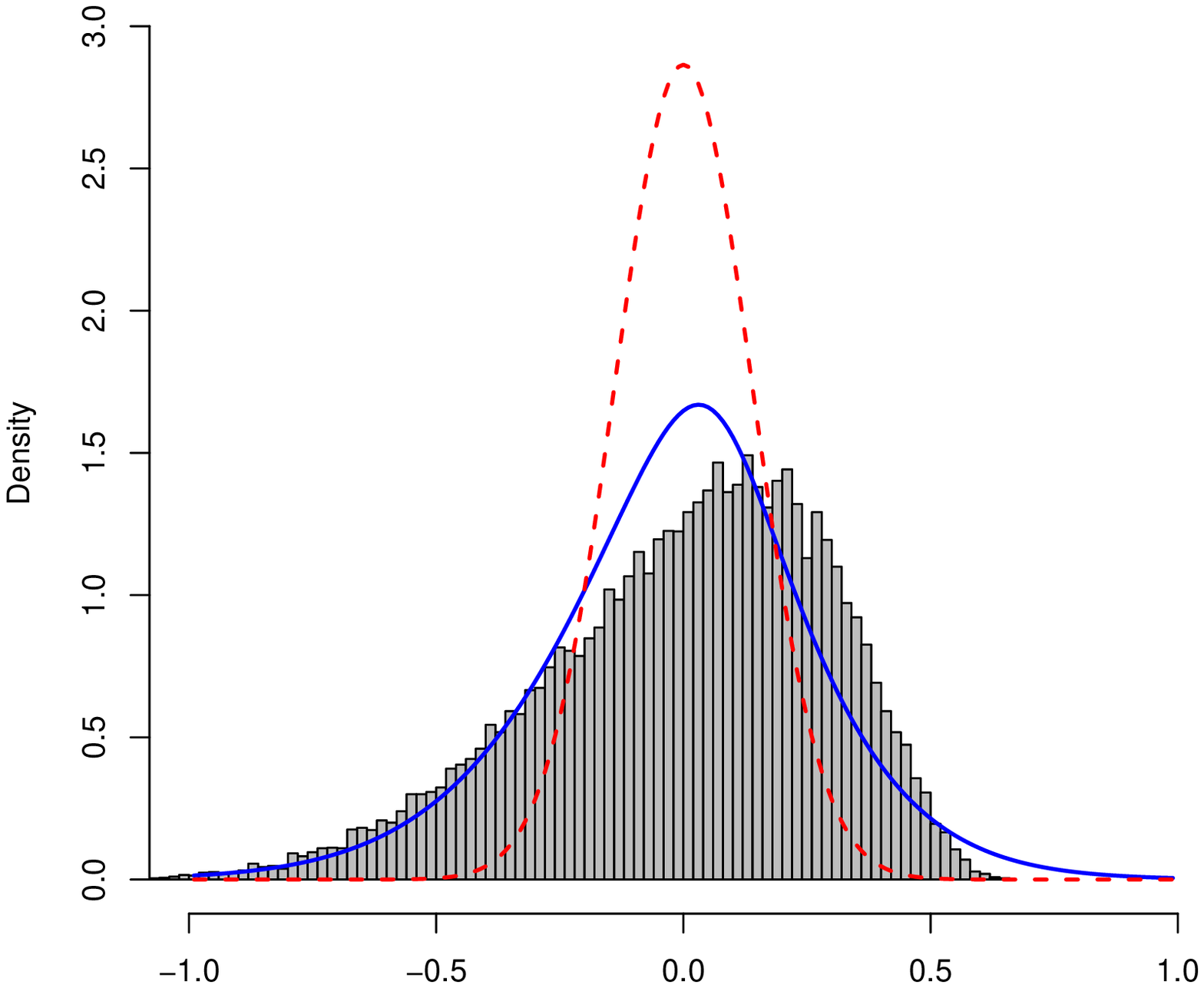} 
\end{tabular}
\caption{SAR(1) model: Density plots for saddlepoint (continuous line) vs asymptotic normal (dotted line) probability approximation to the exact density (as expressed by the histogram and obtained using 
MC with size 25000), for the MLE $\hat\lambda$ and $W_n$ is Rook (left panel) and Queen (right panel). Sample size is $n=24$, while $\lambda_0=0.2$.}
    \label{Fig: Densities}
\end{center}
\end{figure}

\subsection{Saddlepoint approximation vs Edgeworth expansion} 

The Edgeworth expansion derived in Proposition 3 
represents the natural 
alternative to the saddlepoint approximation since it is fully analytic. Thus, we compare
the performance of the two approximations, looking at their relative error for the approximation of the tail area probability. We keep the same Monte Carlo design  as in Section 6.1, 
namely 
$n=24$, and we consider different values of $z$, as in (4.15). 
Figure \ref{Fig: RelErr} displays the absolute value of the relative error, i.e.,
$|\text{approximation}/ \text{exact} -1 |$,  
when $W_n$ is Rook, Queen, and Queen torus. The plots illustrate that the relative error yielded by the saddlepoint approximation is smaller (down to ten times smaller in the case of Rook and Queen Torus) than the relative error entailed by the first-order asymptotic 
approximation (which is always about 100\%). 
The Edgeworth expansion entails a relative error which is typically higher than the one entailed by the saddlepoint approximation---the expansion can even become negative in some parts of the support, with relative error above 100\%. 

\begin{figure}[htbp]
\begin{center}
\begin{tabular}{ccc} 
Rook Left & Queen Left  & Queen Torus Right \\
\includegraphics[width=0.3\textwidth, height=0.25\textheight]{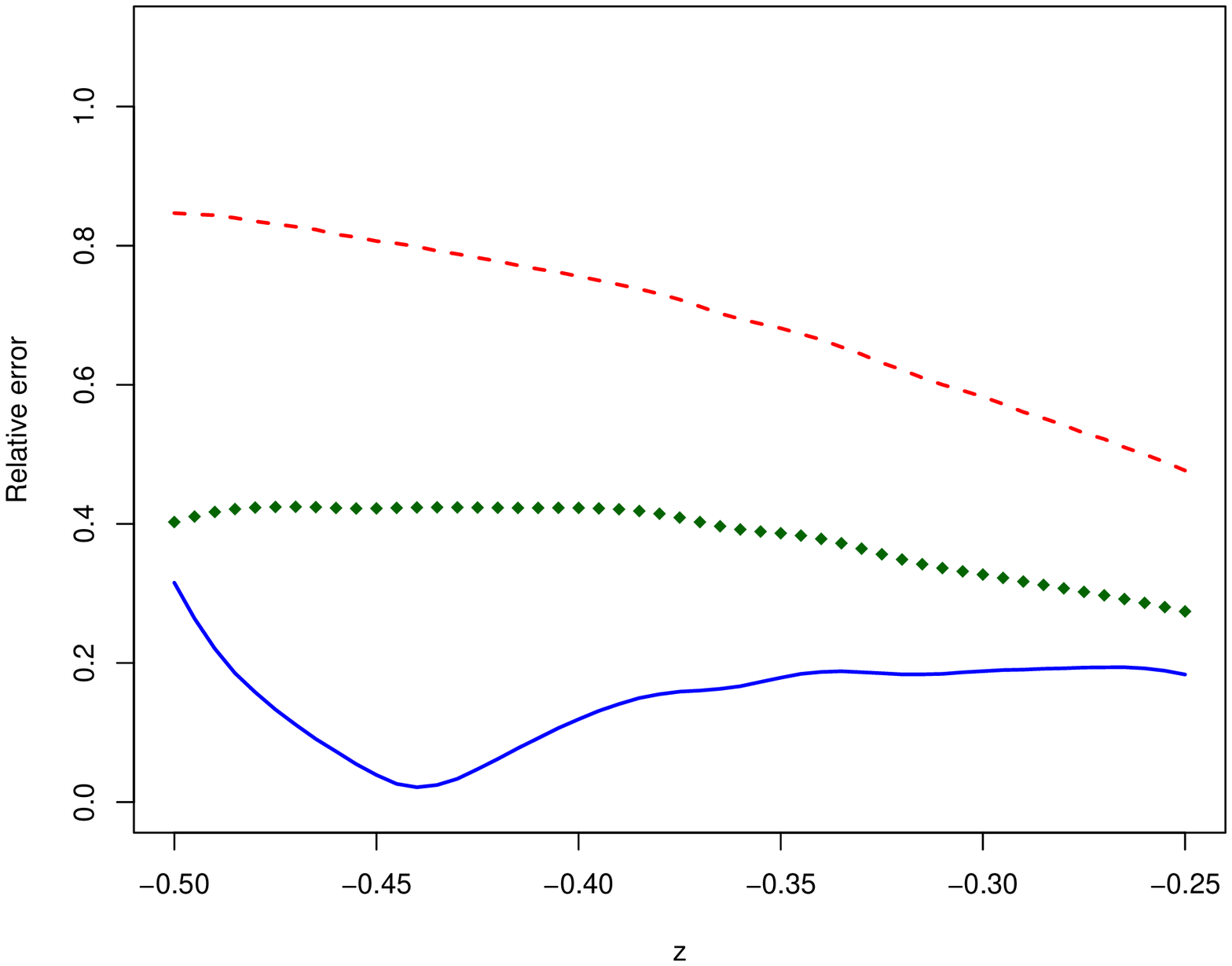}  &
\includegraphics[width=0.3\textwidth, height=0.25\textheight]{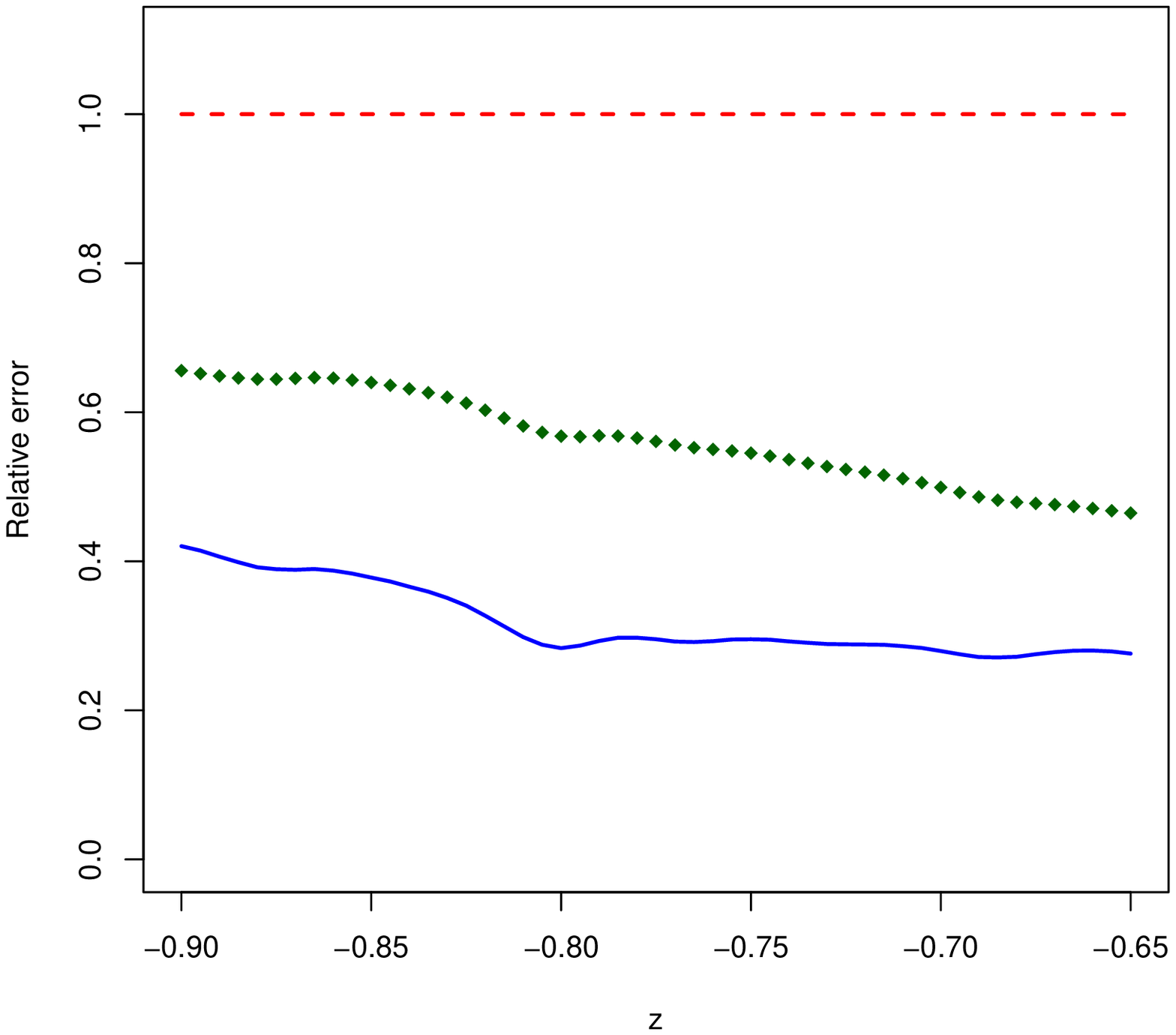}  &
\includegraphics[width=0.3\textwidth, height=0.25\textheight]{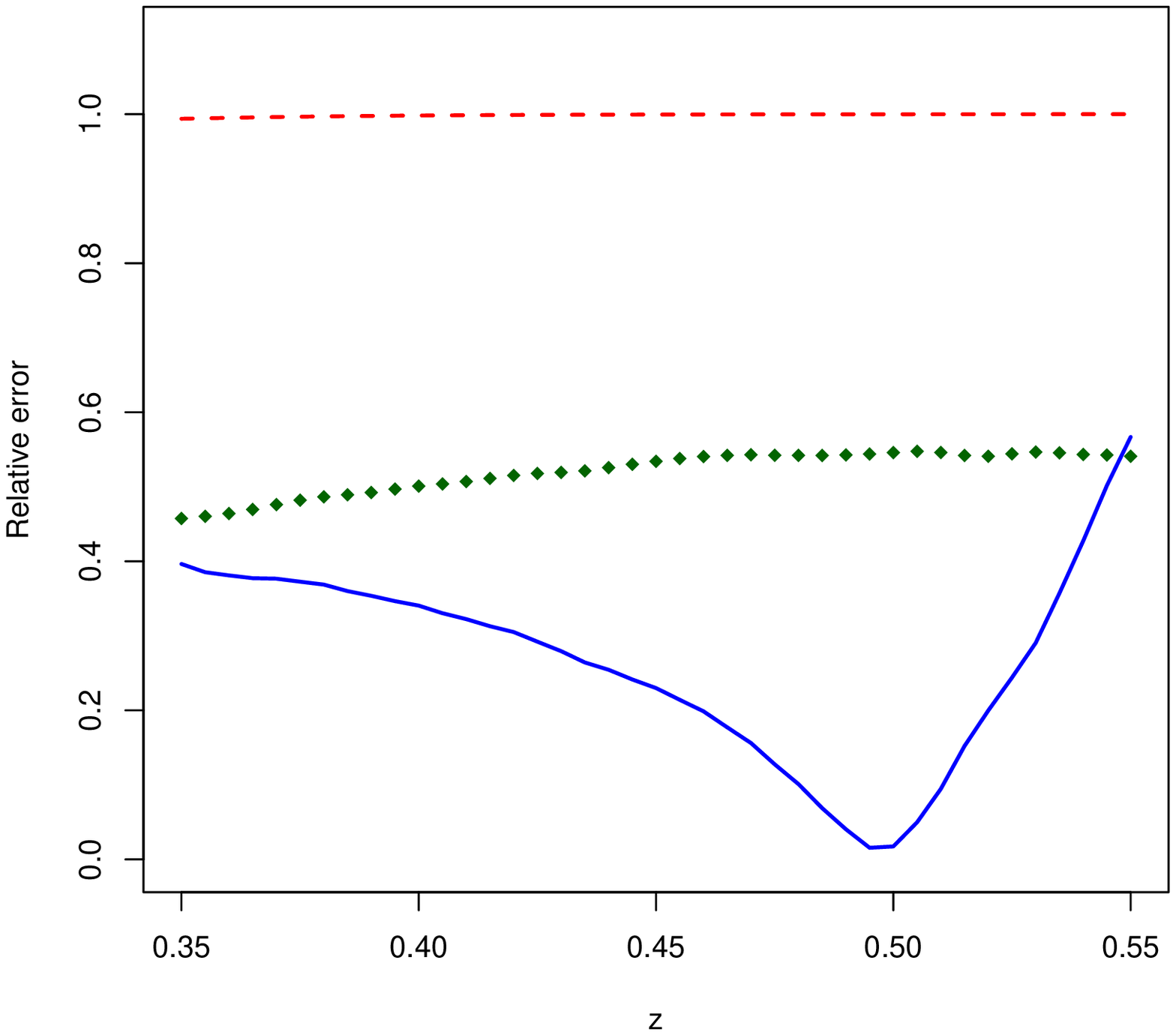}
\end{tabular}
\caption{SAR(1) model: Relative error (in absolute value) for the approximate right and left tail probability, as obtained using the Gaussian asymptotic theory (dotted line), the Edgeworth approximation (dotted line with diamonds) and saddlepoint approximation (continuous line), for the MLE $\hat\lambda$. In each plot, on the x-axes we display different values of $z$.  
The sample size is $n=24$, $\lambda_0=0.2$, and $W_n$ is Rook (left tail), Queen (left tail), and Queen Torus (right tail).}
    \label{Fig: RelErr}

\end{center}
\end{figure}

\subsection{Saddlepoint vs parametric bootstrap}
The parametric bootstrap represents a (computer-based) competitor, commonly applied in statistics and econometrics. 
To compare our saddlepoint approximation to the one obtained by bootstrap, we consider different numbers of bootstrap repetitions, labeled as $B$: we use $B=499$ and $B=999$. For space constraints, in Figure \ref{Fig: FB}, we display the results for $B=499$ (similar plots are available for $B=999$) showing the functional boxplots (as obtained iterating the procedure 100 times) of the bootstrap approximated density, for sample 
size $n =24$ and for $W_n$ is Queen. 
\begin{figure}[hbt!]
\begin{center}
\begin{tabular}{cc} 
\includegraphics[width=0.485\textwidth, height=0.35\textheight]{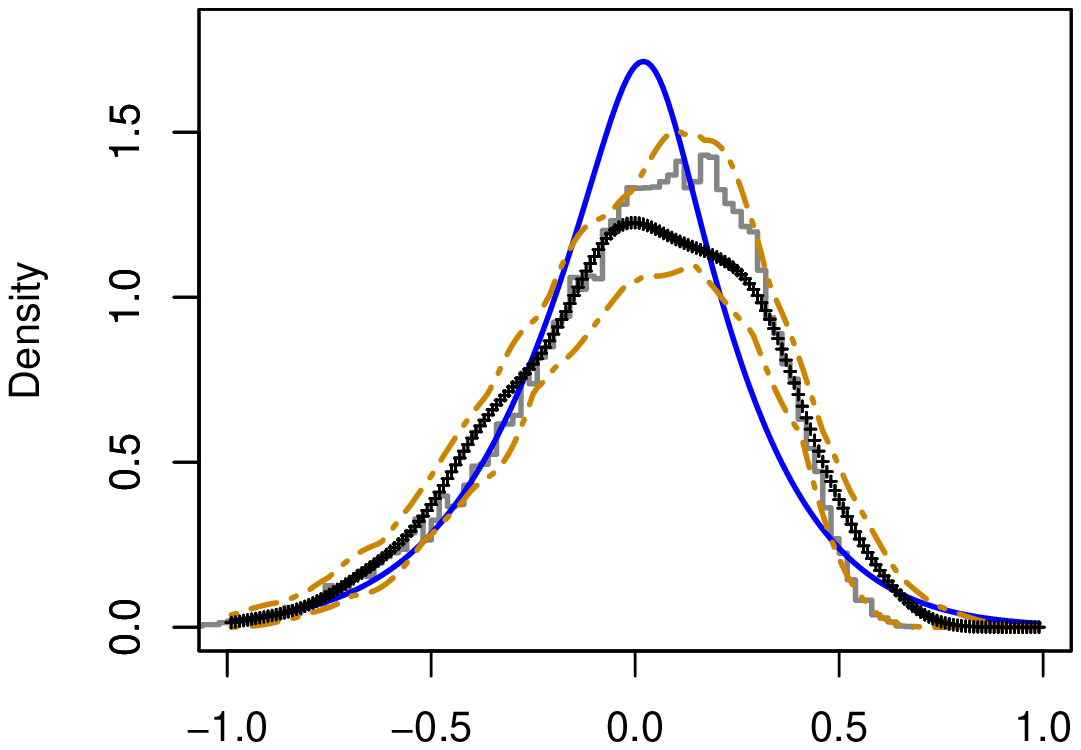} &
\includegraphics[width=0.485\textwidth, height=0.35\textheight]{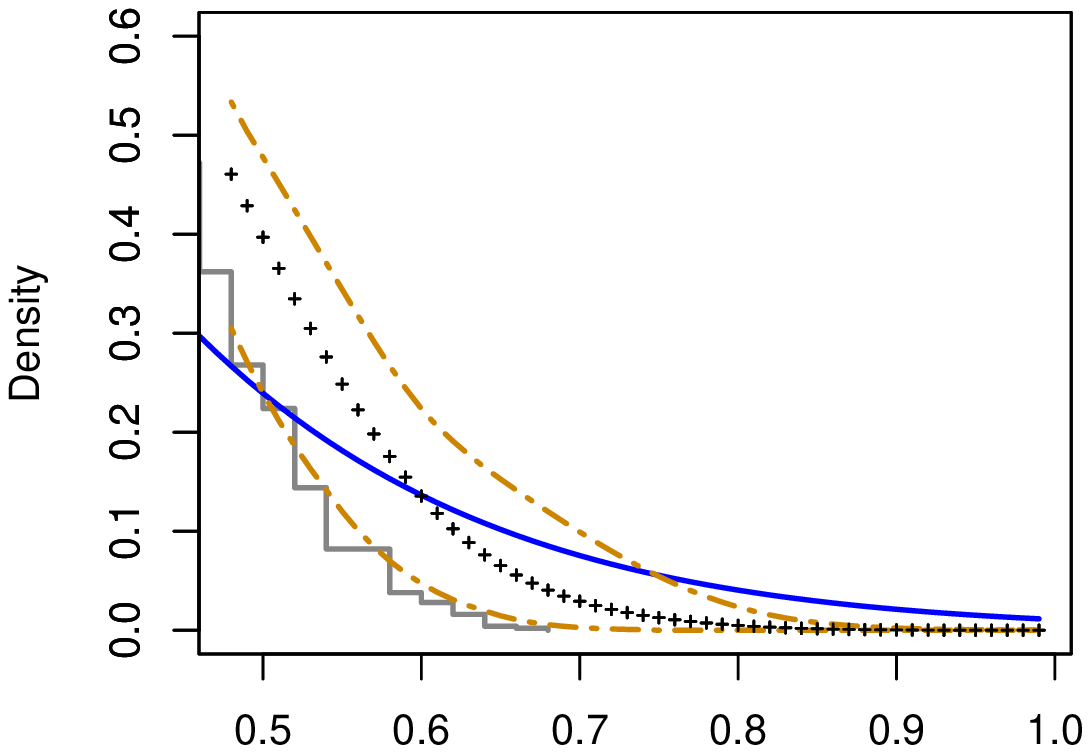}
\end{tabular}
\caption{SAR(1) model. Left panel: Density plots for saddlepoint (continuous line) vs the functional boxplot of the parametric bootstrap probability approximation to the exact density (as expressed by the histogram and obtained using 
MC with size 25000), for the MLE $\hat\lambda$ and $W_n$ is Queen. Sample size is $n=24$, while $\lambda_0=0.2$. Right panel: zoom on the right tail. In each plot, we display the functional central curve (dotted line with crosses), the $1_{st}$ and $3_{rd}$ functional quartile (two-dash lines). 
}
    \label{Fig: FB}

\end{center}
\end{figure}
To visualize the variability entailed by the bootstrap, we display the first and third quartile curves (two-dash lines) and the median functional curve (dotted line with crosses); for details about functional boxplots, we refer to \citet{Sun11} and to R routine \texttt{fbplot}. We notice that, while the bootstrap median functional curve (representing a typical bootstrap density approximation) is close to the actual density (as represented by the histogram), the range between the quartile curves illustrates that the bootstrap approximation has a variability.
Clearly, the variability depends on $B$: the larger is $B$, the smaller is the variability. However, larger values of $B$ entail bigger computational costs: when $B=499$, the bootstrap is almost as fast as the saddlepoint density approximation ({computation time  
about 7 minutes,  on a 2.3 GHz Intel Core i5 processor}), 
but for $B=999$, it is three times slower. {We refer to \citet{PPT15} for comments
on the computational burden of the parametric bootstrap and on the possibility to use
saddlepoint approximations to speed up the computation.} 
For $B=499$ and zooming on the tails, we notice that in the right tail and in the center of the density, the bootstrap yields an approximation (slightly) more accurate than the saddlepoint method, but the saddlepoint approximation is either inside or extremely close to the bootstrap quartile curves (see right panel of Figure \ref{Fig: FB}). In the left tail, the saddlepoint density approximation is closer to the true density than the bootstrap typical functional curve or $\lambda \leq -0.85$. Thus, overall, we cannot conclude that the bootstrap dominates uniformly (in terms of accuracy improvements over the whole domain) the saddlepoint approximation. Even if we are ready to accept a larger computational cost, the accuracy gains yielded by the bootstrap are yet not fully clear:
also for $B=999$, the bootstrap does not dominate uniformly the saddlepoint approximation. Finally, for $B = 49$, the bootstrap is about eight times faster than the saddlepoint approximation, but this gain in speed comes with a large cost in terms of accuracy. As an illustration, for $\hat\lambda-\lambda_0=-0.8 $ (left tail), the true density is $0.074$, the saddlepoint density approximation is $0.061$, while the bootstrap median value is $0.040$, 
with a wide spread between the first and the third quartile, being $0.009$ and 
$0.108$ respectively.

\subsection{Saddlepoint test for composite hypotheses: plug-in approach} \label{Sec: testcomp1}

Our saddlepoint density and/or tail approximations are helpful for testing simple hypotheses about $\theta_0$; see \S 6.1 of the paper. 
Another interesting case suggested by the Associate Editor and an anonymous referee that has a strong practical relevance is related to testing a
composite null hypothesis. It is a problem which is different from the one considered so far in the paper, because it raises the issue of dealing with nuisance parameters. 

To tackle this problem, several possibilities are available. 
For instance, we may 
fix the nuisance  parameters at the MLE estimates. Alternatively, we may consider to use
the (re-centered) profile estimators, as suggested, e.g., in \citet{HM18} and \citet{HM20}.
Combined with the saddlepoint density in (4.13), these techniques  yield a ready solution to the nuisance parameter problem. In our numerical experience, 
these solutions may preserve  reasonable accuracy in some cases.
 
To illustrate this aspect, we consider a numerical exercise about SAR(1) model, in which we set the nuisance
parameter equal to the MLE estimate. Specifically, we consider a SAR(1) with parameter $\theta=(\lambda,\sigma^2)'$ and our goal is to test
$\mathcal{H}_0: \lambda=\lambda_0 = 0$ versus $\mathcal{H}_1: \lambda > 0$, while leaving $\sigma^2$ unspecified. To study the performance of the saddlepoint density approximation for the sampling distribution of $\hat\lambda$, we perform a MC study, in which we set $n=24$, $T=2$, and $\sigma_0^2=1$. For each simulated (panel data) sample, we estimate the model parameter via MLE and get $\hat\theta = (\hat\lambda,\hat\sigma^2)'$. Then, we compute the  saddlepoint density in (4.13) using $\hat\sigma^2$ in lieu of $\sigma^2_0$: for each sample, we have a saddlepoint density approximation. We repeat this procedure 100 times, for $W_n$ Rook and Queen, so we obtain 100 saddlepoint density approximations for each $W_n$ type. In Figure \ref{Fig: Test_FB}, we display functional boxplots  of the resulting saddlepoint density approximation for the MLE $\hat\lambda$. To have a graphical idea of the saddlepoint approximation accuracy, on the same plot we superpose the histogram of $\hat\lambda$, which represents the sampling distribution of the estimated parameter. Even a quick inspection of the right and left plots suggests that the resulting saddlepoint density approximation preserves (typically) a good performance. Indeed, 
we find that the median functional curve (dotted line with crosses, which we consider as the typical saddlepoint density approximation) is close to  the histogram and it gives a good performance in the tails, for both $W_n$ Rook and Queen. The range between the first and third quartile curves (two-dash lines) illustrates the variability of the saddlepoint approximation. When $W_n$ is Queen, even though there is a departure from the median curve from the histogram over the x-axis intervals $(-0.75,0)$ and $(0.25,0.5)$, the histogram is inside the functional confidence interval expressed by the first and third curves. Thus, we conclude that computing the $p$-value for testing $\mathcal{H}_0$ using $\hat\sigma^2$ in the expression of the saddlepoint density approximation seems to preserve accuracy in this SAR(1) model, even for such a small sample.

\begin{figure}[htb!]
\begin{center}
\begin{tabular}{cc}
Rook  \& $n=24$ & Queen  \&  $n=24$  \\
\includegraphics[width=0.475\textwidth, height=0.35\textheight]{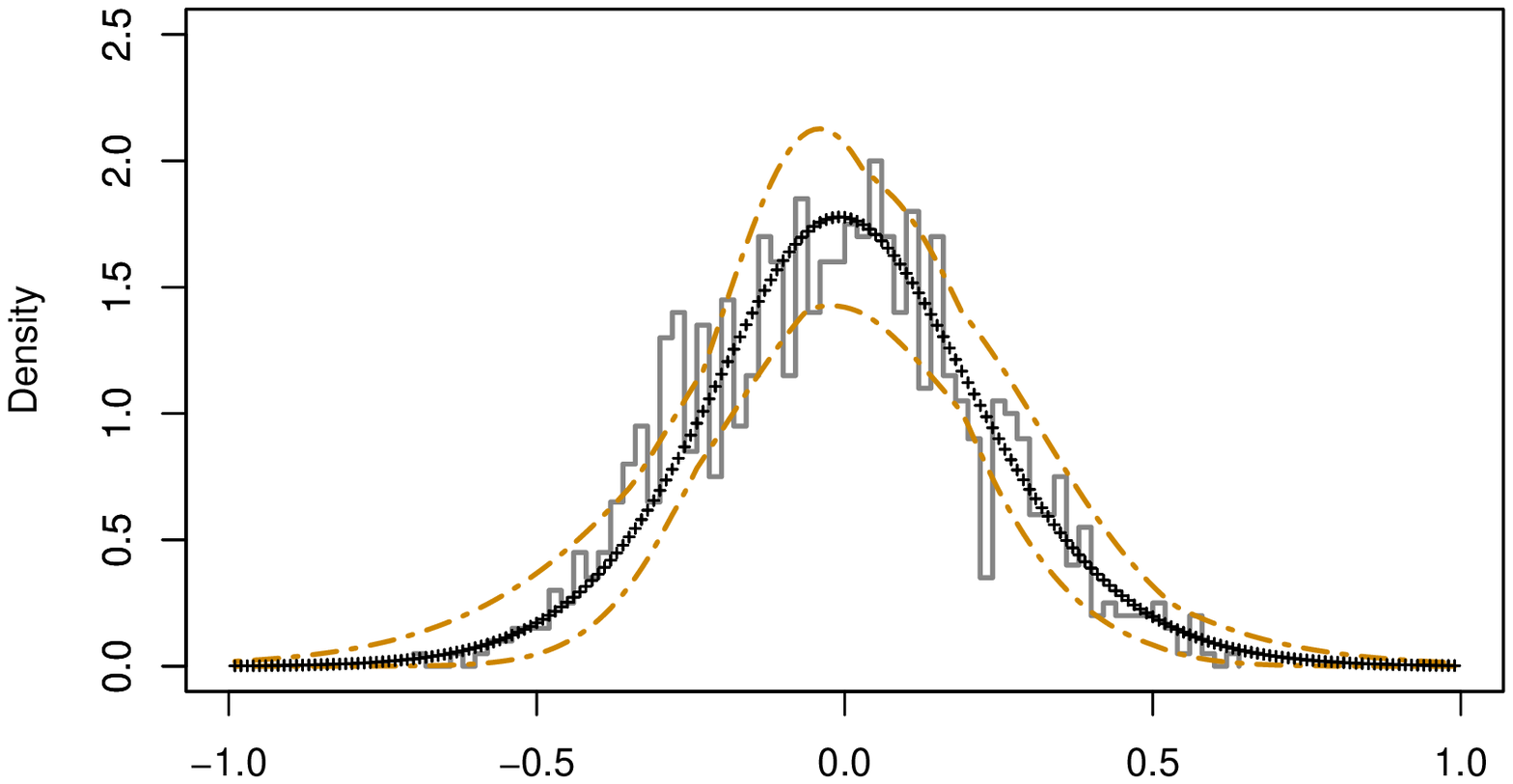} &
 \includegraphics[width=0.475\textwidth, height=0.35\textheight]{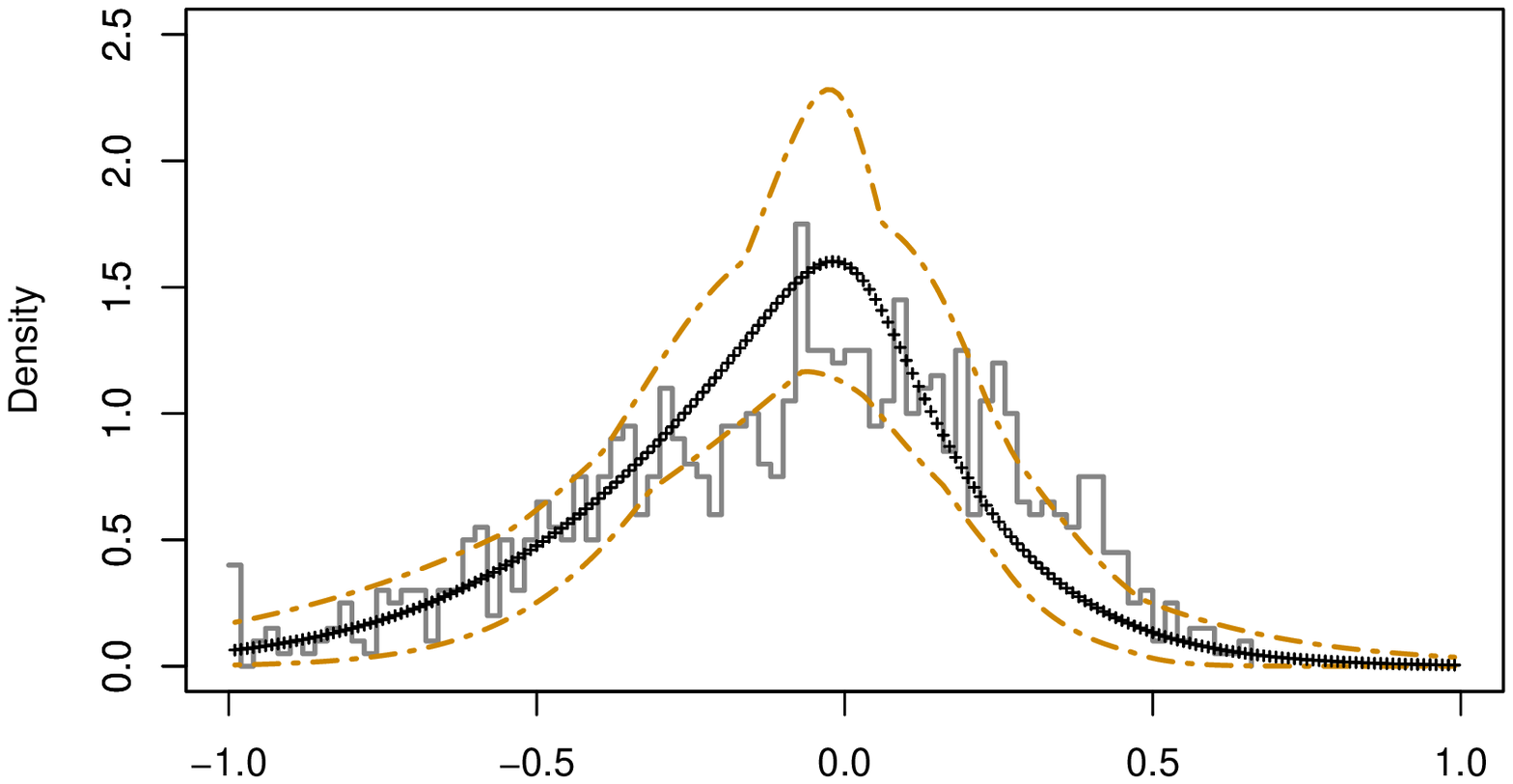} 

\end{tabular}
\caption{SAR(1) model: Functional boxplots of saddlepoint density approximation to the exact density (as expressed by the histogram), for the MLE $\hat\lambda$ and $W_n$ is Rook (left panel) and Queen (right panel). Sample size is $n=24$, while $\lambda_0=0$.}
    \label{Fig: Test_FB}

\end{center}
\end{figure}

\newpage
\small{
\bibliographystyle{myapalike1}
\bibliography{biblio_Panel}}